\documentclass{amsart}

\usepackage{graphicx}
\usepackage{enumerate}
\usepackage{amsmath,amssymb}
\usepackage{mathrsfs}
\usepackage{ascmac}

\newtheorem{maintheorem}{Theorem}

\newtheorem{maincorollary}[maintheorem]{Corollary}
\newtheorem*{question}{Question}
\newtheorem*{TLP}{Takens' Last Problem}

\newtheorem{theorem}{Theorem}[section]
\newtheorem{lemma}[theorem]{Lemma}
\newtheorem{proposition}[theorem]{Proposition}

\theoremstyle{definition}

\newtheorem{remark}[theorem]{Remark}

\numberwithin{equation}{section}
\numberwithin{figure}{section}

\usepackage{color}

\begin{document}

\title[Takens' last problem and wandering domains]{
Takens' last problem and existence of non-trivial wandering domains}

\author
[Shin Kiriki]
{Shin Kiriki} 
\address{
Department of Mathematics, Tokai University, 4-1-1 Kitakaname, 
Hiratuka, Kanagawa, 259-1292, Japan. \\
E-mail: {\tt kiriki@tokai-u.jp}
}

\author
[Teruhiko Soma] 
{Teruhiko Soma}
\address{Department of Mathematics and Information Sciences,
Tokyo Metropolitan University,
Minami-Ohsawa 1-1, Hachioji, Tokyo 192-0397, Japan. \\
E-mail: {\tt tsoma@tmu.ac.jp}
}


\date{\today}


\begin{abstract}
In this paper, we give an answer to a $C^{r}$ $(2\leq r <\infty)$ version of the open problem of 
Takens in \cite{T08} which is related to historic behavior of dynamical 
systems. 
To obtain the answer, we show the existence of non-trivial wandering 
domains near a homoclinic tangency, which is conjectured by Colli-Vargas \cite[\S 2]{CV01}.
Concretely speaking, it is proved that any Newhouse open set in the space of $C^{r}$-diffeomorphisms on a closed surface  
is contained in the closure of the set of diffeomorphisms which have non-trivial 
wandering domains whose forward orbits have historic behavior. 
Moreover, this result implies an answer in the $C^{r}$ category to one of the open problems 
of van Strien \cite{vS10} which is 
concerned with wandering domains for H\'enon family.
\end{abstract}

\maketitle

\section{Introduction}\label{s.Introduction}
\subsection{Historic behavior and wandering domains}
Consider a dynamical system with a compact state space $X$, given by a continuous map 
$\varphi:X\to X$.
We say that the forward orbit $\{x, \varphi(x),\varphi^2(x),\dots\}$  of $x\in X$ has \emph{historic behavior} if 
the partial average  
$$\frac{1}{m+1}\sum_{i=0}^m\delta_{\varphi^i(x)}$$
 dose not converge as $m\to\infty$
in the weak topology, 
where $\delta_{\varphi^i(x)}$ is the Dirac measure on $X$ supported at $\varphi^i(x)$. 
The terminology of historic behavior was given by Ruelle in \cite{R01}. 
The following is the last open problem presented by Takens (see \cite{T08}).

\begin{TLP}
Whether are there persistent classes of smooth 
dynamical systems such that the set of initial states which give rise to orbits with historic
behavior has positive Lebesgue measure?
\end{TLP}

The first example of historic behavior was given in \cite{HK90}, 
where it is shown that the logistic family contains elements for which almost all orbits have historic behavior. 
This was extended to generic full families of unimodal maps, see \cite{dMvS93}. 
While Takens showed in  \cite{T94} that 
Bowen's $2$-dimensional flow with an attracting heteroclinic loop 
has a set of positive Lebesgue measure consisting of initial points of orbits with historic behavior, 
but the property is not preserved 
under arbitrarily small perturbations of the dynamics. Also, 
by using Dowker's result \cite{D53}, Takens  showed that the doubling map on the circle 
persistently has orbits with historic behavior, for which the collection of initial points is a residual subset  on the circle, see \cite{T08}.

In this paper, 
we obtain an answer to Takens' Last Problem for non-hyperbolic diffeomorphisms having homoclinic tangencies 
by a different way from the previous works.
To solve the problem
we use a non-empty connected open set, called a wandering domain, 
whose images do not intersect each other but are wandering around non-trivial hyperbolic sets. 
Wandering domains have been studied from the beginning of 20th century.
In fact, Bohl \cite{Bo16} in 1916 and Denjoy \cite{D32} in 1932 constructed examples of
$C^{1}$ diffeomorphisms on a circle which have wandering domains 
whose $\omega$-limit set is a Cantor set. 
However, it can not be extended to any $C^{2}$ as well as $C^{1}$ diffeomorphism whose derivative is a function of bounded variation, see in \cite{dMvS93}. 
Subsequently, similar phenomena for high dimensional diffeomorphisms 
were studied by several authors, for example  \cite{Kn81, Ha89, Mc93, BGLT94, NS96, KM10}. 
Also, for unimodal as well as multimodal maps on an interval or a circle,  
the main difficulty in  their classification in real analytic category
was to show the absence of wandering domains, which were 
developed by many dynamicists \cite{dMvS89, Ly89,BlLy89, dMvS93, vSV04}, see the survey of van\ Strien \cite{vS10}.

On the other hand,  a wide variety of investigations derived from Smale's works in 1960s 
yielded abundant developments, and 
provided a focal point for us to explore beyond hyperbolic phenomena. 
Thus, we here focus entirely on one of non-hyperbolic phenomena called homoclinic tangencies, which 
were pioneered by Newhouse, Palis, Takens and others. 
It is somewhat surprising that homoclinic tangencies and wandering domains were not studied together until Colli-Vargas' model  in \cite{CV01}. 
We furthermore discuss these two themes in more general situation to solve Takens' Last Problem.

\subsection{Main results}
To state our results  we have to introduce some definitions.  
 Let $M$ be a closed  two-dimensional $C^{\infty}$ manifold  
 and $\mathrm{Diff}^{r}(M)$, $r\geq 2$, the set of $C^{r}$ diffeomorphisms on $M$ endowed with $C^{r}$ topology.  
 We say that $f\in \mathrm{Diff}^{r}(M)$ has a  
\emph{homoclinic tangency} of a saddle periodic point $p$ if the stable manifold $W^{s}(p, f)$ and unstable manifold $W^{u}(p, f)$ 
have a non-empty and non-transversal intersection. 
Newhouse showed in \cite{N79} that, for any $f\in \mathrm{Diff}^r(M)$ with a homoclinic 
tangency of a dissipative saddle fixed point $p$, there is an open set $\mathcal{N} \subset \mathrm{Diff}^{r}(M)$ whose closure $\mathrm{Cl}(\mathcal{N})$ contains $f$ and 
such that any element of $\mathcal{N}$  is arbitrarily $C^r$-approximated by a diffeomorphism 
$g$ with 
a homoclinic tangency associated with a dissipative saddle fixed point $p_g$ which is 
the continuation of $p$, and moreover $g$ has a $C^{r}$-persistent tangency associated with some  
basic sets $\Lambda_g$ containing $p_g$ 
(i.e. there is a $C^{r}$ neighborhood of $g$ any element of which has a homoclinic tangency 
for the continuation of $\Lambda_g$).
Such an open set $\mathcal{N}$ is called a \emph{Newhouse open set} or a \emph{Newhouse domain}. 
Various non-hyperbolic phenomena were observed in $\mathcal{N}$ but still far from being completely understood. 
For example, Newhouse also showed in \cite{N79} that 
generic elements of $\mathcal{N}$ have infinitely many sinks or 
sources. Kaloshin proved  in \cite{Kal00} that  
the number of periodic points for diffeomorphisms in a residual subset of $\mathcal{N}$ 
grows super-exponentially. See \cite[\S 1]{GK07} and \cite[\S 0]{KS08} 
for detail descriptions of other results which are not mentioned here.

Our definition of a wandering domain is the same as one defined for one-dimensional dynamics  in \cite{dMvS93, vS10}. In fact, 
we say that, for $f\in \mathrm{Diff}^{r}(M)$, 
a non-empty connected open set $D\subset M$ is a \emph{wandering domain} if
\begin{itemize}
\item $f^{i}(D)\cap f^{j}(D)=\emptyset$ for any integer $i, j\geq 0$ with $i\neq j$; 
\item the union $\omega(D, f)=\bigcup_{x\in D}\omega(x,f)$ of $\omega$-limit sets is not equal to a single periodic orbit.
\end{itemize}
A wandering domain $D$ is called \emph{contracting} if the diameter of $f^{n}(D)$  
converges to zero as $n\to \infty$. 
Note that Denjoy's example is a contracting wandering domain, see \cite{D32, vS10}. 

We now state the main result of this paper 
where a conjecture of Colli and Vargas is proved affirmatively in $C^{r}$ topology and 
a solution to Takens' Last Problem will be obtained by using some non-trivial wandering domains.

\begin{maintheorem}\label{main-thm}
Let $M$ be a closed surface and $\mathcal{N}$ any Newhouse open set in 
$\mathrm{Diff}^{r}(M)$ with $2\leq r< \infty$.
Then there exists a dense subset of $\mathcal{N}$ each element $f$ of which 
has a contracting wandering domain $D$ such that
\begin{enumerate}[{\rm (1)}]
\item $\omega(D, f)$ contains a hyperbolic set which is not just a periodic orbit;
\item the forward orbit of every $x\in D$ under $f$ has historic behavior.
\end{enumerate}
\end{maintheorem}

\noindent
We note the following: 
\begin{itemize}
\item
(The absence of regularity)
Any diffeomorphisms given in Theorem \ref{main-thm} as well as Colli-Vargas' examples in \cite{CV01}  
are not necessarily guaranteed to be of class $C^\infty$, see 
Remark \ref{rem_disjoint}\,(2).
\item
(Variety of average measures)
In Colli-Vargas' examples, the sequence of the average measures $\mu_x(m)$ can have    
various accumulation points, see \cite[\S 9]{CV01}.
However, in our construction, one can not expect such a variety for some technical reasons, 
see Remark \ref{r_accumulation_pt}.

\item (Positivity of Lebesgue measure and persistent property) Any wandering domain  is an open set, and hence in particular $D$ has positive Lebesgue measure, 
which is the condition required in Takens' Last Problem. 
While the property having a wandering domain obtained in Theorem \ref{main-thm} is not persistent. 
This is the reason why only dense subsets are obtained, 
which is exactly the same sort of restriction as in the paper by Colli \cite{Co98} 
concerning the density of H\'enon-like attractors. 

\item (Higher dimensions) We think that a similar result holds for codimension one dissipative 
homoclinic tangencies in dimensions higher than two, which is studied in \cite{PV94}. 
\end{itemize}

 \medskip
 
Next we consider 
the \emph{H\'enon family} $f_{a,b}:\mathbb{R}^2\to \mathbb{R}^2$ defined  as
\begin{equation}\label{original}
f_{a,b}(x,y)=(1-ax^2+y, bx),
\end{equation}
where $a, b$ are real parameters. 
This family will play a significant role in the renormalization 
near the homoclinic tangency.  
As $b=0$, the dynamics of $f_{a,0}$ is perfectly controlled 
by the family of quadratic maps $\varphi_{a}(x)=1-ax^2$. 
It is known that there is a parameter value of $a$ 
such that 
there is a $C^{1}$ unimodal map which is 
semi-conjugated to $\varphi_{a}$ and 
has a wandering interval \cite{Ha81}. 
Moreover, under the sufficient differentiability, 
some large class of multimodal maps including the quadratic maps
cannot have non-trivial wandering intervals, see
\cite{vSV04}. However, as $b\neq 0$, 
it has not been known whether or not $f_{a,b}$ 
has  wandering domains, 
which is one of open problems in \cite{vS10, LM11}. 
We can answer such a problem 
in the $C^{r}$ category with $2\leq r<\infty$ as in the following corollary 
of Theorem \ref{main-thm} together with the fact that 
there exists $(a,b)$  arbitrarily close to $(2, 0)$ such that 
$f_{a,b}$ has a quadratic homoclinic tangency for some 
saddle fixed point which unfolds generically with respect to $a$, 
see for example in \cite{KLS10, KS13}.

\begin{maincorollary} 
There is an open set $\mathcal{O}$ of the parameter space of H\'enon family 
with $\mathrm{Cl}(\mathcal{O})\ni(2,0)$ such that, for every $(a, b)\in \mathcal{O}$, 
$f_{a,b}$ is $C^{r}$-approximated by diffeomorphisms which have  contracting non-trivial wandering domains.
\end{maincorollary}

In the end of this section we recall that non-trivial wandering phenomena are observable in circle homeomorphisms in the $C^{\infty}$ category by Hall \cite{Hal83} 
but not in the $C^{\omega}$ category by Yoccoz \cite{Yoc84}, which is the answer to one of problems by Poincar\'e \cite{Po1885}. 
Note that every discussion of the present paper unfolds in the $C^{r}$ category with  $2\leq r<\infty$, but  
some tools would not be applied directly 
to discussion in the $C^{\infty}$ as well as $C^{\omega}$ category. See Remark \ref{rem_disjoint}-(2).
Thus, the open problem for wandering domains of the H\'enon family 
by van Strien et al.\ is unsolved yet in $C^{\infty}$ and $C^{\omega}$ categories.
So it is worth recalling the following:
\begin{question}
Does there exist a parameter value $(a, b)$ for the H\'enon family (\ref{original}) such that 
$f_{a,b}$ has a non-trivial wandering domain?
\end{question}

\noindent 
Note that 
Astong et al.\! \cite{ABDPR} study  the existence of wandering Fatou components for polynomial skew-product maps
and 
present an example which admits 
a wandering Fatou component intersecting $\mathbb{R}^{2}$. However, it does not contain any H\'enon family.

\section{Outline of the proof}
In this section, we will sketch the proof of Theorem \ref{main-thm},  
where several technical terminologies, 
e.g.\ \emph{linked pair}, \emph{linking property}, \emph{critical chains}, etc, 
appear without definitions, all of which will be given in the  following sections.

\subsection{Standard  settings for general situations}
For the beginning, it could not  be better than that assumptions are minimized. 
So we start our discussions not for specific models as in \cite{CV01}, 
but rather for any two-dimensional diffeomorphism $f$ which has a homoclinic tangency for a dissipative  saddle fixed point, say $p$. 
See Figure \ref{fig_2_2}. 
For such a situation, 
it naturally reminds us of the renormalization scheme near  homoclinic tangencies by  Palis-Takens \cite{PT93}, see Theorem \ref{lem.renormalization}. 
In fact, we will take much advantage of the scheme as follows.

We may here assume that $f$ is a linear map $f(x,y)=(\lambda x, \sigma y)$ 
in a small 
neighborhood $U(p)$ of $p$ with $0<\lambda<1<\sigma$, $\lambda\sigma<1$ by performing an arbitrarily small perturbation for $f$ and replacing $f$ by 
$f^2$ if necessary.
In our scheme, we have two main  basic sets, 
where the \emph{basic set} means 
a compact hyperbolic and locally maximal invariant set 
which is transitive and contains a dense subset of periodic orbits.
One is a horseshoe $\Lambda$  which is 
associated with a transverse homoclinic intersection of $p$ but \emph{not} affine in general. 
The other is a basic set, denoted by $\Gamma_{m}$,  which is 
created by an H\'enon-like return map $\varphi_n$ of (\ref{eqn_phi_n}) in  
the renormalization near the homoclinic tangency,
where $m$ is the period of some periodic point for $\varphi_n$. 
Those ingredients and their cyclic interconnection by way of persistent heteroclinic tangencies
are precisely described in Sections \ref{sec.prep}. 
\begin{figure}[hbt]
\centering
\scalebox{0.8}{\includegraphics[clip]{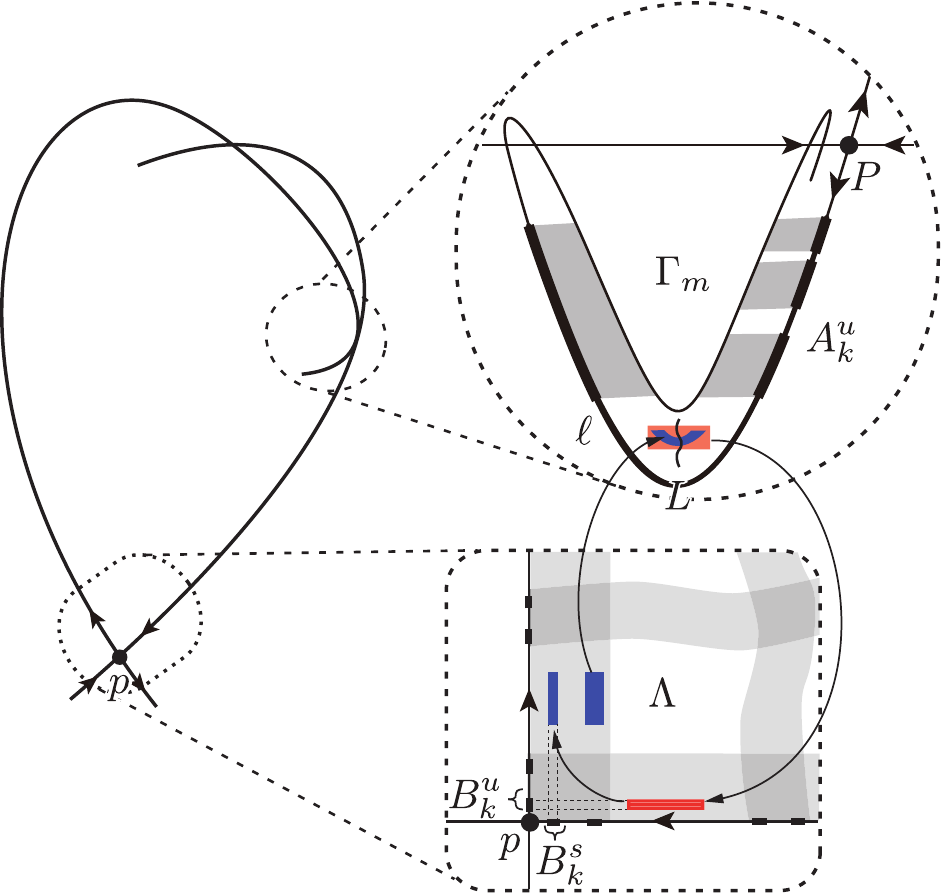}}
\caption{Transition of the wandering domains}
\label{fig_2_2}
\end{figure}

\subsection{Main Cantor sets and bridges} 
 From the basic sets $\Lambda$ and $\Gamma_{m}$
one can obtain several dynamically defined Cantor sets, 
among which the following three are especially important in this paper:
$$
K^{s}_{\Lambda}:=\pi_{\mathcal{F}_{\mathrm{loc}}^{u}(\Lambda)}(\Lambda), \quad
K^{u}_{\Lambda}:=\pi_{\mathcal{F}_{\mathrm{loc}}^{s}(\Lambda)}(\Lambda),\quad
K^{u}_{m}:=\pi_{m}(\Gamma_{m}), 
$$ 
where 
$\pi_{\mathcal{F}_{\mathrm{loc}}^{u}(\Lambda)}$, 
respectively $\pi_{\mathcal{F}_{\mathrm{loc}}^{s}(\Lambda)}$, 
is the  
projection on $W^{s}_{\mathrm{loc}}(p)$, 
respectively $W^{u}_{\mathrm{loc}}(p)$, 
along the leaves of an unstable foliation $\mathcal{F}_{\mathrm{loc}}^{u}(\Lambda)$, 
respectively stable foliation $\mathcal{F}_{\mathrm{loc}}^{s}(\Lambda)$, 
and 
$\pi_{m}$ is the projection on an arc $\ell\subset W^{u}_{\mathrm{loc}}(P)$ along the leaves of a 
stable foliation  $\mathcal{F}^s_{\mathrm{loc}}(\Gamma_{m})$ compatible with $W^s_{\mathrm{loc}}(\Gamma_{m})$. 
Here $P$ stands for the saddle fixed point for $\varphi_{n}$ as illustrated in Figure \ref{fig_2_2}, which is not contained in  $\Gamma_{m}$. 
See
 (\ref{def.Ks}) and (\ref{def.Ku}) respectively.
Moreover, one has the sequences of 
$s$, $u$-bridges related to the three Cantor sets,  respectively denoted by 
$$
\{B^{s}_{k}\}, \quad
\{B^{u}_{k}\},\quad
\{A^{u}_{k}\}, 
$$ 
where the former two are defined in Subsection \ref{subsec.horse}, see Figure \ref{fig_4_1}, 
and 
the latter is in Subsections \ref{subset.one-dim}--\ref{subsec.Hetero}, see Figure \ref{fig_4_3}.
In Section \ref{sec.bounded-distortions}, we give  descriptions of bounded distortion properties of these bridges.

\subsection{Generalized uniformly linking properties for 
$\{B^{s}_{k}(\Delta)\}$ and $\{A^{u}_{k}(\Delta)\}$} 
In the case of the Colli-Vargas model in \cite[\S 3]{CV01}, 
non-trivial wandering domains 
were detected in intersections between stable gaps 
and unstable gaps  of sufficiently thick affine Cantor sets. 
Obviously,  
we cannot directly unfold the same story into our case, 
because there is no promise in general such that 
the product of thickness of $K^{s}_{\Lambda}$ and $K^{u}_{\Lambda}$ is larger than one.

However, a bypass of this problem is already given in the 
renormalization scheme, see Theorem \ref{lem.renormalization}.
In fact, since 
the thickness of the third Cantor set $K^{u}_{m}$ has an arbitrarily large value by taking $m$ large enough, 
one can obtain a $C^{2}$-persistent heteroclinic tangency 
associated with $W^{u}(\Lambda_{g})$ and $W^{s}(\Gamma_{m,g})$, 
where $\Lambda_{g}$ and $\Gamma_{m,g}$ are the continuations of $\Lambda$ and $\Gamma_{m}$, respectively, 
see (S-\ref{setting5}) in Section \ref{sec.prep}. 
For simplicity, we still denote such continuations by $\Lambda$ and $\Gamma_{m}$, respectively, in this outline.
Therefore, in Section \ref{sec.Linking}, 
we can discuss linking properties between the 
images of continuations for bridges 
$\{B^{s}_{k}\}$ and $\{A^{u}_{k}\}$ on the arc $L$ of 
tangencies between 
$f^{N_2}(\mathcal{F}_{\mathrm{loc}}^u(\Lambda))$ 
and
$f^{-N_0}(\mathcal{F}_{\mathrm{loc}}^s(\Gamma_m))$  
for some integers $N_0, N_2>0$.  See Figure \ref{fig_5_3}.  
Actually, these linking situations will be found by the projected images (\ref{def.Cantor sets}) on $L$. 

To obtain a linked pair of the continuations for bridges, 
we have to add the first perturbation in $f^{-1}(\mathcal{B}_{\delta_{0}})$ where 
$\mathcal{B}_{\delta_{0}}$ is a small $\delta_{0}$-disk 
which meets the inverse image $\tilde{L}$ of $L$, see Figure \ref{fig_5_5}.
Precisely, this perturbation is the horizontal $\delta$-shift with $|\delta| \ll \delta_{0}$, and hence 
the perturbed map is given as 
$$f_{\delta}(\boldsymbol{x})=f(\boldsymbol{x})+(\delta,0)$$
 for any 
$\boldsymbol{x}\in f^{-1}(\mathcal{B}_{\delta_0/2})$.
Using this perturbation, 
we present Lemma  \ref{lem_LL1} which is a generalization of \emph{Linking Lemma} in \cite[p.\,1669]{CV01}. 
Moreover, 
in Lemma \ref{lem_LG1}, we show the existence of uniformly linking subsequence of $\{A^{u}_{k}(\Delta)\}$ and 
$\{B^{s}_{k}(\Delta)\}$ 
where $\Delta=\sum_{k=1}^{\infty} \delta_{k}$ for the sequence of local $\delta_k$-shifts given in Lemma \ref{lem_LG2}. 
This is a generalization of  \emph{Linear Growth Lemma} in \cite[p.\,1670]{CV01}.
The proofs of the results are supported by 
Lemma \ref{lem.bdp1}   
and  
Lemma \ref{lem.bdp2} 
in which bounded distortion properties for $s$-bridges  
$\{B^{s}_{k}\}$ and $u$-bridges $\{A^{u}_{k}\}$ are presented.

\subsection{Critical chains in $\{B^{s}_{k}(\Delta)\}$ and $\{B^{u}_{k}(\Delta)\}$}
Note that the uniformly linking subsequence of $\{B^{s}_{k}(\Delta)\}$ and $\{A^{u}_{k}(\Delta)\}$  
are constructed on the arc $L$ including 
\emph{heteroclinic} persistent tangencies of $W^{u}(\Lambda)$ and $W^{s}(\Gamma_{m})$.
At this stage, 
since we use only a one-way from $\Lambda$ to $\Gamma_{m}$,
even if one takes any domain constructed from the linking subsequence by some perturbations, 
there is no certification that the orbit of domain comes back to and 
wanders around $\Lambda$ non-trivially. 

However, we can construct a  return route by using the fact that $\Gamma_{m}$ and $\Lambda$ are homoclinically related to each other, 
see the condition (S-\ref{setting5}) of Section \ref{sec.prep}.
It follows that the stable foliation 
$\mathcal{F}^{s}(\Lambda)$ and 
some gap of $K_{m,L}^u$ have a
transverse intersection. 
Hence, in Section \ref{sec.CC}, 
one obtains  
some gap  $\widehat{G}^{u}_{L,k+1}$ of $A^{u}_{L, k+1}$  which contains 
some $u$-bridge $B^{u}_{L, k+1}:=\pi_{\hat{A}^{u}_{k}}(B^{u}_{k+1})$ as illustrated in Figure \ref{fig_7_1}, 
 where $\pi_{\hat{A}^{u}_{k}}:B(0)\to L$ is the projection defined as  (\ref{eqn_pi_A^u}) and 
 $B^{u}_{k+1}$ is the $u$-bridge whose itinerary satisfies  (\ref{eqn_CC1}).

Moreover, 
by taking the itinerary of the $s$-bridge $B_{k+1}^{s*}$ as in  (\ref{eqn_CC1}), 
one can obtain linking situations which are desired in Lemma \ref{lem_CC} (\emph{Critical Chain Lemma}):  
there is an interval $J^{*}_{k+1}$ such that 
$t\in J_{k+1}^*$ if and only if 
$$(B_{L,k+1}^{s*}+t)\cap B_{L,k+1}^{u}\neq \emptyset,$$ 
where $B_{L,k+1}^{s*}$ is the image of  the $s$-bridge $B_{k+1}^{s*}$ of $K^{u}_{\Lambda}$ by 
the projection $\pi^{s}:S\to L$ of (\ref{eqn_pi^s}). 
Note that the itineraries given in (\ref{eqn_CC1}) will be crucial to control the orbit of any point in the wandering domain obtained in the later sections.

\subsection{Multidirectional perturbations and critical chains of rectangles}
We may consider the inverse image $\widetilde L$ of $L$ which is contained in the neighborhood of $\Lambda$, see Figure \ref{fig_7_3}. 
Lemma \ref{l_slope} implies that, for all sufficiently large $k$, 
there exists  almost horizontal line $L_k$ such that 
 $L_k$ meets $L$ transversely at a single point $\boldsymbol{x}_k$.

Moreover we take a point 
$\widetilde{\boldsymbol x}_k=f^{N_1}\circ f^{\widehat i_k}\circ f^{N_0}(\boldsymbol{x}_k)$ in $\widetilde L_k$,
and 
define a sequence $\{\widehat{\boldsymbol{x}}_k\}$ with 
$\widehat{\boldsymbol{x}}_k=f^{z_kk^2+\langle k\rangle}(\widetilde{\boldsymbol x}_k)$, 
where $z_k$ is either $z_0$ or $z_0+1$ for a fixed positive integer $z_0$ satisfying the 
conditions (\ref{eqn_z0+1}), (\ref{eqn_eta}) and 
$\langle k\rangle$ is the integer of (\ref{eqn_lakra}) with $\lim_{k\to \infty} \langle k\rangle/k^{2}=1$. 
For $\{\boldsymbol{x}_k\}$ and $\{\widehat{\boldsymbol{x}}_k\}$, 
one has 
the sequence 
$$
\underline{\boldsymbol{t}}=(\boldsymbol{t}_{2}, \boldsymbol{t}_{3}, \ldots, \boldsymbol{t}_{k},\ldots),
$$
where each  $\boldsymbol{t}_{k} $ is the vector given by 
 $f^{N_2}(\widehat{\boldsymbol x}_k)+\boldsymbol{t}_{k+1}=\boldsymbol{x}_{k+1}$, see 
 Figure \ref{fig_7_6},
 which is the second perturbation  corresponding to  the \emph{perturbation vector} in \cite[p.1673]{CV01}. 
Note that 
entry vectors of the perturbation not necessarily have the same direction. 

For $\underline{\boldsymbol{t}}$, 
we now define the diffeomorphism $f_{\underline{\boldsymbol{t}}}$ by
$$
f_{\underline{\boldsymbol{t}}}:=f\circ \psi_{\underline{\boldsymbol{t}}} 
$$
so that $f_{\underline{\boldsymbol{t}}}(\widehat{\boldsymbol{x}}_k)=\boldsymbol{x}_{k+1}$, 
where $\psi_{\underline{\boldsymbol{t}}}$ is the $C^r$-map defined 
as (\ref{eqn_psi_t}).  
Note that Lemma \ref{lem.perturb} claims that, 
if $T$ is sufficiently large, then $\psi_{\underline{\boldsymbol{t}}}$ 
is arbitrarily $C^r$-close to 
the identify and hence $f_{\underline{\boldsymbol{t}}}$ is a $C^r$-
diffeomorphism arbitrarily $C^{r}$-close to $f$.

\subsection{Non-trivial wandering domains} 
Around each $\boldsymbol{x}_k=(x_k,y_k)$, we define a rectangle as 
$$R_k=\bigl[x_k-b_k^{\frac12},x_k+b_k^{\frac12}\bigr]\times [y_k-b_k,y_k+b_k],$$
where $b_k$ is the positive number given in (\ref{def.b_k}), 
and show in 
Lemma \ref{lem_RL} (\emph{Rectangle Lemma}) that there is an integer $k_0>0$ such that, for any 
$k\geq k_0$, $$f_{\underline{\boldsymbol{t}}}^{m_k}(R_k)\subset \mathrm{Int} R_{k+1},$$ where $m_k$ is 
the positive integer given in Remark \ref{r_fzDz}.  See Figure \ref{fig_8_1}. 
This implies that the interior $D$ of $R_{k_{0}}$ is a wandering domain. 
It follows  immediately from our construction of $D$ that the wandering domain is contracting.

We will consider a sufficiently large positive integer $z_0$ independent of $k$ and satisfying the conditions (\ref{eqn_z0+1}) and (\ref{eqn_eta}), and consider a sequence $\boldsymbol{z}=\{z_k\}_{k=1}^\infty$ of integers such that each entry $z_k$ is either $z_0$ or $z_0+1$. This implies that 
the linear map $f_{\underline{\boldsymbol{t}}}^{z_kk^2}(x,y)=(\lambda^{z_kk^2}x,\sigma^{z_kk^2}y)$ in a small neighborhood of $p$ occupies 
a major factor of $f_{\underline{\boldsymbol{t}}}^{m_k}$ and 
absorbs fluctuations caused by non-linear factors. 
See also Remark \ref{r_accumulation_pt} for the sequence. 
Moreover, extra words $\underline{v}_{\,k+1}\in \{1,2\}^{k}$ will be added to the itineraries of (\ref{eqn_CC1})
so that the $\omega$-limit set of any point in the wandering domain contains $\Lambda$, which is possible by choosing entries of 
$v_k$ suitably in the proof 
of Proposition \ref{prop_exist_wd}. 

\subsection{Historic behavior }
The diffeomorphism $f_{n}$ obtained in Proposition \ref{prop_exist_wd} 
and the wandering domain $D=\mathrm{Int} R_{k_0}$ 
depend on the sequence $\boldsymbol{z}$.
In the proof of Theorem \ref{main-thm}, 
we express the dependence by $f_{n,\boldsymbol{z}}$ and $D_{\boldsymbol{z}}=\mathrm{Int} R_{k_0,\boldsymbol{z}}$.
From  the setting of  (\ref{eqn_CC1}), 
the  itinerary of any orbit starting from the wandering domain $D_{\boldsymbol{z}}$ contains 
$\underline{1}^{(z_k k^2)}$ and $\underline{2}^{(k^2)}$, 
and the remaining part of the itinerary is corresponding to at most order $k$ iterations of $f_{n,\boldsymbol{z}}$.  
Using this, 
one can show that for any $x\in D_{\boldsymbol{z}}$ there is 
the subsequence  $\{\mu_{x}(\widehat m_k)\}_{k= k_0}^\infty$  of partial averages  
$$\mu_{x}(\widehat m_{k})=\frac{1}{\widehat m_{k}+1}\sum_{i=0}^{\widehat m_{k}}\delta_{f_{n,\boldsymbol{z}}^i(x)}$$  
which tends to the mutually distinct two probability measures:  
$$\nu_{0}=\frac1{z_0+1}\bigl(z_0\delta_{p}+\delta_{\widehat p}\bigr), \quad
\nu_1=\frac1{z_0+2}\bigl((z_0+1)\delta_{p}+\delta_{\widehat p}\bigr)$$ 
as  $k\to \infty$,  
where 
$\widehat{p}$ is a saddle fixed point in the horseshoe $\Lambda$ other than $p$. 
It implies that the orbit of 
any point in the  wandering domain $D_{\boldsymbol{z}}$ has historic behavior.
This finishes the outline of  the proof of Theorem \ref{main-thm}. 

\begin{remark}\label{r_accumulation_pt}
Our diffeomorphism $f$ is linear only in a small neighborhood $U(p)$ of $p$ 
but not necessarily in neighborhoods 
of other points of $M$, which may yield some fluctuations to the dynamics.
To reduce influences of such fluctuations relatively, we first rearranged $f$ 
so that, for any $k\in\mathbb{N}$, the orbit of $x\in D$ under $f$ spends $k^2$ times in 
the linearizing neighborhood $U(p)$  
and $O(k)$ times in any other small neighborhoods.
However, in such a construction, the sequence of $\mu_x(m)$ would converge to the Dirac measure $\delta_p$ and hence the forward orbit of $x$ would not have historic behavior.
So, we have rearranged $f$ again so that the orbit spends $z_kk^2$ times in $U(p)$, $k^2$ times in $U(\widehat p)$ 
and $O(k)$ times in any other small neighborhoods.
In such an example, we can not expect that the sequence of $\mu_x(m)$ has various accumulation points, 
which is essentially different from the example in \cite[Section 9]{CV01}.
Note that we have taken the integer $z_0$ large enough so that any fluctuation caused in $U(\widehat p)$ is relatively small.
On the other hand, 
since $k^2/(z_kk^2)\geq 1/(z_0+1)>0$, the restriction of $\mu_x(m)$ on $U(\widehat p)$ does not converge to zero as $m\rightarrow \infty$.
\end{remark}


\section{Preliminaries}\label{sec.prep}
In this section, 
we present standard notions concerning planer homoclinic bifurcations  introduced by
Palis-Takens \cite{PT93}, which are given in forms adaptable to our discussions.

Throughout the remainder of this paper, 
suppose that $M$ is a closed surface and $r$ an integer $r$ with $2\leq r<\infty$.
Let $\mathcal{N}$ be any Newhouse open set in $\mathrm{Diff}^r(M)$.
Any element of $\mathcal{N}$ is arbitrarily $C^r$-approximated a diffeomorphism $f$ 
with a dissipative saddle periodic point $p$ 
whose stable manifold $W^{s}(p)$ and unstable manifold $W^{u}(p)$ 
have a quadratic tangency, 
where a periodic point $p$ of period $\mathrm{Per}(p)$ is said to be \emph{dissipative} if  $|\det(Df^{\mathrm{Per}(p)})_{p}|< 1$.
Denoting $f^{\mathrm{Per}(p)}$ again by $f$ for simplicity, 
one can suppose that $p$ is a dissipative saddle fixed point of $f$.

The following lemma is an elementary but crucial fact for justifying our replacement of 
$f$ by $f^{n\mathrm{Per}(p)}$ with $n\in \mathbb{N}$ here and later.

\begin{lemma}\label{l_f^nD}
Suppose that $a$ is a positive multiple of $\mathrm{Per}(p)$.
If $D$ is a contracting wandering domain for $f^a$ with $\omega_{f^a}(\boldsymbol{x})\ni p$ for some (or equivalently any) 
$\boldsymbol{x}\in D$, then $D$ is a contracting wandering domain also for $f$.
\end{lemma}

\begin{proof}
Since $C=\sup\bigl\{\|Df^i_{\boldsymbol{x}}\|\,;\, \boldsymbol{x}\in M,\ i=0,1,\dots,a-1\bigr\}<\infty$, 
$\mathrm{diam}(f^j(D))\leq C\mathrm{diam}(f^{am}(D))$ for any $j\in \mathbb{N}$, where
$m$ is the greatest integer with $ma\leq j$.
Thus our assumption $\lim_{m\to\infty}\mathrm{diam}(f^{am}(D))=0$ implies that  $\lim_{j\to\infty}\mathrm{diam}(f^{j}(D))=0$.  

If $D$ were not a wandering domain for $f$, then there would exist a positive integer $n$ with 
$D\cap f^n(D)\neq \emptyset$.
Then we have the chain 
$$D,\ f^n(D),\ f^{2n}(D),\ \dots,\ f^{an}(D)$$
with $f^{ni}(D)\cap f^{n(i+1)}(D)=f^{ni}(D\cap f^n(D))\neq \emptyset$ for $i=0,1,\dots,a-1$.
Since $\lim_{j\to\infty}\mathrm{diam}(f^{j}(D))=0$, 
\begin{equation}\label{eqn_f^anD}
\lim_{m\to\infty}\mathrm{diam}\bigl(f^{am}\bigl(D\cup f^n(D)\cup \cdots\cup f^{na}(D)\bigr)\bigr)=0.
\end{equation}
Moreover we have $D\cap W^s(p)=\emptyset$.
Otherwise, the sequence $\bigl\{f^{aj}(D)\bigr\}$ of open sets would converge to $W^u(p)$, 
which contradicts that $D$ is a contracting wandering domain.
Since $\omega_{f^a}(\boldsymbol{x})\ni p$ and $\boldsymbol{x}\not\in W^s(p)$ for any $\boldsymbol{x}\in D$, there exists a monotone increasing sequence $\{m_j\}$ of 
positive integers such that $f^{am_j}(\boldsymbol{x})$ converges to a point $\boldsymbol{z}$ in $W_{\mathrm{loc}}^u(p)\setminus \{p\}$ as $j\to \infty$.
Then $\lim_{j\to \infty}f^{a(m_j+n)}(\boldsymbol{x})=f^{an}(\boldsymbol{z})\in W_{\mathrm{loc}}^u(p)\setminus \{p\}$.
On the other hand, (\ref{eqn_f^anD}) implies that $\lim_{j\to \infty}f^{a(m_j+n)}(\boldsymbol{x})=\lim_{j\to \infty}f^{am_j}(\boldsymbol{x})=\boldsymbol{z}$ 
and so $f^{an}(\boldsymbol{z})=\boldsymbol{z}$.
This contradicts that $\boldsymbol{z}\in W_{\mathrm{loc}}^u(p)\setminus \{p\}$ is not a fixed point of 
$f^{an}$.
Thus $D$ is a contracting wandering domain for $f$.
\end{proof}

By a small perturbation near the homoclinic tangency $q$, 
one can also suppose that the tangency is quadratic. 
Moreover, as in \cite[Section 6.5]{PT93}, performing several arbitrarily small perturbations, 
we obtain a diffeomorphism 
which has both a transverse homoclinic intersection and 
a homoclinic tangency of the continuation of $p$ simultaneously. 
Using the transverse homoclinic intersection, 
one obtains a \emph{basic set} containing the continuation of $p$ called a \emph{horseshoe}, i.e.\ 
a compact invariant hyperbolic set which is transitive and contains a dense subset of periodic orbits and such that the restriction of the diffeomorphism on the set is conjugate to the $2$-shift. 
In what follows, 
if no confusion can arise, we will denote a $C^\infty$ diffeomorphism 
which is arbitrarily $C^{r}$ close to $f$ again by $f$. 
So we will work under the assumption that $f$ is a diffeomorphism of $C^\infty$-class 
and return to a diffeomorphism of $C^{\overline{r}}$-class with $3\leq \overline{r}<\infty$ at 
the stage of (\ref{def_f_mu,t}) in Subsection \ref{SS_Pert}.

In accordance with the discussion as above, we may suppose that 
$f$ has 
\begin{enumerate}[({S-}i)]
\item \label{setting1} 
a horseshoe $\Lambda$ containing a dissipative saddle fixed point $p$;
\item \label{setting2} 
a non-degenerate homoclinic tangency $q$ of $p$.
\end{enumerate}
Moreover, if a $C^{\infty}$ diffeomorphism $C^r$-close to $f$ 
satisfies an open and dense Sternberg condition \cite{St80} concerning the eigenvalues at 
the continuation of $p$, then $f$ is  $C^r$-linearizable in a neighborhood $U$
of the continuation.
So, one can suppose   that the above $f$ has 
\begin{enumerate}[({S-}i)]
\setcounter{enumi}{2}
\item \label{setting3} 
a $C^{r}$-linearizing coordinate in a neighborhood of $p$ 
such that $f(x,y)=(\lambda x, \sigma y)$ with $0<\lambda<1<\sigma$ and $\lambda\sigma<1$ (we replace 
$f$ by $f^{2}$ if $\lambda$ or $\sigma$ is negative).
\end{enumerate}
Note that one might proceed without (S-\ref{setting3}) by using techniques in Gonchenko et al.\ \cite{GST3,GST4,GST5, GST7}, 
but  (S-\ref{setting3}) is more appropriate here from a standpoint of simple descriptions.

Let $\{f_{\mu}\}_{\mu\in \mathbb{R}}$ be a 
one-parameter family of $C^{\infty}$ diffeomorphisms on  $M$ 
with $f_{0}=f$ and such that the homoclinic quadratic tangency $q$ of $p$ unfolds generically at $\mu=0$. 
The \emph{renormalization} of return maps near the tangency provides a better description as follows.

\begin{theorem}[Renormalization Theorem, Palis-Takens \cite{PT93}]\label{lem.renormalization}
There exists an integer $N_*>0$ such that, for any sufficiently large integer $n>0$, 
there are a $C^{r}$ parametrization  $\Theta_{n}:\mathbb{R}\to \mathbb{R}$ and a $\bar \mu$-dependent $C^r$ coordinate change $\Phi_{n}:\mathbb{R}^2\to M$ satisfying the following:
\begin{itemize}
\item
$\frac{d\Theta_{n}}{d\mu}(\bar \mu)>0$;
\item
for any $(\bar\mu,\bar x,\bar y)\in \mathbb{R}\times \mathbb{R}^2$,
$(\Theta_n(\bar \mu), \Phi_n(\bar x,\bar y))$ converges to $(0, q)$ as $n\rightarrow \infty$;
\item
for any $\bar \mu\in \mathbb{R}$, the diffeomorphisms $\varphi_n$ 
on $\mathbb{R}^2$ defined by
\begin{equation}\label{eqn_phi_n}
\varphi_n:(\bar{x},\bar{y})\longmapsto 
  \Phi_{n}^{-1}\circ f_{\mu_{n}}^{N_*+n}\circ \Phi_{n}(\bar{x},\bar{y})
\end{equation}
converge to 
\begin{equation}\label{henon-like endo}
( \bar{x}, \bar{y})\longmapsto( \bar{y}, \bar{y}^{2}+\bar{\mu} ).
\end{equation}
as $n\rightarrow \infty$  in the $C^r$ topology, where $\mu_{n}:=\Theta_{n}(\bar{\mu})$.
\end{itemize}
\end{theorem}
\begin{proof}
See Palis-Takens \cite[\S 3.4, Theorem 1]{PT93} for the proof.
\end{proof}

Here the integer $N_*$ is taken so that $f^{-N_*}(q)$ is a point in $W_{\mathrm{loc}}^u(p)$ sufficiently near $p$.
In short
this lemma ensures that the return maps $f^{N_*+n}_{\mu_{n}}$ near the homoclinic tangency 
can be approximated by diffeomorphisms 
so called \emph{H\'enon-like maps} which are close to the quadratic endomorphism (\ref{henon-like endo}).

For simplicity, we consider the map
$$\varphi_{\mu,\nu}(x,y)=(y, \mu+\nu x+y^{2}), $$ 
which is equivalent to the original H\'enon map given in (\ref{original}) via 
appropriate parameter and coordinate changes, see \cite{KLS10}. 
More general H\'enon-like families are obtained by adding small higher-order terms 
to the above form.
From  \cite[\S 6.3, Proposition~1]{PT93}, 
for a given $m\geq 3$, there exist 
a neighborhood $\mathcal{U}(-2,0)$ of the point $(-2,0)$ in the parameter space and 
continuous maps 
$P$, $Q_{m}$ and $\Gamma_{m}$ 
which map $(\mu,\nu)\in \mathcal{U}(-2,0)$ to the fixed point $P_{\mu,\nu}$, 
the periodic point $Q_{m;\mu,\nu}$, and 
the non-trivial invariant set $\Gamma_{m;\mu,\nu}$ for $\varphi_{\mu,\nu}$, respectively, 
and furthermore satisfy the following properties: 
\begin{itemize}
\item 
$P_{-2,0}=(2,2)$ is a saddle fixed point and $Q_{m;-2,0}$ is a saddle periodic point of period $m$ both of which are contained in 
a parabolic arc which is convex downward between $P_{-2,0}$ and $\tilde{P}_{-2,0}:=(-2, 2)$, 
see Figure \ref{fig_4_4}. 
\item 
For any $(\mu,\nu)\in \mathcal{U}(-2,0)$ with $\nu\neq 0$, $\Gamma_{m;\mu,\nu}$ is a basic set  
containing the orbit of $Q_{m;\mu,\nu}$.
\end{itemize}
More detailed information on these ingredients will be given 
in the next section.
Note that 
the same properties hold for any  H\'enon-like map which is sufficiently close to $\varphi_{-2,0}$. 

By the above results, there exists a positive integer $n(m)>0$ such that, for any $n\geq n(m)$, 
$f_{\nu_n}$ is a diffeomorphism arbitrarily $C^{r}$ close to the original $f$ 
and satisfying not only (S-\ref{setting1})--(S-\ref{setting3}) but also   
\begin{enumerate}[({S-}i)]
\setcounter{enumi}{3}
\item \label{setting4} 
the restriction $\varphi_n$ of $f^{N_*+n}_{\mu_{n}}$ near the tangency $q$ is a return map $C^{r}$-approximated by an H\'enon-like map and has the continuation $P_{\bar\mu}$ of the saddle fixed point $P$, 
the continuation $\Gamma_{m;\bar\mu}$ of the basic set $\Gamma_{m}$ which 
contains the continuation $Q_{m;\bar\mu}$ of the saddle periodic point $Q_{m}$ of period $m$, where $\bar\mu=\Theta_n^{-1}(\mu_n)$.
\end{enumerate}

We here recall two important relations on a pair of basic sets.
We say that disjoint basic sets 
$\Lambda$ and $\Gamma$ are \emph{homoclinically related}   
if both 
$W^u(\Lambda)\cap W^s (\Gamma)$ and 
$W^s(\Lambda)\cap W^u (\Gamma)$ contain non-trivial transverse intersections.  
Basic sets  $\Lambda$ and $\Gamma$ for $f$ have a  \emph{$C^{2}$-robust tangency}  
if there exists a $C^{2}$ neighborhood $\mathcal{U}(f)$ of $f$ satisfying the following condition: for every  $g\in \mathcal{U}(f)$, 
either $W^u(\Lambda_{g})\cap W^s (\Gamma_{g})$ or 
$W^s(\Lambda_{g})\cap W^u (\Gamma_{g})$  contains a tangency, 
where $\Lambda_{g}$ and $\Gamma_{g}$ are the continuations of $\Lambda$ and $\Gamma$, respectively.

By \cite[Section 6.4]{PT93}, we may also suppose that 
\begin{enumerate}[({S-}i)]
\setcounter{enumi}{4}
\item \label{setting5} 
the continuation $\Lambda_{n}:=\Lambda(f_{\mu_n})$ of the horseshoe $\Lambda$ in (S-\ref{setting1}) and 
the basic set $\Gamma_{m,n}:=\Gamma_m(\varphi_n)$ in (S-\ref{setting4})  are homoclinically related,  and 
they have a $C^{2}$-robust tangency. 
To be more precise, 
$W^{u}(\Lambda_{g})\cap W^{s}(\Gamma_{m, g})$ contains a tangency $a$ 
for any diffeomorphism $g$ $C^{2}$-near $f$, see Figure \ref{fig_3_1}.
\end{enumerate}

\begin{figure}[hbt]
\centering
\scalebox{0.7}{\includegraphics[clip]{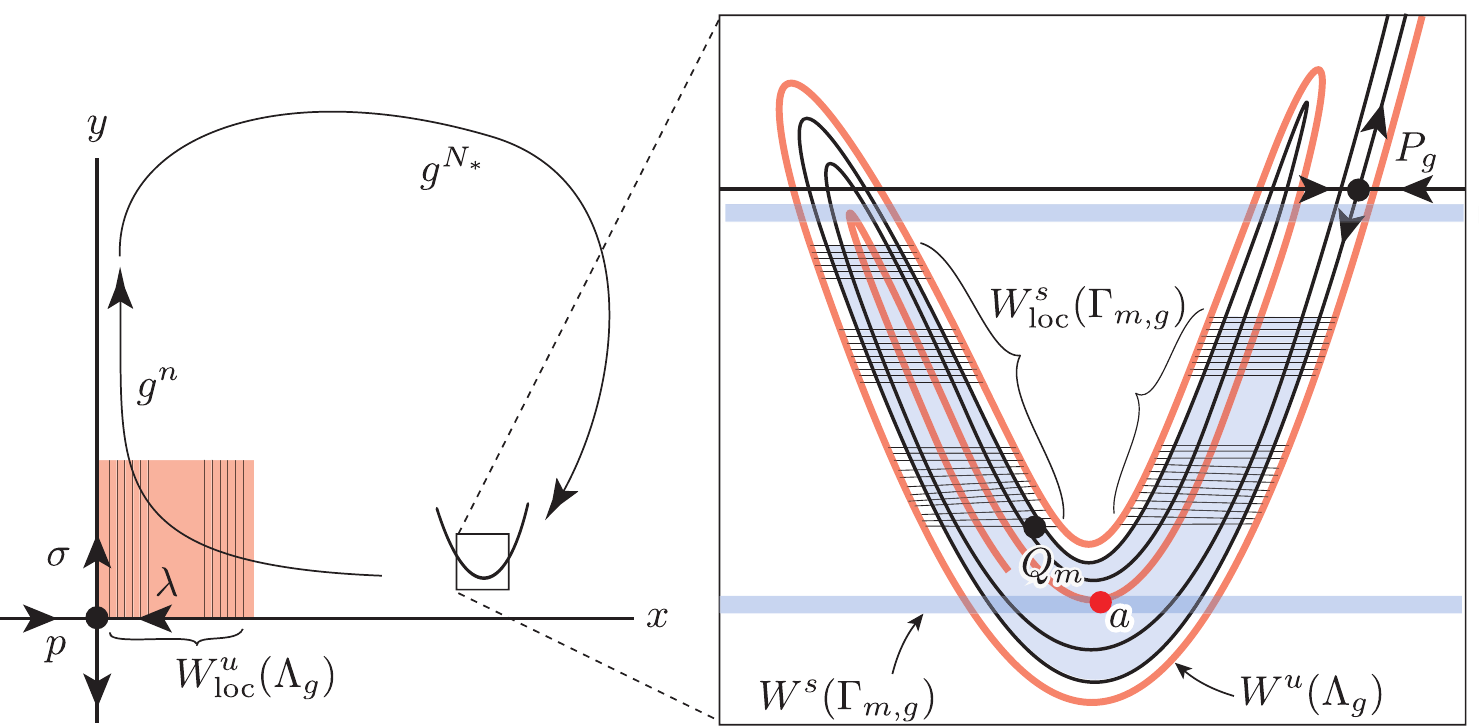}}
\caption{A homoclinical relation between $\Lambda_g$ and $\Gamma_{m,g}$ and 
a quadratic tangency between $W^u(\Lambda_g)$ and $W^s(\Gamma_{m,g})$.}
\label{fig_3_1}
\end{figure}

Now we recall the construction of the return map $\varphi_n$ by Palis-Takens.
There exists a small rectangle $D_n$ near $q$ such that $\Gamma_{m,n}=\bigcap_{\,k=-\infty}^\infty
f_{\mu_n}^{k(N_*+n)}(D_n)$.
According to \cite[Section 6.4]{PT93}, there is a transverse intersection point $b$ of $W^u(p)$ and $W^s(P)$ 
such that
\begin{enumerate}[({S-}i)]
\setcounter{enumi}{5}
\item \label{setting6}
the sub-arc $\alpha^u$ in $W^u(p)$ connecting $p$ with $b$ is disjoint from the union $X_n=D_n\cup f_{\mu_n}(D_n)\cup\cdots\cup 
f_{\mu_n}^{N_*+n}(D_n)$, and 
\item\label{setting7}
$f_{\mu_n}^i(D_n)\cap f_{\mu_n}^j(D_n)=\emptyset$ for any $i,j\in \{1,2,\dots,N_*+n\}$ with $i\neq j$, see Figure \ref{fig_3_2}.
\end{enumerate}
\begin{figure}[hbt]
\centering
\scalebox{0.7}{\includegraphics[clip]{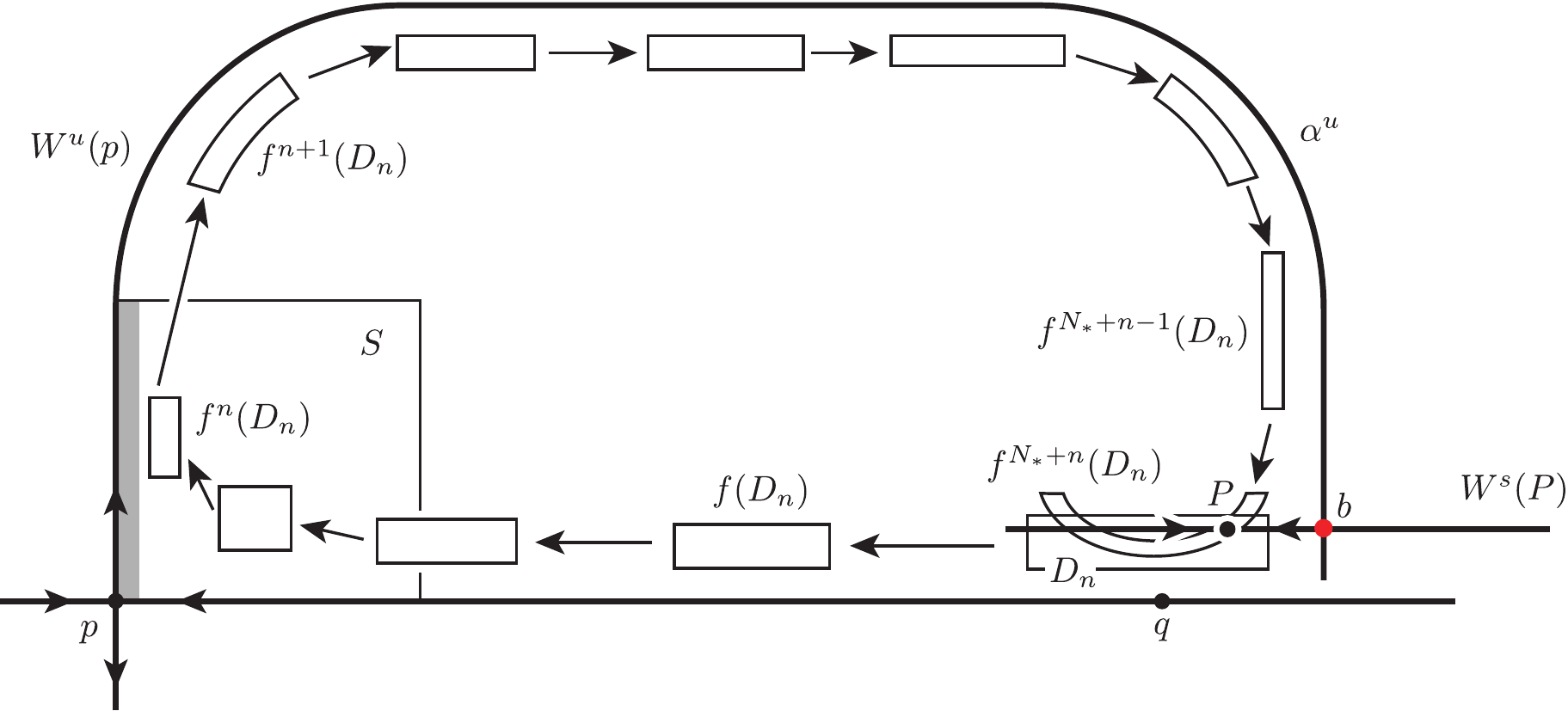}}
\caption{$f_{\mu_n}$, $P(\varphi_n)$ are represented shortly by $f$ and $P$ respectively.
The shaded region disjoint from $X_n$ contains the supports of the perturbations of $f_{\mu_n}$ considered in Subsections \ref{subsec.LP} and \ref{SS_Pert} respectively.}
\label{fig_3_2}
\end{figure}
Note that the condition (S-\ref{setting7}) does not necessarily hold for $i,j\not\in \{1,2,\dots,N_*+n\}$.
In fact, for any integer $N$ sufficiently larger than $N_*$, $f_{\mu_n}^N(\alpha^u)$ meets all leaves of 
$W_{\mathrm{loc}}^s(\Gamma_m)$ transversely, which is suggested in Figure \ref{fig_3_1}, and hence 
in particular $f_{\mu_n}^N(\alpha^u)\cap D_n\neq\emptyset$.

A nonempty compact subset $K$ of an interval $I$ is called a \emph{Cantor set} if 
$K$ has neither interior points nor isolated points. 
A \textit{gap} of the Cantor set $K$ is the closure of a connected component of $I\setminus K$. 
Let $G$ be a gap and $p$ a boundary point of $G$. 
A closed interval $B\subset  I$ is called the \textit{bridge} at $p$ 
if $B$ is maximal among all closed intervals $B'$ in $I$ with $G\cap B'=\{p\}$ and such that $B'$ does not intersect any gap whose length is at least that of $G$. 
The  \textit{thickness}  for the Cantor set  $K$ at $p$ is defined by  $\tau(K,p)= |B|/|G|$,  
where $B$ and  $G$ are a bridge and a gap satisfying  $G\cap B=\{p\}$.
The \textit{thickness} $\tau(K)$ of $K$ is the infimum  over these $\tau(K,p)$ for all boundary points $p$ of gaps of $K$.
Two Cantor sets $K_1$ and $K_2$ are said to be \emph{linked} if 
neither $K_{1}$ is contained in the interior of any gap of $K_{2}$ nor $K_{2}$ is contained in the 
interior of any gap of $K_{1}$.
\emph{Gap Lemma} (see \cite[\S 4]{N79}, \cite[\S 4.2]{PT93}, \cite{Kraft92}) shows that, 
for any linked Cantor sets $K_{1}$ and $K_{2}$ with $\tau(K_{1})\tau(K_{2})>1$, 
$K_{1}\cap K_{2}\neq \emptyset$ holds.  

We say that a bridge $B$ of $K$ is \emph{adjacent} to a gap $G$ if $B\cap G\neq\emptyset$ and 
$\mathrm{Int} B\cap G=\emptyset$.
If two bridges $B$, $B'$ are adjacent to a common gap $G$, then 
$G$ is called the \emph{connecting gap} for $B$ and $B'$ and denoted by $\mathrm{Gap}(B,B')$.

\section{Bounded distortions}\label{sec.bounded-distortions}
\subsection{Classical bounded distortion lemma}\label{subsec.trad}
A Cantor set $K$ in an interval $I$ is said to be \emph{dynamically defined} 
if  the following conditions hold: 
there are mutually disjoint closed sub-intervals $B_{1}, B_{2},\ldots, B_{r} \subset I$ and a 
differentiable map $\Psi$ 
defined in a neighborhood $U$ of $B_{1} \sqcup\dots \sqcup B_{r}$ in $I$ 
such that 
\begin{itemize}
\item $\Psi$ is uniformly hyperbolic on $K$, that is, 
there are constants $C>0$ and $\sigma>1$ such that $|(\Psi^{n})^{\prime}(x)|\geq C\sigma^{n}$ for 
every $x\in K$ and $n\geq 1$, and 
\item $\{B_{1},\dots, B_{r}\}$ is a Markov partition
satisfying 
$$K=\bigcap_{n\in \mathbb{N}} \Psi^{-n}(B_{1} \sqcup\dots \sqcup B_{r}).$$
\end{itemize}

The next classical result, called \emph{Bounded Distortion Lemma},  
will play an important role in this paper.

\begin{lemma}[Palis-Takens \cite{PT93}]\label{lem.bdp0}
Let $K$ be a dynamically defined Cantor set as above associated with a uniformly hyperbolic map $\Psi$ and $a_{0}$ the minimum positive integer with 
$C\sigma^{a_{0}}>1$.
If $\Psi$ satisfies the $C^{1+\alpha}$ H\"older condition for some $0<\alpha\leq1$, 
then, for every $\delta>0$, there exists a constant $c(\delta)>0$ satisfying
\begin{equation}\label{eqn_bdp0}
e^{-c(\delta)} \leq
|(\Psi^{na_{0}})^{\prime}(q)|  |(\Psi^{na_{0}})^{\prime}(\tilde{q})|^{-1}
\leq e^{c(\delta)}
\end{equation} 
for any $q, \tilde{q}\in I$ and integer $n\geq 1$ such that  
(i) $| \Psi^{na_{0}}(q)- \Psi^{na_{0}}(\tilde{q})|\leq \delta$;
(ii) the interval between $\Psi^{i}(q)$ and $\Psi^{i}(\tilde{q})$ is contained in $B_{1} \sqcup\dots \sqcup B_{r}$ for all $0\leq i\leq (n-1)a_{0}$.
Moreover, $c(\delta)$  is of order $\delta^{\alpha}$.
In particular, $c(\delta)$ converges to zero as $\delta\rightarrow 0$.
\end{lemma}

\begin{proof}
Since $\Psi$ is uniformly hyperbolic on $K$,  
$|(\Psi^{n})^{\prime}(x)|\geq C\sigma^{n}$ for 
any $x\in K$ and $n\geq 1$.
Then there exists a constant $\gamma>1$ and an open neighborhood $V$ of $K$ in $U$ such that $|(\Psi^{a_0})^{\prime}(x)|\geq \gamma$ for any $x\in V$ and $n\geq 1$.
For all sufficiently small $\delta>0$, 
any points $q$, $\tilde q$ of $I$ satisfying the conditions (i) and (ii) bound an interval contained in $V$.
Hence, one can apply \cite[\S 4.1, Theorem~1]{PT93} directly to the expanding map $\Psi^{a_{0}}$, and 
obtain a constant 
$$c(\delta)=\widetilde{C}\delta^{\alpha}\frac{(C\sigma^{a_{0}})^{-\alpha}}{1-(C\sigma^{a_{0}})^{-\alpha}}$$
satisfying the condition (\ref{eqn_bdp0}), where $\widetilde{C}$ is a positive constant independent of $\delta$. This completes the proof.
\end{proof}

\subsection{Horseshoes and $s$-bridges}\label{subsec.horse}
Let us recall a simple example of a dynamically defined Cantor set   
and give the bounded distortion property for the set.

Consider 
a two-dimensional $C^{r}$ diffeomorphism $f$ admitting a horseshoe $\Lambda$ 
as in (S-\ref{setting1}) of Section \ref{sec.prep}, which contains a saddle fixed point $p$.
Let $\mathcal{F}_{\mathrm{loc}}^{u}(\Lambda)$ and $\mathcal{F}_{\mathrm{loc}}^{s}(\Lambda)$ be local unstable and stable foliations on $S=[0,2]\times [0,2]$ compatible with $W_{\mathrm{loc}}^{u}(\Lambda)$ and 
$W_{\mathrm{loc}}^{s}(\Lambda)$ respectively.
Here we may assume that
\begin{itemize}
\item
$W_{\mathrm{loc}}^s(p)=[-2,2]\times \{0\}$, 
$W_{\mathrm{loc}}^u(p)=\{0\}\times [-2,2]$,
\item
$[0,2]\times \{2\}$ is a leaf of $\mathcal{F}_{\mathrm{loc}}^s(\Lambda)$ disjoint from $W_{\mathrm{loc}}^s(\Lambda)$ 
and 
$\{2\}\times [0,2]$ is a leaf of $\mathcal{F}_{\mathrm{loc}}^u(\Lambda)$ disjoint from $W_{\mathrm{loc}}^u(\Lambda)$.  
\end{itemize}
Let $\pi_{\mathcal{F}_{\mathrm{loc}}^{u}(\Lambda)}:S\to W^{s}_{\mathrm{loc}}(p)$
 be the projection along the leaves of $\mathcal{F}_{\mathrm{loc}}^{u}(\Lambda)$, 
and 
$\pi_{\mathcal{F}_{\mathrm{loc}}^{s}(\Lambda)}:S\to W^{u}_{\mathrm{loc}}(p)$ the projection along the leaves of $\mathcal{F}_{\mathrm{loc}}^{s}(\Lambda)$. 
Consider the Cantor sets   
\begin{equation}\label{def.Ks}
K^{s}_{\Lambda}:=\pi_{\mathcal{F}_{\mathrm{loc}}^{u}(\Lambda)}(\Lambda), \quad
K^{u}_{\Lambda}:=\pi_{\mathcal{F}_{\mathrm{loc}}^{s}(\Lambda)}(\Lambda),
\end{equation}
associated with $\Lambda$ dynamically defined by 
$\Psi_{s}:=\pi_{\mathcal{F}_{\mathrm{loc}}^{u}(\Lambda)}\circ f^{-1}$ and 
$\Psi_{u}:=\pi_{\mathcal{F}_{\mathrm{loc}}^{s}(\Lambda)} \circ f$, respectively.  
Since $f$ is a $C^{r}$ map and $\pi_{\mathcal{F}_{\mathrm{loc}}^{u}}$, $\pi_{\mathcal{F}_{\mathrm{loc}}^{s}}$ 
are $C^{1+\alpha}$ maps with $0<\alpha<1$ (see \cite[\S 4.1]{PT93}), it follows that 
both $\Psi_{s}$ and $\Psi_{u}$ are of $C^{1+\alpha}$ class.
Note that both $\Psi_{s}$ and $\Psi_{u}$ are expanding maps.

\begin{remark}\label{r_three_foliations}
\begin{enumerate}[(1)]
\item
In this paper, we use three local `stable' foliation on $S$, one of which is the above $\mathcal{F}_{\mathrm{loc}}^{s}(\Lambda)$.
The other two are $\mathcal{F}_S$ in Subsection \ref{subsec.Hetero} and 
$\mathcal{G}_{\mathrm{loc}}^{s}(\Lambda)$ in Subsection \ref{SS_Pert}.
\item
The unstable foliation $\mathcal{F}_{\mathrm{loc}}^{u}(\Lambda)$ will be chosen carefully in Subsection 
\ref{subsec.Hetero} 
so as to be suitable to our purpose.
However the Cantor set $K_\Lambda^s$ and hence the bridges and gaps associated with $K_\Lambda^s$ 
are independent of the choice.
\end{enumerate}
\end{remark}

Let $B^s(0)$ (respectively $B^u(0)$) be the smallest interval in $W_{\mathrm{loc}}^s(p)$ (respectively $W_{\mathrm{loc}}^u(p)$) containing $K^s$ (respectively $K^u$).
We give here descriptions associated only with $K_\Lambda^{s}$.
Similar arguments work also in the case of $K_\Lambda^u$.  
There exists a Markov partition of $K_\Lambda^{s}$ in $B^{s}(0)$ which consists of two components, and 
denote one of them by $B^{s}(1; 1)$ and the other by  $B^{s}(1; 2)$. 
For each integer $k\geq 1$ and $w_i\in \{1,2\}$ $(i=1,\dots,k)$, we define the interval $B^{s}(k;w_{1}\ldots w_{k})$, called an \emph{$s$-bridge of generation $k$}, as 
$$B^{s}(k;w_{1}\ldots w_{k})=\left\{ x\in I\ ;\  \Psi_{s}^{i-1}(x)\in B^{s}(1;w_{i}),\ i=1, \ldots, k\right\},$$
where the sequence $w_{1}w_{2}\ldots w_{k}$ is the \emph{itinerary} for the $s$-bridge. 
If one writes $\underline{w}=w_{1}w_{2}\ldots w_{k}$, then $\underline{w}^{-1}$ stands for 
 the reverse sequence $w_{k}w_{k-1}\ldots w_{1}$. 
The \emph{$u$-bridges} $B^{u}(k;w_{1}\ldots w_{k})$ of \emph{generation $k$} associated with $K_\Lambda^{u}$ 
can be defined similarly by using $\Psi_{u}$.
For our convenience, we regard that $B^s(0)$ and $B^u(0)$ are the bridges of generation $0$ with 
empty itinerary.

For any integer $k\geq 0$, let $\mathcal{B}^{s}_{k}$ be the collection of all $B^{s}(k;w_{1}\ldots w_{k})$, see Figure \ref{fig_4_1}.
Note that $\mathcal{B}^{s}_{k}$ consists of mutually disjoint $2^{k}$ $s$-bridges. 
The union $\mathcal{B}^s=\bigcup_{k=0}^\infty \mathcal{B}_k^s$ is the set of all $u$-bridges of $K_\Lambda^s$.
The set $\mathcal{B}^u=\bigcup_{k=0}^\infty \mathcal{B}_k^u$ of all $u$-bridges of $K_\Lambda^u$ 
is defined similarly.
\begin{figure}[hbt]
\centering
\scalebox{0.9}{\includegraphics[clip]{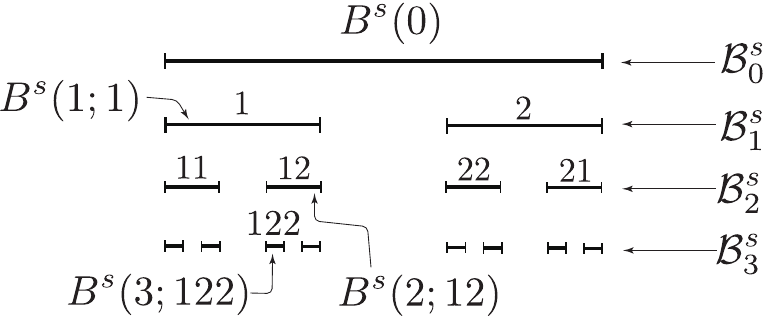}}
\caption{A collection of nested stable bridges.
For any $w_1,w_2\in \{1,2\}$, 
$\Psi_s(B^s(2;w_1w_2))=B^s(1;w_2)$.}
\label{fig_4_1}
\end{figure}

If necessary replacing the original $f$ by $f^{a_0}$ with a large integer $a_0$, $S$ by a thiner sub-rectangle $S'$ with $\partial S'\supset [0,2]\times \{0\}$ and such that $\Lambda'=\bigcap_{n=-\infty}^\infty f^{na_0}(S')$ 
is an $f^{a_0}$-invariant basic set, we may suppose that 
\begin{equation}\label{eqn_BB<B}
\max\bigl\{
|B^{s}(1; 1)|,\ |B^{s}(1; 2)|\bigr\}< \frac9{21}|B^{s}(0)|, 
\end{equation}
where $|\cdot|$ stands for the length of the corresponding interval.
See Figure \ref{fig_4_2}.
\begin{figure}[hbt]
\centering
\scalebox{0.8}{\includegraphics[clip]{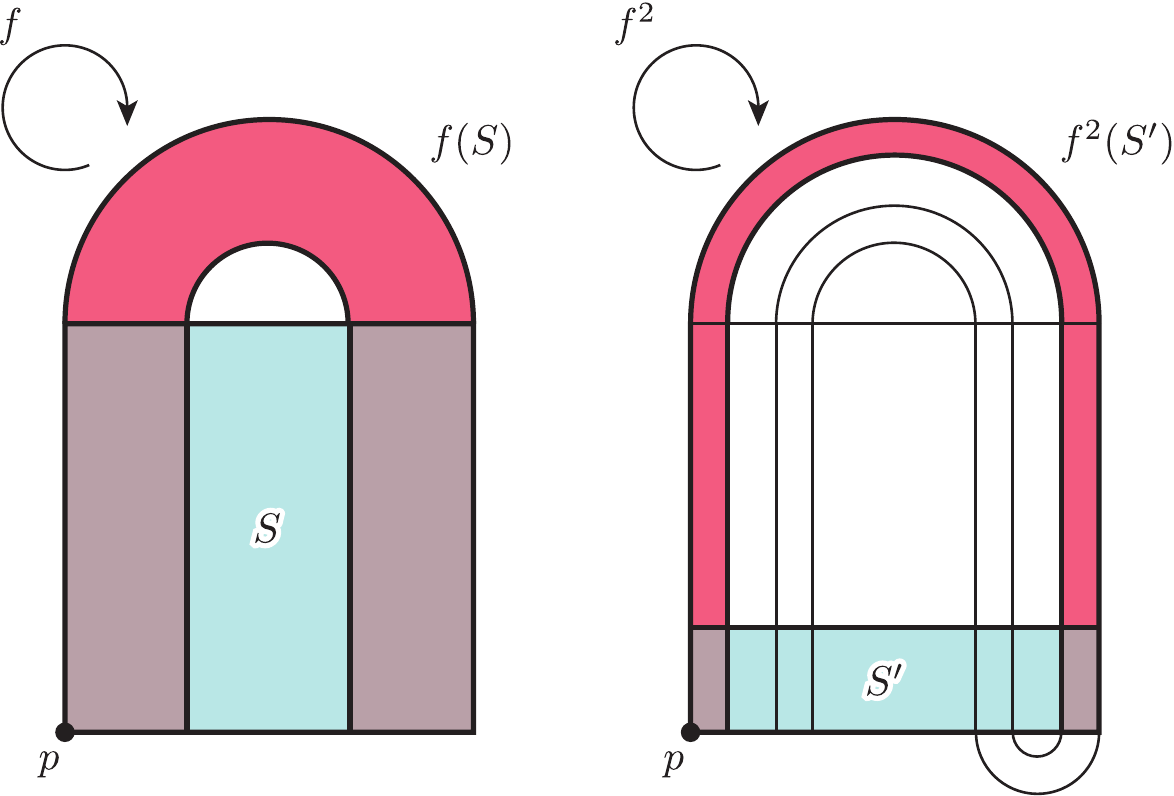}}
\caption{Replacement of a horseshoe by a more slender one.
The case of $a_0=2$.}
\label{fig_4_2}
\end{figure}
Note that the new bridges $B^s(1;1)$, $B^s(1;2)$ for $f^{a_0}$ coincide with the original bridges 
$B^s(a_0;11\cdots1)$ and $B^s(a_0;21\cdots1)$ for $f$ 
respectively.
Since moreover one can choose $a_0$ so that $\max\bigl\{
|B^{s}(1; 1)|,\ |B^{s}(1; 2)|\bigr\}$ is arbitrarily small, we may also 
assume by Lemma \ref{lem.bdp0} that
\begin{equation}\label{eqn_Psiqq}
\frac9{10}\leq |(\Psi_s^{a_0n})'(q)||(\Psi_s^{a_0n})'(\tilde q)|^{-1}\leq \frac{11}{10}
\end{equation}
for any 
$q, \tilde q\in I$ such that the intervals between $\Psi_s^{a_0i}(q)$ and $\Psi_s^{a_0i}(\tilde q)$ 
$(i=0,1,\dots,n-1)$ are contained in $B(1;1)\sqcup B(1;2)$.
We set $f^{a_0}$, $S'$, $\Lambda'$ again by $f$, $S$ and $\Lambda$ respectively, and retake the local coordinate on $M$ near $p$ so that the new $S$ equals $[0,2]\times [0,2]$.

\begin{lemma}[Bounded distortions  for  $s$-bridges]\label{lem.bdp1}
The exist constants $r_{s+}\geq r_{s-}>2$ satisfying 
the following conditions.
For any integer $k>0$, 
let $B^{s}_{k}$ and $B^{s}_{k+1}$ be $s$-bridges for $K_\Lambda^{s}$ of generation $k$ and $k+1$ 
such that 
the first $k$ entries in the itinerary of $B^{s}_{k+1}$
are identical to the entries in the itinerary of $B^{s}_{k}$, that is,
$$B^{s}_{k}:=B^{s}(k;w_{1}\ldots w_{k}),\quad
B^{s}_{k+1}:=B^{s}(k+1;w_{1}\ldots w_{k} w_{k+1}).$$
Then the inequality  
$$
r_{s-}\leq |B_{k}^{s}| | B_{k+1}^{s}|^{-1} \leq r_{s+}.
$$
holds.
\end{lemma}
\begin{proof}
By the mean-value theorem, for any bridges $B^{s}_{k}$ and $B^{s}_{k+1}$ as above, 
there are points $q\in B^{s}_{k+1}$, 
$\tilde{q}\in B^{s}_{k}$ 
such that 
$$
|B^s(0)|=|B^{s}_{k}|| (\Psi_s^{k})^{\prime}(\tilde{q}) |,\quad 
|B^{s}(1;w_{k+1})|=|B^{s}_{k+1}|| (\Psi_s^{k})^{\prime}(q) |.
$$
By (\ref{eqn_BB<B}) and (\ref{eqn_Psiqq}),
\begin{align*}
|B^{s}_{k}| |B^{s}_{k+1}|^{-1}&=|B^s(0)| |B^s(1;w_{k+1})|^{-1} 
|(\Psi_s^{k})^{\prime}(q) | | (\Psi_s^{k})^{\prime}(\tilde{q}) |^{-1}\\
&>\frac{21}9\cdot \frac9{10}=\frac{21}{10}=:r_{s-}.
\end{align*}
Moreover, we have
$$|B^{s}_{k}| |B^{s}_{k+1}|^{-1}\leq \frac{|B^s(0)|}{\min\bigl\{|B^s(1;1)|,\, |B^s(1;2)|\bigr\}}
\cdot \frac{11}{10}=:r_{s+}.$$
This completes the proof.
\end{proof}

\subsection{Quadratic maps and $u$-bridges} \label{subset.one-dim}
We recall another Cantor set  for 
one-dimensional maps fundamental properties of which are succeeded by H\'enon-like maps in Section \ref{subsec.Henon}.

Let $F_{\mu}$ be the family of one-dimensional quadratic maps on $\mathbb{R}$ defined as 
\begin{equation}\label{eqn_Fmu}
F_{\mu}(x)=x^{2}+\mu,
\end{equation}
where $\mu$ is a real parameter. 
For any integer $m\geq 3$,  
$F_{\mu}$ has a periodic orbit of period $m$ if $\mu$ is sufficiently close to $-2$.
Denote by $q_{1}$ and $q_{2}$, respectively, 
the minimum and maximum points of the $m$-periodic orbit, which satisfies $F_{\mu}(q_{1})=q_{2}$. 
We denote by 
$A^{u}(0)$ the interval $[q_{1}, q_{2}]$, and 
define the mutually disjoint $m-1$ intervals as 
$$
A^{u}(1;z_{1}):=\left\{
\begin{array}{ll}
F_{\mu}^{-1}(A^{u}(0))\cap\{x<0\} & \text{if}\ z_{1}=1; \\
F_{\mu}^{-z_{1}+1}(
F_{\mu}^{-1}(A^{u}(0))\cap\{x<0\}
)\cap\{x>0\} &  \text{if}\  z_{1}=2,\ldots,m-1.
\end{array}
\right.
$$
Moreover,  
for every integer $k \geq 2$,  we inductively define  the $(m-1)^{k}$ intervals 
$A^{u}(k;z_{1}z_{2}\ldots z_{k})$ as
 $$
A^{u}(k;z_{1}z_{2}\ldots z_{k}):=A^{u}(k-1;z_{1}\ldots z_{k-1})\cap F_{\mu}^{-z_{1}}(A^{u}(k-1;z_{2}\ldots z_{k})),
 $$
where  each entry $z_{i}$ is an element of $\{1,2,\ldots, m-1\}$.  
See Figure \ref{fig_4_3} for the case of $m=4$. 
The interval  $A^{u}(k;z_{1}z_{2}\ldots z_{k})$ is called a \emph{$u$-bridge} for $F_\mu$ of \emph{generation} $k$, 
and the sequence $(z_{1}\ldots z_{k})$ is the \emph{itinerary} of the $u$-bridge. 
A unique point of $\partial A^{u}(k;z_{1}z_{2}\ldots z_{k})\cap \partial A^{u}(k+1;z_{1}z_{2}\ldots z_{k}1)$ (resp.\ $\partial A^{u}(k;z_{1}z_{2}\ldots z_{k})\cap \partial A^{u}(k+1;z_{1}z_{2}\ldots z_{k}m-1)$) is called the \emph{leading point} (resp.\ \emph{bottom point}) of $A^{u}(k;z_{1}z_{2}\ldots z_{k})$.
Consider the Cantor set  
 $$
 K^{u}_m(\mu):=A^{u}(0)\cap  \bigcap_{k=1}^{\infty} 
 \bigsqcup_{\begin{subarray}{c}(z_{1},\ldots, z_{k})\\ \in  \{1,\ldots,m-1\}^{k}\end{subarray}}A^{u}(k;z_{1}\ldots z_{k})
 $$
dynamically defined by $F_{\mu}$ and associated with the $m$-periodic orbit. 
For our convenience, we regard that $A^u(0)$ is a $u$-bridge of generation $0$ with 
empty itinerary.
The \emph{leading gap} of $A^{u}(k;z_1\ldots z_{k})$ is the gap in $A^{u}(k;z_1\ldots z_{k})$  
bounded by $A^{u}(k+1;z_1\ldots z_{k}1)$ and $A^{u}(k+1;z_1\ldots z_{k}2)$.
For each integer $k\geq 0$, let $\mathcal{A}^{u}_{k}$ be the collection of all $k$-bridges $A^{u}(k;z_{1}\ldots z_{k})$, 
see Figure \ref{fig_4_3}. 
\begin{figure}[hbt]
\centering
\scalebox{0.8}{\includegraphics[clip]{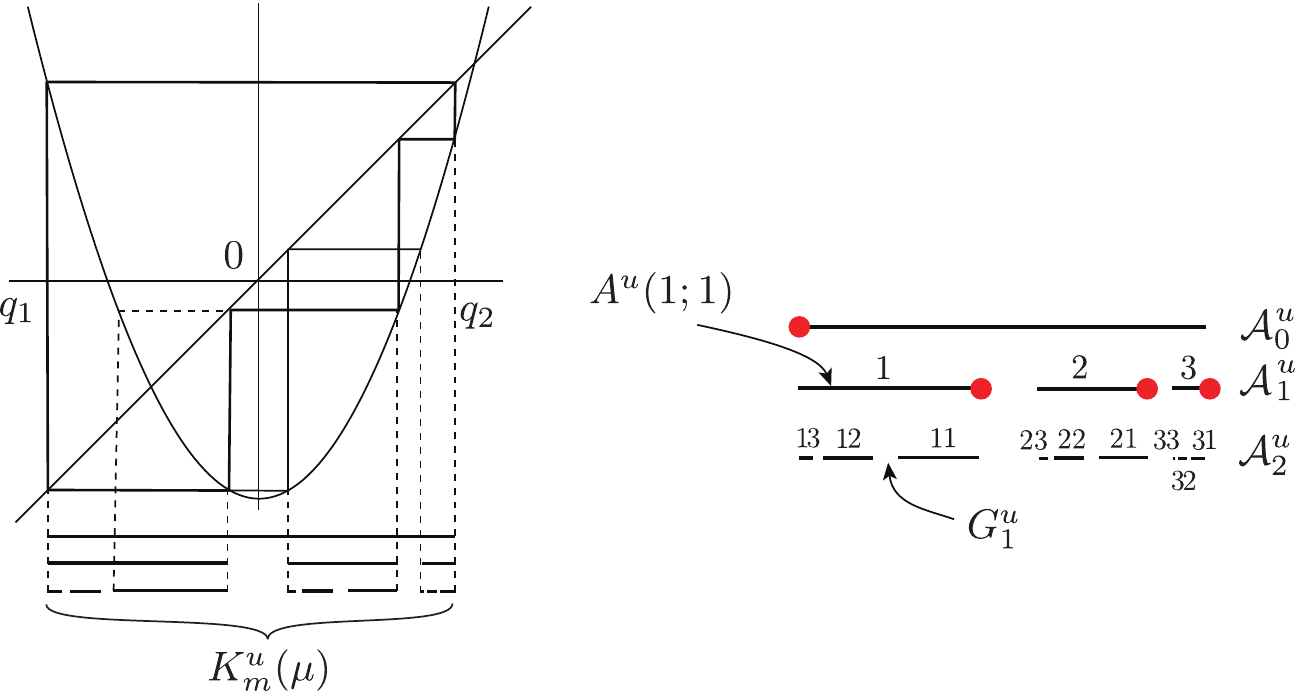}}
\caption{A nested sequence of unstable bridges.
For any $z\in\{1,2,3\}$, $F_\mu(A^u(2;3z))=A^u(2;2z)$, $F_\mu(A^u(2;2z))=A^u(2;1z)$, and $F_\mu(A^u(2;1z))=A^u(1;z)$.
The red dots represent the leading points of $u$-bridges of generation $0$ and $1$.
$G_1^u$ is the leading gap of $A^u(1;1)$.}
\label{fig_4_3}
\end{figure}

\begin{remark}\label{rmk.hyperbolicity}
As $\mu=-2$, $F_{-2}$  
is topologically conjugate to the tent map $T:x\mapsto |2x-1|+1$ 
via $g:x\mapsto 2-4\sin^{2}(\pi/2)x$. It implies that, 
if $\mu$ is contained in a small neighborhood $I$ of $-2$ and $m\geq 3$, 
$K^{u}_m(\mu)$ is a uniformly hyperbolic set for $F_{\mu}$, see  \cite[\S 6.2]{PT93}. 
\end{remark}

\begin{lemma}[Bounded distortions  for  $u$-bridges]\label{lem.bdp2}
For any integer $m\geq 3$, there exist $\eta(m)>0$ and an integer $\kappa(m)\geq 1$ such that, 
for any $k\geq \kappa(m)$ and $\mu \in (-2-\eta(m),-2+\eta(m))$, the following conditions {\rm (\ref{bdp2_AA})}--{\rm (\ref{bdp2_tilde AA})} hold, 
where $A_k^u$ and $A_{k+1}^u$ are $u$-bridges for $K_m^{u}(\mu)$ of generation 
$k$ and $k+1$ respectively such that the first $k$ entries in the itinerary of $A^{u}_{k+1}$ are the same as the entries in the itinerary of $A^{u}_{k}$, that is,   
$$A^{u}_{k}=A^{u}(k;z_{1}\ldots z_{k}),\   
A^{u}_{k+1}=A^{u}(k+1;z_{1}\ldots z_{k} z_{k+1}).$$
\begin{enumerate}[\rm (1)]
\item\label{bdp2_AA}
If $z_{k+1}=j\in \{1,\dots,m-1\}$, then 
$$
3\cdot 2^{j-2} \leq |A_{k}^{u}| | A_{k+1}^{u}|^{-1} \leq 5\cdot 2^{j-2}.
$$
In particular, we have
$$
\frac{3}{2} \leq |A_{k}^{u}| | A_{k+1}^{u}|^{-1} \leq 5\cdot 2^{m-3}
$$
for any $z_{k+1}\in \{1,\dots,m-1\}$.
Moreover, $|A_{k}^{u}| | A_{k+1}^{u}|^{-1} \leq 5$ 
if $z_{k+1}$ is either $1$ or $2$.
\item\label{bdp2_AI}
Let $I_k^u$ be the minimum sub-interval of $A_k^u$ containing 
$A_{k+1}^u$ and the bottom point of $A_k^{u}$. 
Then
$$|A_{k+1}^u|\, |I_k^u|^{-1}\geq \frac13.$$
\item\label{bdp2_tilde AA}
Suppose that $z_{k+1}\leq m-2$ and $\widetilde A^{u}_{k+1}=A^{u}(k+1;z_{1}\ldots z_{k} \ z_{k+1}+1)$.
Let $G_{k+1}^u$ be the connecting gap for $A_{k+1}^u$ and $\widetilde A_{k+1}^u$, i.e.\ 
$G_{k+1}=\mathrm{Gap}(A_{k+1}^u,\widetilde A_{k+1}^u)$.
Then
$$|\widetilde A_{k+1}^u|\, |A_{k+1}^u|^{-1}\geq \frac13
\quad\text{and}\quad |G_{k+1}^u|\, |A_{k+1}^u|^{-1}\geq \frac1{2^{m+1}}.$$
\end{enumerate}
\end{lemma}
\begin{proof}
(\ref{bdp2_AA})
There are constants $C>0$ and $\sigma>1$ such that $|(F_{\mu}^{k})^{\prime}(x)|\geq C\sigma^{k}$ 
for any $x\in K_m^{u}(\mu)$ and $k\geq 1$.
Let $k_{0}=k_{0}(\mu,m)$ be the minimum non-negative integer with $C\sigma^{k_{0}}>1$. 
As in the proof of Lemma \ref{lem.bdp0}, we have a constant $\gamma>1$ and an open neighborhood $V$ of $K_m^u(\mu)$ such that $|(F_\mu^{k_0})^{\prime}(x)|\geq \gamma$ for any $x\in V$ and $n\geq 1$.

Take a positive integer $\kappa$.
For any integer $k\geq k_0+\kappa$, there exist integers $n\geq 0$ and $h\in\{\kappa,\ldots, \kappa+k_{0}-1\}$ such that 
$F_\mu^{nk_0}(A_k^u)=A_h^u$, $F_\mu^{nk_0}(A_{k+1}^u)=A_{h+1}^u$, where
$$
A^{u}_{h}=A^{u}(h; z_{k-h+1} \ldots z_{k}),\ 
A^{u}_{h+1}=A^{u}(h+1; z_{k-h+1} \ldots z_{k} z_{k+1}).
$$
By the mean-value theorem, there are 
$q\in A^{u}_{k}$ and $\tilde q\in A^{u}_{k+1}$ 
  such that 
\begin{equation}\label{eqn.bdp2-1}
|A^{u}_{h}|=|A^{u}_{k}| | (F^{nk_{0}}_{\mu})^{\prime}(q) |,\quad
|A^{u}_{h+1}|=|A^{u}_{k+1}| | (F^{nk_{0}}_{\mu})^{\prime}(\tilde{q}) |.
\end{equation}
Let $\delta(k)$ be the maximum width of elements in $\mathcal{A}(k)$.
Applying Lemma \ref{lem.bdp0} to $F^{nk_{0}}_{\mu}$, 
there exists a constant $\tilde{c}=\tilde{c}(\delta(\kappa))$ of order $\delta(\kappa)^{\alpha}$   
independent of $n$ such that  
$$
e^{-\tilde{c}}
\leq 
 | (F^{nk_{0}}_{\mu})^{\prime}(\tilde{q}) | | (F^{nk_{0}}_{\mu})^{\prime}(q) |^{-1}
\leq e^{\tilde{c}}.
 $$
 Thus we have 
 $$
e^{-\tilde{c}} r_{u-}(m,j) \leq |A_{k}^{u}| | A_{k+1}^{u}|^{-1}
 \leq e^{\tilde{c}} r_{u+}(m,j),
$$
where
\begin{align*}
r_{u-}(m,j)& =\min \bigl\{ |A^{u}_{h}|\, |A^{u}_{h+1}|^{-1}\ ;\ h\in \{\kappa ,\ldots, \kappa+k_{0}-1\},\\
&\qquad\qquad\qquad
z_{k-h+1},\dots, z_{k}\in\{1,\dots,m-1\},\
z_{k+1}=j\bigr\}; \\
r_{u+}(m,j)& =\max \bigl\{ |A^{u}_{h}|\, |A^{u}_{h+1}|^{-1}\ ;\ h\in \{\kappa,\ldots, \kappa+k_{0}-1\},\\
&\qquad\qquad\qquad
z_{k-h+1},\dots, z_{k}\in\{1,\dots,m-1\},\
z_{k+1}=j\bigr\}.\\
\end{align*}
One can suppose that $\delta(\kappa)$ is arbitrarily small by taking $\kappa$ large enough, and hence 
$\tilde c(\delta(\kappa))$ is arbitrarily close to $0$.

First we consider the case of $\mu=-2$.
Set ${A_h^u}' =g^{-1}(A_h^u)$ for the conjugation map $g$ given in Remark \ref{rmk.hyperbolicity}.
By \cite[\S 6.2]{PT93}, ${A^u}'(0)=[2\delta,1-\delta]$ and ${A^u}'(1;j)=[\frac1{2^j}+\frac1{2^j}\delta,\frac1{2^{j-1}}-\frac1{2^{j-1}}\delta]$ for $j=1,\dots,m-1$. 
${A^u}'(1;2)=[\frac14+\frac14\delta,\frac12-\frac12\delta]$, where $\delta=\frac1{2^m-1}$. 
Thus we have $|{A^u}'(0)|=1-3\delta$ and $|{A^u}'(1;j)|=\frac1{2^j}(1-3\delta)$.
This implies that
$$|{A^u}'(0)|\,|{A^u}'(1;j)|^{-1}=2^{j}.$$
Since the tent map $T$ is a piecewise linear map with $|DT_x|=2$ for any $x\neq \frac12$, 
$|{A^{u}_{h}}'|\, |{A^{u}_{h+1}}'|^{-1}=2^j$ for any $h\in \{\kappa,\ldots, \kappa+k_{0}-1\}$, $z_{k-h+1},\dots, z_{k}\in\{1,\dots,m-1\}$ 
and $z_{k+1}=j$.
The width of the interval ${A^{u}_{h}}'$ can be arbitrarily small if we take $\kappa$ sufficiently large.
Since the conjugation map $g$ is almost affine on such a short interval, one can suppose that 
$r_{u-}(m,j)> \frac34\cdot 2^j=3\cdot 2^{j-2}$ and $r_{u+}(m,j)<\frac54 2^j=5\cdot 2^{j-2}$ for $\mu=-2$.
Since $F_\mu^t$ uniformly $C^1$ converges to $F_{-2}^t$ on $[-3,3]$ as $\mu\to -2$ 
for $t=1,\dots,k_0$, there exist $\eta(m)>0$ and an integer $\kappa(m)\geq 1$ such that $e^{-\tilde c(k)}r_{u-}(m,j)> 3\cdot 2^{j-2}$ and $e^{\tilde c(k)}r_{u+}(m,j)<5\cdot 2^{j-2}$  
if $k\geq \kappa(m)$ and $\mu\in (-2-\eta(m),-2+\eta(m))$.
This shows (\ref{bdp2_AA}).

(\ref{bdp2_AI})
Suppose that $\mu=-2$.
Then ${A^u}'(0)=[2\delta,1-\delta]$ and ${A^u}'(1;l)=\bigl[\frac1{2^l}+\frac1{2^l}\delta,\frac1{2^{l-1}}-\frac1{2^{l-1}}\delta\bigr]$  
for $l=1,\dots,m-1$.
Let ${I_0^u}'$ be the minimum sub-interval of ${A^u}'(0)$ containing ${A^u}'(1;l)$ and $2\delta$, 
that is, ${I_0^u}'=\bigl[2\delta,\frac1{2^{l-1}}-\frac1{2^{l-1}}\delta\bigr]$.
Since $|{A^u}'(1;l)|=\frac1{2^l}(1-3\delta)$ and $|{I_0^u}'|=\bigl(\frac1{2^{l-1}}-\frac1{2^{l-2}}\delta\bigr)-2\delta\leq \frac1{2^{l-1}}(1-3\delta)$, 
$|{A^u}'(1;l)|\,|{I_0^u}'|^{-1}\geq \frac12$.
By using the argument as in (\ref{bdp2_AA}), one can show that 
$$|A_{k+1}^u|\, |I_k^u|^{-1}\geq \frac13$$
if necessary retaking $\kappa(m)$ by a larger integer and $\mu(m)$ by a smaller positive number.
This shows (\ref{bdp2_AI}).

(\ref{bdp2_tilde AA})
The proof is quite similar to that of (\ref{bdp2_AI}).
Since $|{A^u}'(1;l)|=\frac1{2^l}(1-3\delta)$ and $|{A^u}'(1;l+1)|=\frac1{2^{l+1}}(1-3\delta)$ 
for any $l\in \{1,\dots,m-2\}$, we have $|{A^u}'(1;l+1)|\,|{A^u}'(1;l)|^{-1}=\frac12$.
The connecting gap ${G_l^u}'$ for ${A^u}'(1;l)$ and ${A^u}'(1;l+1)$ has the length $|{G_l^u}'|=\frac1{2^l}\delta$.
Thus $|{A^u}'(1;l)|\,|{G_l^u}'|^{-1}\leq (1-3\delta)\delta^{-1}=2^m-4<2^m$.
Then 
one can retake $\kappa(m)$ and $\eta(m)$ again so that 
the inequalities of (\ref{bdp2_tilde AA}) hold.
This completes the proof.
\end{proof}

As $\mu$ is close to $-2$, 
$K^{u}_m(\mu)$ is contained in the interior $\mathrm{Int}(I)$ of $I=[-3,3]$.
In such a situation, 
observe that, for every $p\in K^{u}_m(\mu)$, 
 there exist $k\geq 0$,  
 $A^{u}\in \mathcal{A}_{k}^{u}$ and 
the closure $G$ a component of $I\setminus \bigsqcup_{A^{u}\in \mathcal{A}_{k}^{u}}  A^{u}$ such that 
 $A^{u}$ and $G$ are the bridge and gap satisfying $A^{u}\cap G=\{p\}$. 
 Note that $K^{u}_m(\mu)$  depends on the initially given $m$-periodic orbit containing $q_{1}$ and $q_{2}$. 
This fact together with Remark \ref{rmk.hyperbolicity} implies the following: 
\begin{remark}\label{rmk.large-thickness1}
 $\tau(K^{u}_{m}(\mu))$ can be arbitrarily large if we take $m$ sufficiently large, see  \cite[\S 6.2]{PT93}.
\end{remark}

\subsection{Translation into H\'enon-like maps}\label{subsec.Henon}
H\'enon map introduced in Section \ref {sec.prep} is written as 
$$\varphi_{\mu,\nu}(x,y)=(y,\ \nu x+F_{\mu}(y)),$$ 
where $\nu$ is a real parameter and $F_{\mu}$ is the quadratic map of (\ref{eqn_Fmu}). 
Fortunately all the properties given in Section \ref{subset.one-dim} are inherited by H\'enon maps 
with $\nu\approx 0$.

Using the Cantor set $K^{u}_{m}(\mu)$ for $F_{\mu}$ with $\mu\approx -2$, 
one can define the subset $\{(x, F_{\mu}(x))\ ;\ x\in K^{u}_{m}(\mu)\}$ of $\mathbb{R}^2$, which is a Cantor set on the 
parabolic curve 
$\mathrm{Im}(\varphi_{\mu,0})=\{(x,\mu+x^2); -\infty <x<\infty\}$. 
For simplicity, we denote the Cantor set $\varphi_{\mu,0}(K^{u}_{m}(\mu))$ in  $\mathrm{Im}(\varphi_{\mu,0})$ again by 
$K^{u}_{m}(\mu)$.

Let $\mathcal{U}(-2,0)$ be a small neighborhood of $(-2,0)$ in the parameter space, and let $P_{\mu,\nu}$, $Q_{m;\mu,\nu}$, $\Gamma_{m;\mu,\nu}$ be respectively the fixed point, periodic point and basic set for $\varphi_{\mu,\nu}$ with $(\mu,\nu)\in \mathcal{U}(-2,0)$ defined in Section \ref{sec.prep}.
Now we consider a continuous map $\widetilde{P}:\mathcal{U}(-2,0)\to \mathbb{R}^2$  with $\widetilde P_{\mu,\nu}:=\widetilde{P}(\mu,\nu)\in W^u(P_{\mu,\nu})$ and 
 $\widetilde{P}_{-2,0}=(-2,2)$, see Figure \ref{fig_4_4}. 
For any $(\mu,\nu)\in \mathcal{U}(-2,0)$, let $\widehat\ell_{\mu,\nu}^{u}$ be the arc in 
$W^{u}(P_{\mu, \nu})$ connecting $P_{\mu,\nu}$ with $\widetilde P_{\mu,\nu}$.
By the stable manifold theorem (for example see \cite{Rob99}, 
Chapter 5, Theorem 10.1U)),  
$\widehat\ell_{\mu,\nu}^{u}$ $C^r$-converges to $\widehat\ell_{-2,0}^{u}$ as $(\mu,\nu)\to (-2,0)$, see Figure \ref{fig_4_4}.
\begin{figure}[hbt]
\centering
\scalebox{0.8}{\includegraphics[clip]{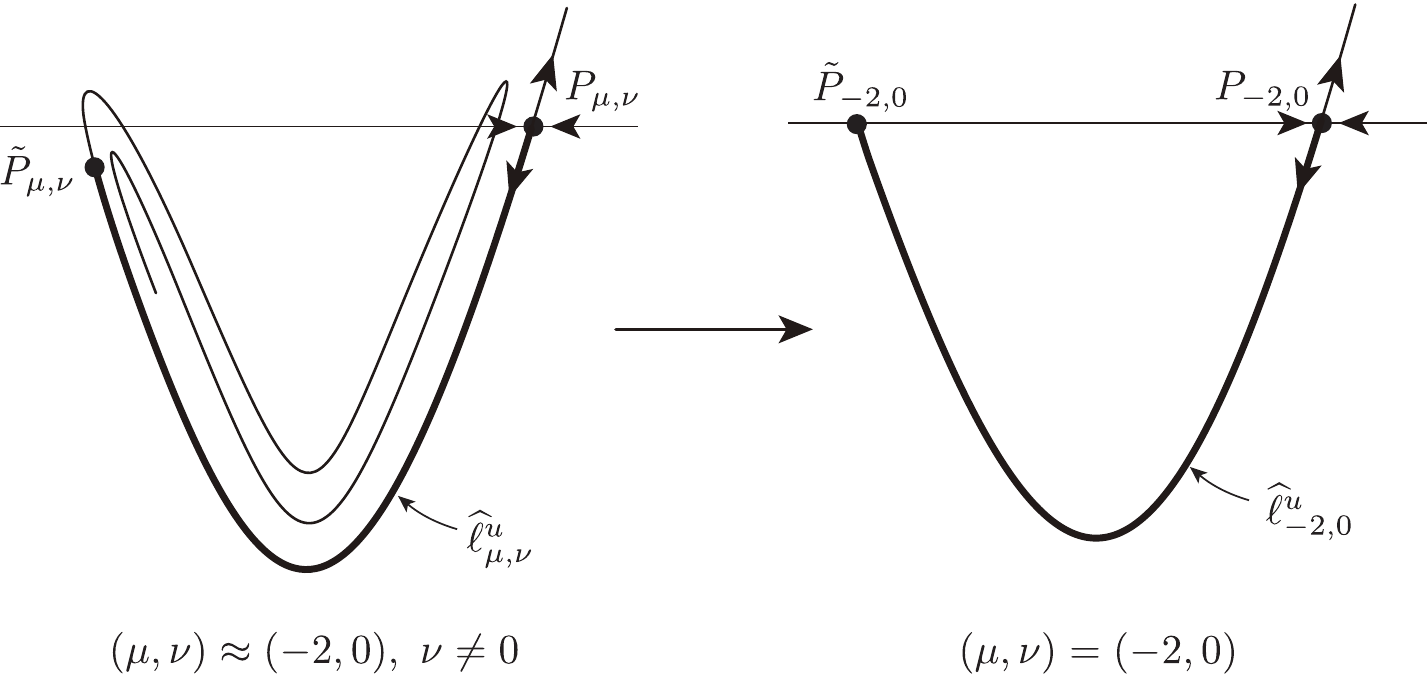}}
\caption{Folding of $\widehat\ell_{\mu,\nu}^u$ as $(\mu,\nu)\to (-2,0)$.}
\label{fig_4_4}
\end{figure}

For any $(\mu,\nu)\approx (-2,0)$ with $\nu\neq 0$, consider arcs $h_i$ $(i=1,2)$ in $W_{\mathrm{loc}}^s(\varphi_{\mu,\nu}^i(Q_{m;\mu,\nu}))$ as illustrated in Figure \ref{fig_4_5}.
\begin{figure}[hbt]
\centering
\scalebox{0.8}{\includegraphics[clip]{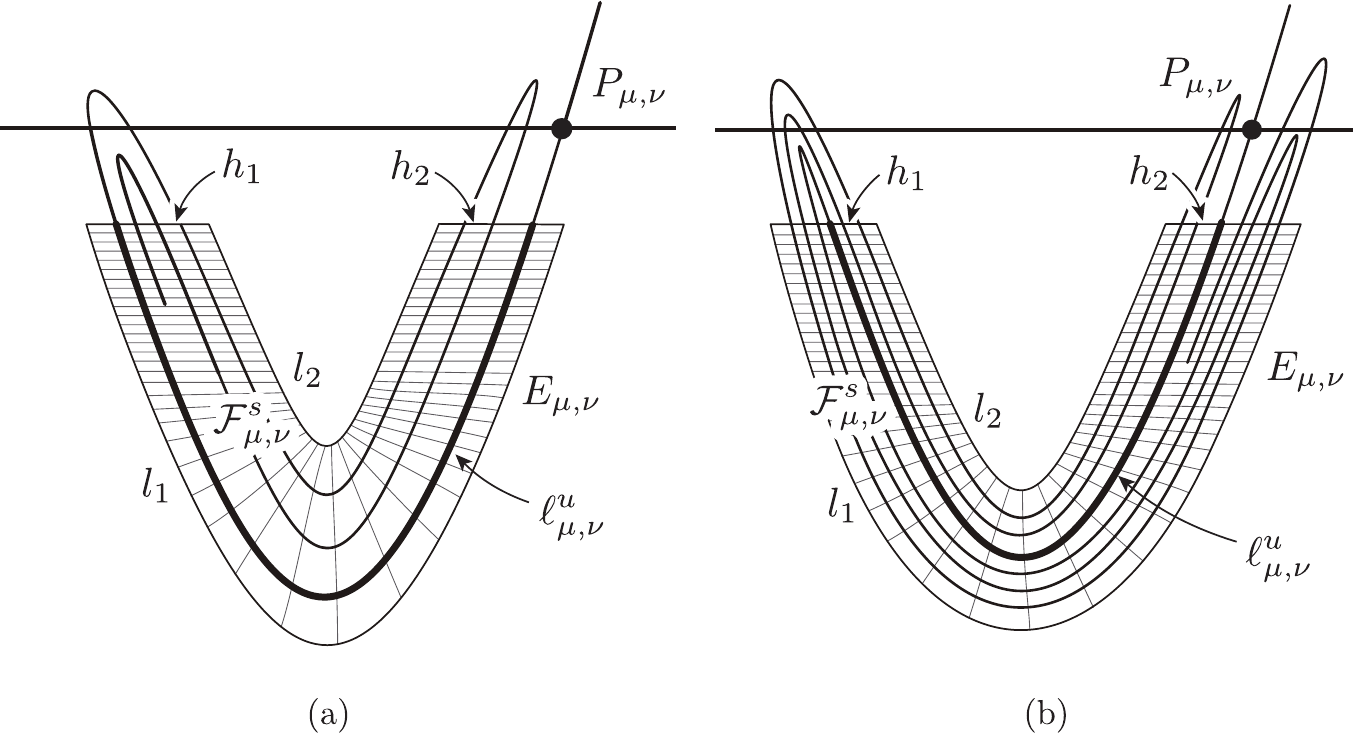}}
\caption{Local stable foliations compatible with $W_{\mathrm{loc}}^s(\Gamma_{m;\mu,\nu})$.
(a) The case of $\nu>0$. (b) The case of $\nu<0$.}
\label{fig_4_5}
\end{figure}
Let $l_j$ $(j=1,2)$ be parabolic curves in $\mathbb{R}^2$ with $\partial l_1\cup \partial l_2=\partial h_1\cup 
\partial h_2$ and such that the union $l_1\cup h_1\cup l_2\cup h_2$ is a simple closed curve in $\mathbb{R}^2$ 
bounding a compact region $E_{\mu,\nu}$ which contains the basic set $\Gamma_{m;\mu,\nu}$ and 
such that $\ell^u_{\mu,\nu}=\widehat \ell^u_{\mu,\nu}\cap E_{\mu,\nu}$ is an arc connecting $h_1$ with $h_2$.  
Consider a local stable foliation $\mathcal{F}^s_{\mu,\nu}=\mathcal{F}^s_{\mathrm{loc}}(\Gamma_{m;\mu,\nu})$  on $E_{\mu,\nu}$ \emph{compatible with} $W^s_{\mathrm{loc}}(\Gamma_{m;\mu,\nu})$,  that is, 
\begin{enumerate}[({F-}i)]
\item\label{F1}
each component of $W^s_{\mathrm{loc}}(\Gamma_{m;\mu,\nu})\cap E_{\mu,\nu}$ is a leaf of $\mathcal{F}^s_{\mu,\nu}$;
\item\label{F2}
$\ell_{\mu,\nu}^{u}$ crosses $\mathcal{F}^s_{\mu,\nu}$ \emph{exactly}, that is, 
 each leaf of $\mathcal{F}^s_{\mu,\nu}$ intersects $\ell_{\mu,\nu}^{u}$ transversely in a single point and any point of $\ell_{\mu,\nu}^{u}$ is passed through by a leaf of $\mathcal{F}^s_{\mu,\nu}$;
\item\label{F3}
leaves of $\mathcal{F}^s_{\mu,\nu}$ are $C^3$ curves such that themselves, their directions, and their curvatures vary $C^1$ with respect to any transverse direction and $(\mu,\nu)$.
\end{enumerate}
See \cite[Lemma 4.1]{KKY} and \cite[\S 2.3]{KLS10} for details.

Here we take the curves $l_1$, $l_2$ so that they are sufficiently $C^r$-close to the outermost components of 
$W^u_{\mathrm{loc}}(\Gamma_{m;\mu,\nu})\cap E_{\mu,\nu}$.
Then there exists 
a local stable foliation $\mathcal{F}^s_{\mu,\nu}=\mathcal{F}^s_{\mathrm{loc}}(\Gamma_{m;\mu,\nu})$  on $E_{\mu,\nu}$ compatible with $W^s_{\mathrm{loc}}(\Gamma_{m;\mu,\nu})$ which contains $h_1$, $h_2$ as 
leaves.

Let $\pi^s_{\mu,\nu}: E_{\mu,\nu}\to \ell_{\mu,\nu}^{u}$ be the 
projection along the leaves of $\mathcal{F}^s_{\mu,\nu}$.
Define
$$
 K^{u}_{m;\mu,\nu}:=\pi^s_{\mu,\nu}(\Gamma_{m;\mu,\nu}),
$$
which is a Cantor set dynamically defined by 
$\pi^s_{\mu,\nu} \circ \varphi_{\mu,\nu}$.
Here we note that the set $K^{u}_{m;\mu,\nu}$ does not depend on the choice of the local stable foliation 
$\mathcal{F}^s_{\mu,\nu}$ on $E_{\mu,\nu}$ compatible with $W_{\mathrm{loc}}^s(\Gamma_{m;\mu,\nu})$.

If $(\mu,\nu)$ is close to $(-2,0)$, then one can define the presentation involved with bridges and gaps for the Cantor set $K^{u}_{m;\mu,\nu}$
in a manner quite similar to that for $K^{u}_{m}$. 
Then Lemma \ref{lem.bdp2} and 
Remark \ref{rmk.large-thickness1} are translated as follows if necessary replacing 
$\mathcal{U}(-2,0)$ by a smaller neighborhood of $(-2,0)$:

\begin{remark}\label{rmk.translation-of-bdp2}

\begin{enumerate}[(1)]
\item
For every $(\mu,\nu)\in \mathcal{U}(-2,0)$, the bounded distortion property in Lemma \ref{lem.bdp2} holds for $K^{u}_{m;\mu,\nu}$. 
\item
The thickness $\tau(K^{u}_{m;\mu,\nu})$ of $K^{u}_{m;\mu,\nu}$ 
converges to $\tau(K^{u}_{m;\mu,0})$ as $\nu\to 0$, and hence it 
can have an arbitrarily large value if we take $m$ large enough and $(\mu,\nu)\in \mathcal{U}(-2,0)$ sufficiently close to $(-2,0)$, see   \cite[\S 6.3, Proposition~1]{PT93}.
\end{enumerate}
\end{remark}

Note that the return map 
$$\varphi_{n}:=\Phi_{n}^{-1}\circ f_{\mu_{n}}^{N_*+n}\circ \Phi_{n}(\bar{x},\bar{y})$$
of Theorem \ref{lem.renormalization} is arbitrarily $C^{r}$-close to $\varphi_{-2,0}$ if 
$n$ is sufficiently large.
Locally identifying the coordinate on $\mathbb{R}^2$ with that on a small neighborhood $U(q)$ of $q$ in $M$, we may set 
\begin{equation}\label{eqn_varphi2}
\varphi_{n}=f_{\mu_{n}}^{N_*+n}.
\end{equation} 
The $\varphi_{n}$ is an H\'enon-like map with  
the saddle fixed point $P(\varphi_{n})$ 
and 
the basic set $\Gamma_{m}=\Gamma_{m}(\varphi_{n})$ 
corresponding to $P_{\mu,\nu}$ and $\Gamma_{m;\mu,\nu}$.
Let $E$ be a compact region corresponding to $E_{\mu,\nu}$ and $\pi_m: E\to \ell^{u}(\varphi_{n})$ the projection along the leaves of a stable foliation $\mathcal{F}^s_{\mathrm{loc}}(\Gamma_{m})$ on $E$  
compatible with $W^s_{\mathrm{loc}}(\Gamma_{m})$,
where $\ell^{u}(\varphi_{n})$ is the curve in $W^{u}(P(\varphi_{n}))\cap E$ corresponding to $l^u(\varphi_{\mu,\nu})$.
Then one obtains the dynamically defined Cantor set 
\begin{equation}\label{def.Ku}
 K^{u}_{m}:=\pi_{m}(\Gamma_{m}). 
\end{equation}

\begin{remark}\label{rmk.large-thickness3}
The distortion and thickness for 
$K^{u}_{m}$ have the same properties  
as those in Remarks \ref{rmk.translation-of-bdp2} 
if $\varphi_{n}$ is sufficiently close to $\varphi_{-2,0}$. 
\end{remark}

\section{Linking property for bridges}\label{sec.Linking}

Recall that $\{f_{\mu}\}_{\mu\in \mathbb{R}}$ is the  
one-parameter family of two-dimensional  $C^r$ diffeomorphisms given in Section \ref{sec.prep}.
In this section, we will present Linking Lemma, which is crucial in the proof of Theorem \ref{main-thm}.

\subsection{Heteroclinic tangencies}\label{subsec.Hetero}

Let $h:[0,1]\times[0,1]\rightarrow M$ be an embedding.
Set $R=h([0,1]\times[0,1])$, $\sharp R=h([0,1]\times\{0,1\})$ and $\flat R=h(\{0,1\}\times [0,1])$.
Note that $\sharp R\cup \flat R=\partial R$.
We say that the pair $(R,\sharp R)$ (for short $R$) is a \emph{strip} with 
the \emph{edge} $\sharp R$.
Similarly, the pair $(R,\flat R)$ is a \emph{strip} with 
the \emph{edge} $\flat R$.
A strip $(R',\sharp R')$ is called a \emph{sub-strip} of $(R,\sharp R)$ if $R'\subset R$ 
and each component of $\sharp R$ contains a component of $\sharp R'$.

One can take a coordinate neighborhood of the fixed point $p$ in $M$ such that 
$p=(0,0)$, $f_0^{-N_*}(q)=(0,1)$ and $f_\mu$ is linear on $S=[0,2]\times [0,2]$ and satisfying the condition (S-\ref{setting3}) in Section \ref{sec.prep}, 
where $N_*$ is the positive integer given in Theorem \ref{lem.renormalization}.
Consider the foliation $\mathcal{F}_S$ on $S$ consisting of horizontal leaves, that is, 
any leaf of $\mathcal{F}_S$ has form $[0,2]\times \{y\}$ for some $0\leq y\leq 2$.
Though $\mathcal{F}_S$ is not in general an $f_\mu$-invariant foliation, the $f_\mu^{-n}$-image of 
the restriction of $\mathcal{F}_S$ on $[0,2\lambda^n]\times [0,2]$ is a sub-lamination of $\mathcal{F}_S$ 
for any positive integer $n$.
Such a partial invariance for $\mathcal{F}_S$ is sufficient in our arguments below.

For any $\bar \mu$ near $-2$, let $\mu_n=\Theta_n(\bar \mu)$ be the parameter value used in 
(\ref{henon-like endo}).
We set $f_{\mu_n}=f_n$ for short throughout the remainder of this subsection.
Take an integer integer $m\geq 3$.
From the condition (S-\ref{setting5}), 
for any sufficiently large integer $n$, we have the 
horseshoe $\Lambda:=\Lambda(f_n)$ containing $p$ for $f_n$ and 
the basic set $\Gamma_{m}:=\Gamma_m(\varphi_n)$ for the return map $\varphi_n:=f^{N_*+n}_{n}$ defined as (\ref{eqn_varphi2}) near the homoclinic tangency $q$. 
Recall that 
$\Gamma_{m}$ contains the saddle fixed point $P:=P(\varphi_n)$ and the $m$-periodic orbit 
$Q_{m}^{(i)}:=\varphi_n^{i}(Q_{m})$ for $i=0,1,\ldots,m-1$.
Let $E$ be the compact region and $\ell^u(\varphi_n)$ the arc given in Subsection \ref{subsec.Henon}.
Again by the condition (S-\ref{setting5}), 
there is an integer $N_1>0$ such that each leaf of 
$f_n^{-N_{1}}(\mathcal{F}_S)$ meets leaves of $W_{\mathrm{loc}}^{u}(\Gamma_{m})$ transversely in $E$.
See Figure \ref{fig_5_1} for the case of $m=4$.
\begin{figure}[hbt]
\centering
\scalebox{0.7}{\includegraphics[clip]{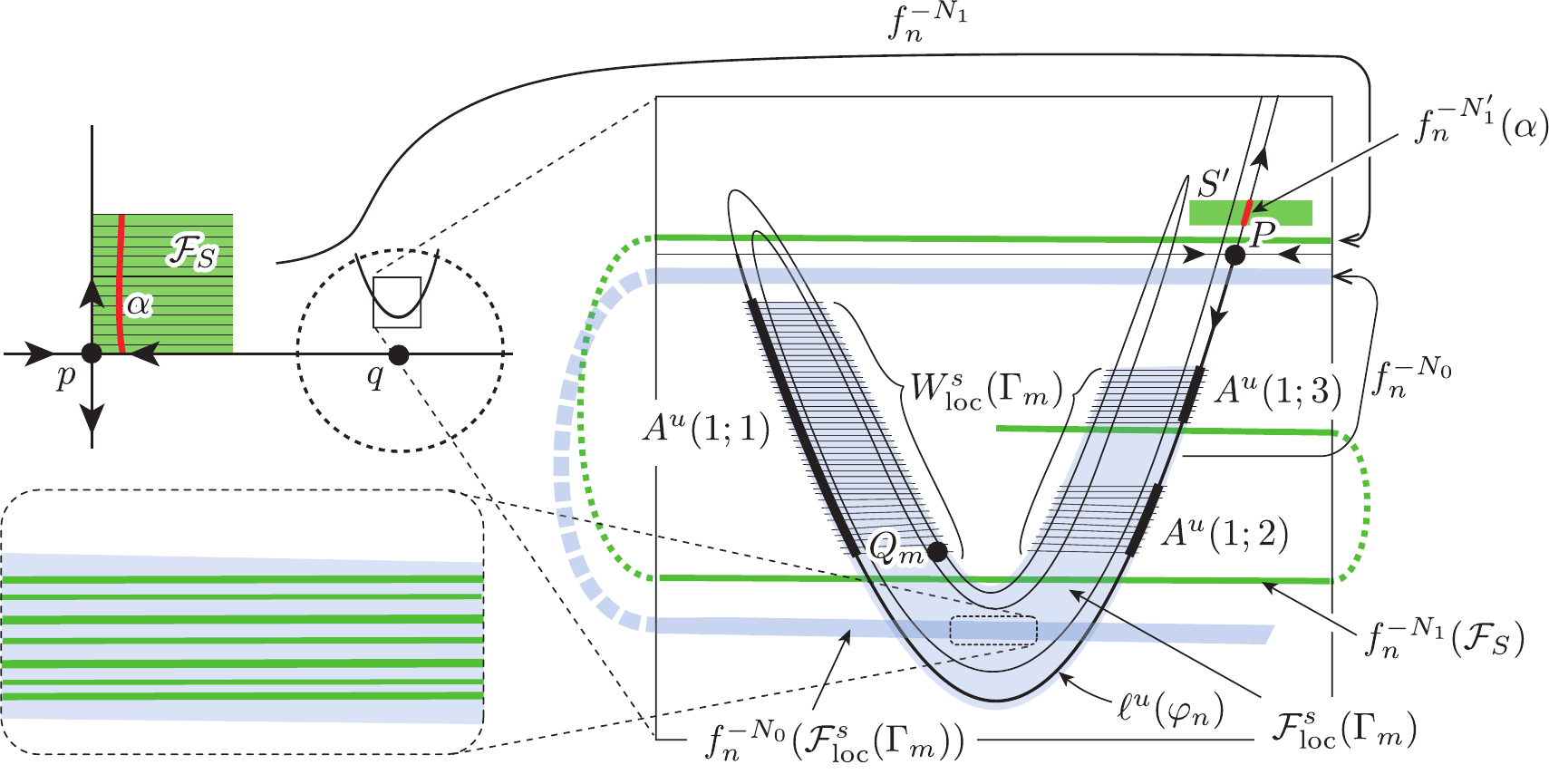}}
\caption{Pull backs of the foliation $\mathcal{F}_S$.}
\label{fig_5_1}
\end{figure}
For any integer $i\geq 0$, let $\mathcal{A}^u(i)$ be the set of $u$-bridges $A^u(i,\underline{z})$ in $\ell^u(\varphi_n)$  of generation $i$ with $A^u(0)=\ell^u(\varphi_n)$ with respect to the Cantor set $K_{m}^u$ of (\ref{def.Ku}), 
and let $\mathcal{A}^u=\bigcup_{i=0}^\infty \mathcal{A}^u(i)$.
Note that the itinerary $\underline{z}=(z_1\dots z_i)$ of $A^u(i,\underline{z})$ is an element of 
$\{1,\dots,m-1\}^i$.
Let $G^u(0)$ be the leading gap of $A^u(0)$.
The closure $\mathbb{G}^u(0)$ of the component of $E\setminus W_{\mathrm{loc}}^{s}(\Gamma_{m})$ 
containing $G^u(0)$ is a strip with $W_{\mathrm{loc}}^s(\Gamma_m)\cap \mathbb{G}^u(0)=\flat\mathbb{G}^u(0)$.
For any $A^u\in \mathcal{A}^u$, the strip $\mathbb{G}(A^u)$ containing the leading gap $G(A^u)$ of $A^u$ is defined similarly, see Figure \ref{fig_5_2}.
\begin{figure}[hbt]
\centering
\scalebox{0.7}{\includegraphics[clip]{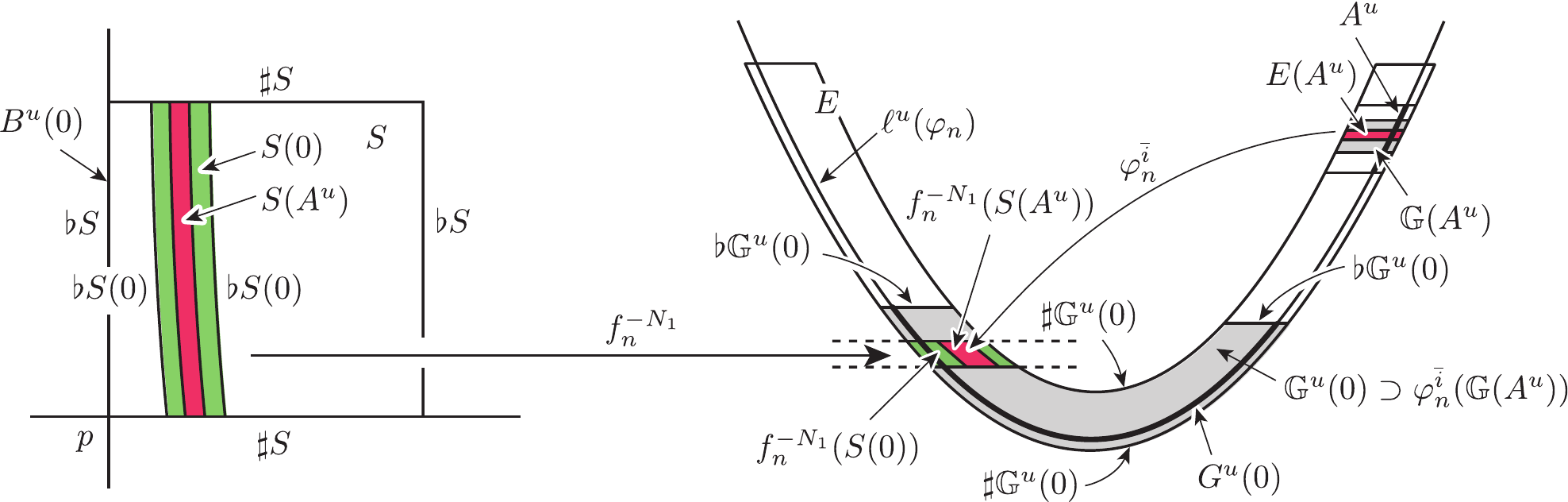}}
\caption{Pull back of the strip $S(A^u)$.}
\label{fig_5_2}
\end{figure}

For the square $S=[0,2]\times [0,2]$, set $\sharp S=[0,2]\times \{0,2\}$ and $\flat S=\{0,2\}\times [0,2]$.
For our choice of $N_1$, there exists a almost vertical sub-strip $(S(0),\sharp S(0))$ of $(S,\sharp S)$ such that $\flat S(0)$ consists of two arcs in $S$ meeting $\mathcal{F}_S$ transversely and 
$(f_n^{-N_1}(S(0)),f_n^{-N_1}(\flat S(0)))$ is a sub-strip of the strip $(\mathbb{G}^u(0),\sharp \mathbb{G}^u(0))$.
We show that one can choose the strip so that
\begin{equation}\label{eqn_S(0)}
S(0)\cap W_{\mathrm{loc}}^u(\Lambda)=\emptyset.
\end{equation}
Since $\Gamma_m$ is homoclinically related to $\Lambda$, $W^u(P)$ contains a almost vertical arc $\alpha$ in 
$S$ connecting the components of $\sharp S$.
If we take an integer $N_1'$ sufficiently large, the arc $f^{-N_1'}(\alpha)$ is contained in 
a small neighborhood of $P$, see Figure \ref{fig_5_1}.
Since $\alpha\cap W_{\mathrm{loc}}^u(\Lambda)=\emptyset$, 
we have a sub-strip $S'$ of $f_n^{-N_1'}(S)$ with $f_n^{-N_1'}(\alpha)$ as a core and such that 
$f_n^{N_1'}(S')\cap W_{\mathrm{loc}}^u(\Lambda)=\emptyset$.
If we take an integer $N_1$ sufficiently larger than $N_1'$, then 
$f_n^{-(N_1-N_1')}(S')$ is a sub-strip of $f_n^{-N_1}(S)$ containing the sub-strip $f^{-N_1}(S(0))$ of $\mathbb{G}^u(0)$ as above.
Then $S(0)$ is contained in $f^{N_1'}(S')$ and hence in particular it satisfies (\ref{eqn_S(0)}).

If the generation of a $u$-bridge $A^u$ is $i$, then there exists a unique integer $\bar i$ with
\begin{equation}\label{eqn_widehat m}
i\leq \bar i\leq (m-1)i
\end{equation}
such that $(\varphi_n^{\bar i}(\mathbb{G}(A^u)),\varphi_n^{\bar i}(\flat\mathbb{G}(A^u)))$ is a 
sub-strip of $(\mathbb{G}^u(0),\flat\mathbb{G}^u(0))$.
Then 
$$\ell^u(\varphi_n)=\pi^m\circ \varphi_n^{\bar i}(A^u)$$
holds.
There exists a sub-strip $(S(A^u),\sharp S(A^u))$ of $(S(0),\sharp S(0))$ such that 
$f_n^{-N_1}(S(A^u))=\varphi_n^{\bar i}(\mathbb{G}(A^u))\cap f_n^{-N_1}(S(0))$.
Then we have a sub-strip $(E(A^u),\sharp E(A^u))$ of $(\mathbb{G}(A^u), \sharp\mathbb{G}(A^u))$ with
$$(\varphi_n^{\bar i}(E(A^u)),\varphi_n^{\bar i}(\flat E(A^u)))=
(f_n^{-N_1}(S(A^u)),f_n^{-N_1}(\sharp S(A^u))).$$

The strip $E(A^u)$ has the foliation $\mathcal{F}_{A^u}$ induced from $\mathcal{F}_S$ via 
$\varphi_n^{-\bar i}\circ f_n^{-N_1}$.
One can retake the local unstable foliation $\mathcal{F}_{\mathrm{loc}}^s(\Gamma_m)$ on $E$ compatible with 
$W_{\mathrm{loc}}^s(\Gamma_m)$ and extending $\bigcup_{A^u\in \mathcal{A}}\mathcal{F}_{A^u}$.
The reason why we choose the unstable foliation with $\mathcal{F}_{\mathrm{loc}}^s(\Gamma_m)\supset 
\varphi_n^{-\bar i}\circ f_n^{-N_1}(\mathcal{F}_S|_{S(A^u)})$ will be explained in the proof of 
Lemma \ref{lem_RL}.

Now we take the unstable foliation $\mathcal{F}^u_{\mathrm{loc}}(\Lambda)$ compatible with $W_{\mathrm{loc}}^u(\Lambda)$ 
carefully 
as notified in Remark \ref{r_three_foliations}-(2).
The strip $S(0)$ has the foliation $\mathcal{F}_{S(0)}^u$ induced from $\mathcal{F}_{\mathrm{loc}}^u(\Gamma_m)|_{f^{-N_1}(S(0))}$ via $f^{N_1}$.
By (\ref{eqn_S(0)}), $S(0)$ is an almost vertical strip disjoint from $W_{\mathrm{loc}}^u(\Lambda)$.
The foliation $\mathcal{F}^u_{\mathrm{loc}}(\Lambda)$ can be defined so that 
$\mathcal{F}^u_{\mathrm{loc}}(\Lambda)|_{S(0)}=\mathcal{F}_{S(0)}^u$.

From the conditions (S-\ref{setting5}), (S-\ref{setting6}) of Section \ref{sec.prep}, 
there are integers $N_{0}>0$, $N_2>N_*$ and a $C^1$-arc $L$ in $U(q)$ meeting $f_n^{-N_0}(\mathcal{F}_{\mathrm{loc}}^s(\Gamma_m))$ exactly and such that 
$L'=L\cap f_n^{N_2}(\mathcal{F}_{\mathrm{loc}}^u(\Lambda))$ is a sub-arc of $L$ each element of which is a quadratic tangency of leaves of 
$f_n^{-N_0}(\mathcal{F}_{\mathrm{loc}}^s(\Gamma_m))$ and $f_n^{N_2}(\mathcal{F}_{\mathrm{loc}}^u(\Lambda))$.
See Figure \ref{fig_5_3}.
\begin{figure}[hbt]
\centering
\scalebox{0.7}{\includegraphics[clip]{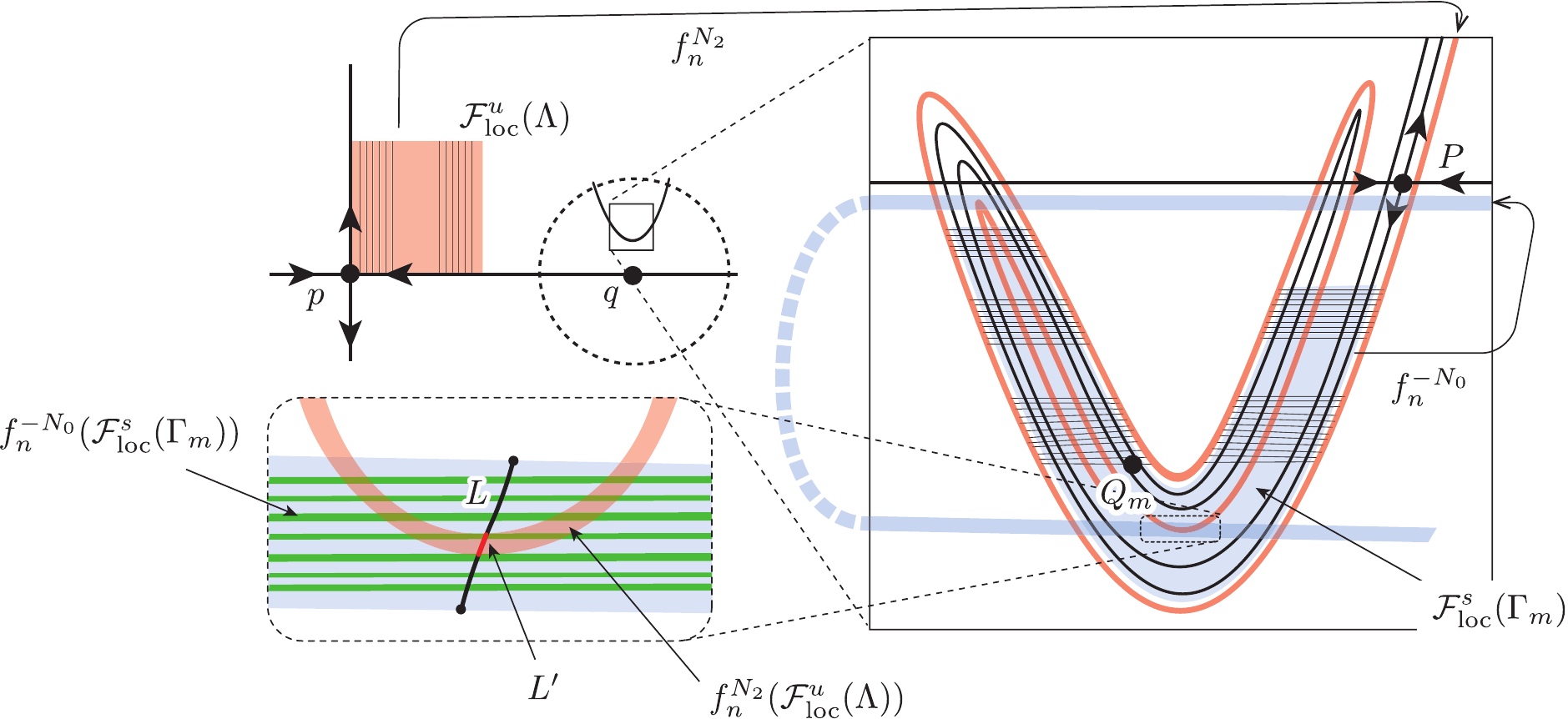}}
\caption{Heteroclinic tangencies on $L'$ of stable and unstable foliations.}
\label{fig_5_3}
\end{figure}

Let
\begin{equation}\label{eqn_pi^u}
\pi^u:E\longrightarrow L
\end{equation}
be the projection along the leaves of  $f_n^{-N_0}(\mathcal{F}_{\mathrm{loc}}^s(\Gamma_m))$.
For the sub-strip $S(A^u)$ of $S(0)$ given as above, consider the projection $\pi_{S,A^u}:S(A^u)\longrightarrow L$ defined 
by $\pi_{S,A^u}=\pi^u\circ \varphi_n^{-\bar i}\circ f_n^{-N_1}|_{S(A^u)}$ 
and the composition
\begin{equation}\label{eqn_pi_A^u}
\pi_{A^u}:=\pi_{S,A^u}\circ \psi_{A^u}:B^u(0)\longrightarrow L,
\end{equation}
where
$\psi_{A^u}:B^u(0)\longrightarrow S(A^u)$ is a diffeomorphism from $B^u(0)=\{0\}\times [0,2]$ onto a component of 
$\flat S(A^u)$ along the leaves of $\mathcal{F}_S$.
See Figure \ref{fig_5_2} and Figure \ref{fig_7_2} with $\widehat A_k^u$, $f^{-\widehat n_k}$, $f^{-N_1}$ 
replaced by $A^u$, $\varphi_n^{-\bar i}$, $f_n^{-N_1}$ respectively.
Let $K_{\Lambda,L}^s$, $K_{m,L}^u$ be the Cantor sets on $L$ defined as
\begin{equation}\label{def.Cantor sets}
K_{\Lambda,L}^s=\pi^s(K_\Lambda^s)\quad\mbox{and}\quad K_{m,L}^u=\pi^u(K_{m}^u),
\end{equation}
where
\begin{equation}\label{eqn_pi^s}
\pi^s:S\to L
\end{equation} is the projection along the leaves of $f^{N_2}(\mathcal{F}_{\mathrm{loc}}^u(\Lambda))$.
For any $B^s\in \mathcal{B}^s$, let $B_L^s:=\pi^s(B^s)$ is a bridge of $K_{\Lambda,L}^s$.
Similarly, for any $A^u\in \mathcal{A}^u$, let $A_L^u:=\pi^s(A^u)$ is a bridge of $K_{m,L}^s$.

\subsection{Encounter of $s$-bridges and $u$-bridges, I}\label{subsec.LP}
Here we will study the heteroclinical connection between $s$-bridges $B^s$ of $K_{\Lambda,L}^s$ and $u$-bridges $A^u$ of $K_{m,L}^u$ in $L$.

Let $\mathcal{U}(\varphi_{-2,0})$ be a sufficiently small $C^r$-neighborhood of $\varphi_{-2,0}$.
For any sufficiently large $n$, the return map $\varphi:=\varphi_{n}$ of (\ref{eqn_varphi2}) 
 is contained in $\mathcal{U}(\varphi_{-2,0})$.
Theorem \ref{lem.renormalization} together with Remark \ref{rmk.large-thickness3} assures that, if we take 
the integer $m$ sufficiently large, 
then the Cantor set $K_{m}^u$ in $\ell^u(\varphi)$ with respect to 
\begin{equation}\label{eqn_varphi3}
\varphi=\pi_m\circ f_{\mu_{n_*}}^{N_*+n_*}
\end{equation}
satisfies
\begin{equation} \label{cond.Ku}
\tau(K^{u}_{m})> \max\left\{ r_{s+}^{2},\ \tau(K^{s}_{\Lambda})^{-1},\ 3^8\right\}.
\end{equation}
Here $n_*=n_*(m)$ is a positive integer with $\lim_{m\to\infty}n_*(m)=\infty$ such that 
$f_{\mu_{n_*}}$ is arbitrarily $C^r$-close to $f$.
Write $\tau:=\tau(K_{m}^{u})$ for short and let
\begin{equation}\label{def.xi_0}
\xi_{0}:=\left(\frac{1}{r_{s+}}-\frac{1}{2\tau^{1/2}}\right)r_{s-}.
\end{equation}
By (\ref{cond.Ku}) and Lemma \ref{lem.bdp1}, we have $0<\xi_{0}<1$.

By Theorem \ref{lem.renormalization}, one can choose $\mu_{n_*}=\Theta_{n_*}(\bar \mu)\neq 0$ so that the Cantor sets $K_{\Lambda,L}^s$ and $K_{m,L}^u$ on $f_{\mu_{n_*}}$ are linked in $L$. 
We denote the $f_{\mu_{n_*}}$ again by $f$.
Note that the leaves of $f^{-N_0}(\mathcal{F}_{\mathrm{loc}}^s(\Gamma_m))$ are almost horizontal in $U(q)$.
Now we fix a $C^{1+\alpha}$-coordinate on a small neighborhood $U(L)$ of $L$ in $U(q)$ such that 
each horizontal line is a leaf of $f^{-N_0}(\mathcal{F}_{\mathrm{loc}}^s(\Gamma_m))$ and each vertical line 
is a leaf of $f^{-N_0}(\mathcal{F}_{\mathrm{loc}}^u(\Gamma_m))$, see Figure \ref{fig_5_4}.
\begin{figure}[hbt]
\centering
\scalebox{0.8}{\includegraphics[clip]{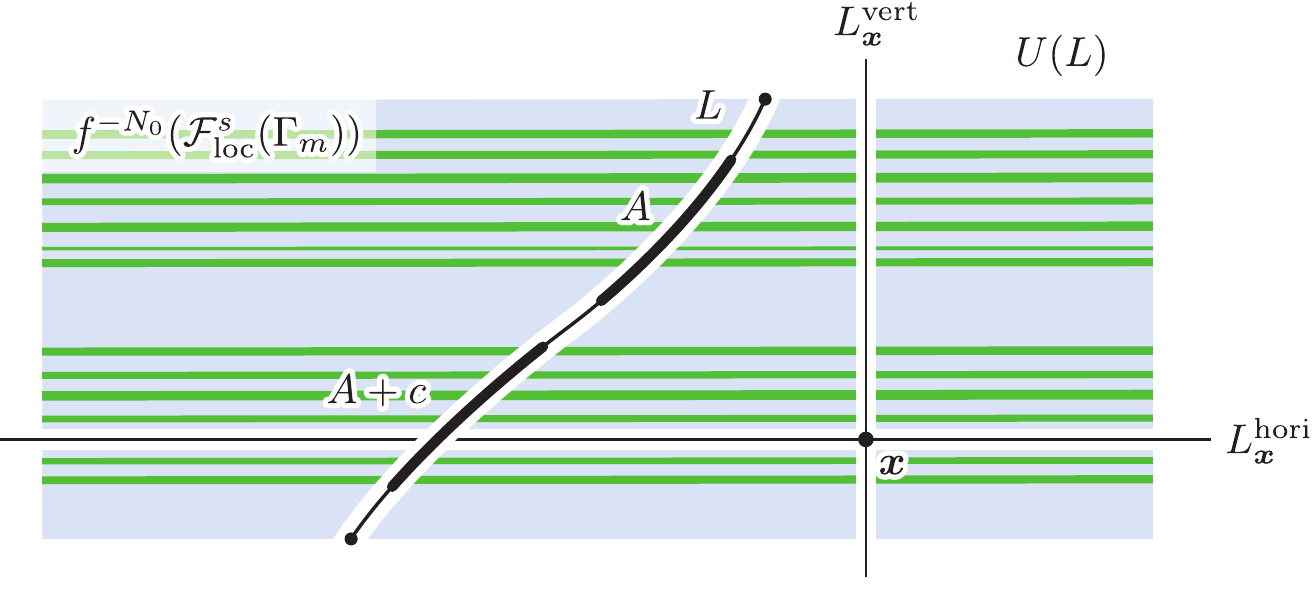}}
\caption{A $C^{1+\alpha}$-coordinate on $U(L)$.}
\label{fig_5_4}
\end{figure}
From the definition of the coordinate and those of $\mathcal{F}_{\mathrm{loc}}^s(\Gamma_m)$ and  $\mathcal{F}_{\mathrm{loc}}^u(\Lambda)$ in Subsection \ref{subsec.Hetero}, for any $\boldsymbol{x}\in U(L)$ with 
$f^{(N_0+N_1)}(\boldsymbol{x})\in S(0)$, the $f^{(N_0+N_1)}$-images of the horizontal and vertical lines 
passing through $\boldsymbol{x}$ are leaves of $\mathcal{F}_S$ and $\mathcal{F}_{\mathrm{loc}}^u(\Lambda)$ 
respectively.
This fact is used in the proof of Lemma \ref{lem_RL}, see also Remark \ref{r_Jk}.

\subsection*{A parametrization of $L$.}
Though $C^1$-arc $L$ has the parametrization naturally induced from the $C^{1+\alpha}$-coordinate on $U(L)$, we need another parameterization on $L$ suitable to our argument.
For the sub-arc $L'=L\cap f^{N_2}(\mathcal{F}_{\mathrm{loc}}^u(\Lambda))$, 
$\widetilde L=f^{-N_2}(L')$ is a $C^1$-arc in $S$ connecting the components of $\flat S$.
See Figures \ref{fig_5_3} and \ref{fig_7_3}.
Since $L$ meets $f^{N_2}(\{0\}\times [0,2])$ transversely, $\widetilde L$ also meets $\{0\}\times [0,2]$ 
transversely.
Thus there exists a small $\delta>0$ such that $\widetilde L_\delta=\widetilde L\cap ([0,\delta]\times [0,2])$ is a sub-arc of $\widetilde L$ meeting vertical lines in $[0,\delta]\times [0,2]$ transversely.
Let $\varpi^u:S\to [0,2]\times \{0\}\subset W_{\mathrm{loc}}^s(p)$ be the vertical projection  
with respect to the orthogonal coordinate on $S$.
We fix $C^1$-parametrizations on $L$ and $\widetilde L$ extending those on 
$\widetilde L_\delta$, $L_\delta:=f^{N_2}(\widetilde L_\delta)$ such that $\varpi^u|_{\widetilde L_\delta}:\widetilde L_\delta \to [0,2]\times \{0\}$ and $f^{N_2}|_{\widetilde L_\delta}:\widetilde L_\delta\to L'$ are parameter-preserving embeddings.
In particular, the parameter value of $\boldsymbol{x}\in L_\delta$ is $t$ if $\varpi^u(f^{-N_2}(\boldsymbol{x}))=(t,0)$.
For any interval $A$ in $L$ represented by $[\alpha_0,\alpha_1]$ with respect to the parametrization and any constant  $c$, the interval $[\alpha_0+c,\alpha_1+c]$ in $L$ is denoted by $A+c$ if it is well defined.
See Figure \ref{fig_5_4}.
Then the length $|A|=\alpha_1-\alpha_0$ of $A$ is equal to the length $|A+c|$ of $A+c$.

\medskip

For given bridges $B_L^{s}$ of $K_{\Lambda,L}^{s}$ and $A_L^{u}$ of $K_{m,L}^{u}$,  
we say that the pair $(B_L^{s}, A_L^{u})$  is \emph{linked} if 
\begin{itemize}
\item $\mathrm{Int}(B_L^{s}\cap A_L^{u})\neq \emptyset$, and 
\item $B_L^{s}$ is not contained in a gap of $A_L^{u}\cap K_{m,L}^{u}$ and 
$A_L^{u}$ is not contained in a gap of $B_L^{s}\cap K_{\Lambda,L}^{s}$.
\end{itemize}
Moreover the pair is called  \emph{$\xi$-linked} for a constant $0<\xi\leq 1$ if 
$$| B_L^{s} \cap A_L^{u} | \geq \xi \min\{ | B_L^{s} |,  | A_L^{u} | \}.$$
The pair $(B_L^{s}, A_L^{u})$ is $\gamma$-\emph{proportional} for a constant $\gamma$ with $0<\gamma<1$
if  
$$| A_L^{u} |\geq | B_L^{s} |\geq \gamma | A_L^{u} |,$$
and the pair is $u$-\emph{dominating} if $|A_L^{u}|\geq |B_L^{s}|$.

\begin{remark}\label{rem_BA}
Let $B^s$, $A^u$ be the bridges of $K_\Lambda^s$ and $K_{m}^u$ respectively.
Since the projections $\pi^s$, $\pi^u$ of (\ref{eqn_pi^s}) and (\ref{eqn_pi^u}) are $C^1$-maps, $\pi^s|_{B^s}:B^s\to B_L^s$ and $\pi^u|_{A^u}:A^u\to A_L^u$ are almost affine if $|B^s|$ and $|A^u|$ are 
sufficiently small.
Thus one can suppose the following conditions without loss of generality.
\begin{enumerate}[(i)]
\item
For bridges $B_{L,k}^s$, $B_{L,k+1}^s$ of bridges of $K_{m,L}^s$ with sufficiently large generations $k$, $k+1$ and $B_{L,k}^s\supset B_{L,k+1}^s$, the conclusion of Lemma \ref{lem.bdp1} holds, 
that is, 
\begin{equation}\label{eqn_bdp1}
r_{s-}\leq |B_{L,k}^s||B_{L,k+1}^s|^{-1}\leq r_{s+}
\end{equation}
if necessary modifying $r_{s-}$ and $r_{s+}$ slightly.
\item
For bridges $A_{L,k}^u$, $A_{L,k+1}^u$ of bridges of $K_{\Lambda,L}^u$ with sufficiently large generations $k$, $k+1$ and $A_{L,k}^u\supset A_{L,k+1}^u$ and the gaps $G_{L,k+1}^u$, the interval $I_k^u$ corresponding to 
those in Lemma \ref{lem.bdp2}, the conclusions (1)--(3) of Lemma \ref{lem.bdp2} hold. 
\item
From the definition of the thickness, the restricted Cantors sets satisfy $\tau(B^s\cap K_\Lambda^s)\geq \tau(K_\Lambda^s)$ and $\tau(A^u\cap K_{m}^u)\geq \tau(K_{m}^u)$.
Thus, for any $0<\varepsilon<1$ and bridges $B^s$, $A^u$ with sufficiently large generation, 
$\tau(B_L^s\cap K_{\Lambda,L}^s)\geq (1-\varepsilon)\tau(B^s\cap K_\Lambda^s)$ and 
$\tau(A_L^u\cap K_{m,L}^u)\geq (1-\varepsilon)\tau(A^u\cap K_{m}^u)$ hold.
From (\ref{cond.Ku}), one can suppose that
\begin{equation}\label{eqn_rem_BA}
\tau(B_L^s\cap K_{\Lambda,L}^s)\tau(A_L^u\cap K_{m,L}^u)>1.
\end{equation}
See Subsections 4.1 and 4.2 of \cite{KS08} for similar arguments.
\end{enumerate}
\end{remark}

\subsection*{The first perturbation of $f$}
Now we consider the perturbation corresponding to the $\Delta$-sliding in \cite[p.\ 1672]{CV01}.
Since, in \cite{CV01}, the sliding is done along the straight segment, 
the perturbed diffeomorphism is still of $C^r$ class.
However, in our case, $\widetilde L$ is guaranteed only to be of class $C^1$.
Hence a perturbed map along $\widetilde L$ would not be a $C^r$-diffeomorphism.
So we need another $C^r$-perturbation which gives an effect similar to the $\Delta$-sliding.

Consider the gap strip $\mathbb{G}_\Lambda^u(0):=\pi_{\mathcal{F}_{\mathrm{loc}}^s(\Lambda)}^{-1}(G^u(0))$, 
for short $\mathbb{G}^u(0)$, in $S$ 
associated to the gap $G^u(0):=\mathrm{Gap}(B^u(1;1),B^u(1;2))$ and the projection
$\pi_{\mathcal{F}_{\mathrm{loc}}^s(\Lambda)}:S\to \{0\}\times [0,2]$ along the leaves of $\mathcal{F}_{\mathrm{loc}}^s(\Lambda)$.
Recall that $L'$ is the sub-arc of $L$ given in Subsection \ref{subsec.Hetero} meeting $f^{N_2}(\mathcal{F}_{\mathrm{loc}}^u(\Lambda))$ exactly (see Figure \ref{fig_5_3}) 
and $\widetilde L=f^{-N_2}(L')$ is a $C^1$-arc in the strip $\mathbb{G}^u(0)$ disjoint from $\sharp\mathbb{G}^u(0)$ and meeting $\mathcal{F}_{\mathrm{loc}}^u(\Lambda)$ exactly.
See the left-hand panel of Figure \ref{fig_5_5} and Figure \ref{fig_7_3}.
For a sufficiently small $\delta_0>0$, let $\mathcal{B}_{\delta_0/2}$ and $\mathcal{B}_{\delta_0}$ be the 
disks in $M$ centered at the left edge of $\widetilde L$ of radius $\delta_0/2$ and $\delta_0$ respectively 
such that 
$\mathcal{B}_{\delta_0}\cap S\subset \mathbb{G}^u(0)$ 
and $f^{-1}(\mathcal{B}_{\delta_0})\cap (\mathcal{B}_{\delta_0}\cup X_{n_*})=\emptyset$, where 
$X_{n_*}$ is the union of rectangles used in Section \ref{sec.prep} to define the basic set $\Gamma_m$ satisfying the conditions (S-\ref{setting6}) and (S-\ref{setting7}).
See Figure \ref{fig_3_2}.
From the disjointness, any perturbation of $f$ supported on $f^{-1}(\mathcal{B}_{\delta_0})$ dose not 
affect the invariant set $\Gamma_m$ and hence the local stable foliation $\mathcal{F}_{\mathrm{loc}}^s(\Gamma_m)$ on $E$.

For any $\delta$ with $|\delta|$ sufficiently smaller than $\delta_0$, 
consider the perturbation of $f$ supported on $f^{-1}(\mathcal{B}_{\delta_0})$ such that the 
restriction of the perturbed map $f_{\delta}$ on $f^{-1}(\mathcal{B}_{\delta_0/2})$ is the horizontal $\delta$-shift.
Strictly, consider a $C^r$-diffeomorphism $h_\delta:M\to M$ which is the identify on $M\setminus \mathcal{B}_{\delta_0}$ and the $(\delta,0)$-shift $\boldsymbol{x}\mapsto \boldsymbol{x}+(\delta,0)$ 
on $\mathcal{B}_{\delta_0/2}$, and define $f_\delta=h_\delta\circ f$.
Then, for any $\boldsymbol{x}\in f^{-1}(\mathcal{B}_{\delta_0/2})$,
$$f_\delta(\boldsymbol{x})=h_\delta\circ f(\boldsymbol{x})=f(\boldsymbol{x})+(\delta,0).$$ 
One can construct the maps $f_{\delta}$ with fixed $\delta_0$ so as to $C^r$-converge to $f$ as $\delta\to 0$.
In particular, any $f_{\delta}$ can be supposed to satisfy the conditions (\ref{cond.Ku}), (\ref{eqn_bdp1}) and (\ref{eqn_rem_BA}).

Note that this perturbation moves the arc $\widetilde L$ of tangencies.
One can choose local unstable and stable foliations $\mathcal{F}_{\mathrm{loc}}^u(\Lambda;\delta)$, 
$\mathcal{F}_{\mathrm{loc}}^s(\Gamma_m;\delta)$ of $f_\delta$ compatible with 
$W_{\mathrm{loc}}^u(\Lambda)$ and $W_{\mathrm{loc}}^s(\Gamma_m)$ so that  
$$
\mathcal{F}_{\mathrm{loc}}^u(\Lambda;\delta)|_{\mathcal{B}_{\delta_0}}
=h_\delta(\mathcal{F}_{\mathrm{loc}}^u(\Lambda)|_{\mathcal{B}_{\delta_0}}),\quad 
f_\delta^{-N_0}(\mathcal{F}_{\mathrm{loc}}^s(\Gamma_m;\delta))|_{U(L)}=f^{-N_0}(\mathcal{F}_{\mathrm{loc}}^s(\Gamma_m))|_{U(L)}.
$$
The latter equality implies 
$$
f_\delta^{-(N_0+N_2)}(\mathcal{F}_{\mathrm{loc}}^s(\Gamma_m;\delta))|_{\mathcal{B}_{\delta_0}}=f^{-(N_0+N_2)}(\mathcal{F}_{\mathrm{loc}}^s(\Gamma_m))|_{\mathcal{B}_{\delta_0}}.$$
However  
$
f_\delta^{-(N_0+N_2)}(\mathcal{F}_{\mathrm{loc}}^s(\Gamma_m;\delta))|_{f_\delta^{-1}(\mathcal{B}_{\delta_0})}$ 
is not equal to $f^{-(N_0+N_2)}(\mathcal{F}_{\mathrm{loc}}^s(\Gamma_m))|_{f^{-1}(\mathcal{B}_{\delta_0})}$.
Let $\widetilde L(\delta)$ be the arc consisting of tangencies between $\mathcal{F}_{\mathrm{loc}}^u(\Lambda;\delta)$ and 
$f_\delta^{-(N_0+N_2)}(\mathcal{F}_{\mathrm{loc}}^s(\Gamma_m;\delta))$.
See Figure \ref{fig_5_5}.
\begin{figure}[hbt]
\centering
\scalebox{0.7}{\includegraphics[clip]{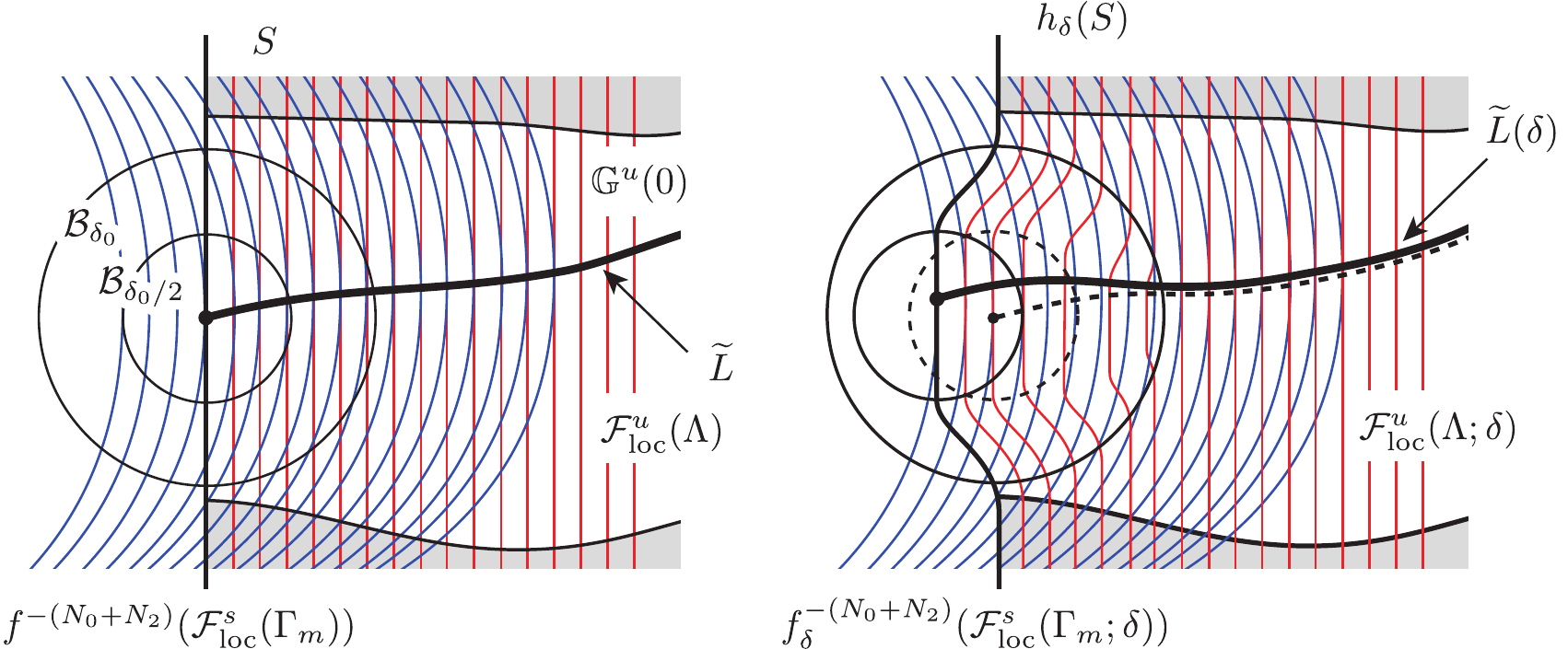}}
\caption{Shifting of $\mathcal{F}_{\mathrm{loc}}(\Lambda)$.
The case of $\delta<0$.}
\label{fig_5_5}
\end{figure}
From the Implicit Function Theorem, the arcs $\widetilde L(\delta)$ $C^1$-depend on $\delta$ 
and $\widetilde L(0)=\widetilde L$.
Let $B^s(\delta)$ be the bridge and $K_{\Lambda}^s(\delta)$ the Cantor set in $\widetilde L(\delta)$ which are 
projected respectively onto the bridge $B^s$ and the Cantor $K_\Lambda^s$ in $W_{\mathrm{loc}}^s(p)$ along the leaves of 
$\mathcal{F}_{\mathrm{loc}}^u(\Lambda;\delta)$.
Let $L(\delta)$ be an arc in $U(L)$ containing $f_\delta^{N_2}(\widetilde L(\delta))$ and crossing  $f_\delta^{-N_0}(\mathcal{F}_{\mathrm{loc}}^s(\Gamma_m;\delta))$ exactly, and let 
$\pi_\delta^u:E\to L(\delta)$ be the projection along the leaves of $f_\delta^{-N_0}(\mathcal{F}_{\mathrm{loc}}^s(\Gamma_m;\delta))$ such that $\pi_0^u$ is equal to $\pi^u$ of (\ref{eqn_pi^u}).
We also consider the bridge $A^u(\delta)$ and the Cantor set $K_{m}^u(\delta)$ in $\widetilde L(\delta)$ 
with $f_\delta^{N_2}(A^u(\delta))=\pi_\delta^u(A^u)=:A_L^u(\delta)$ and $f_\delta^{N_2}(K_{m}^u(\delta))=\pi_\delta^u(K_{m}^u)=: K_{m,L}^u(\delta)$. 
Then the situation is illustrated as follows.
$$
\begin{array}{ccccc}
\widetilde L(\delta)&\xrightarrow{\quad f_\delta^{N_2}|_{\widetilde{L}(\delta)}\quad} & L(\delta)&\xleftarrow{\qquad \pi_\delta^u\qquad }& E\\[6pt]
\bigcup& &\bigcup & &\bigcup\\[6pt]
A^u(\delta),K_m^u(\delta) &\xrightarrow{\qquad\quad} & A_L^u(\delta),K_{m,L}^u(\delta) &\xleftarrow{\qquad\quad} &  A^u,K_m^u
\end{array}
$$

\medskip

Let $\widetilde\pi^s_{\delta}:\widetilde L\to \widetilde L(\delta)$ be the composition of 
the $\delta$-shift map $\boldsymbol{x}\mapsto \boldsymbol{x}+(\delta,0)$ followed by the projection along the leaves of $\mathcal{F}_{\mathrm{loc}}^u(\Lambda;\delta)$, and let 
$\widetilde\pi^u_{\delta}:\widetilde L\to \widetilde L(\delta)$ be the projection along the leaves of $f_\delta^{-(N_0+N_2)}(\mathcal{F}_{\mathrm{loc}}^s(\Gamma_m;\delta))$.
From our construction of $f_\delta$, $B^s(\delta)=\widetilde\pi^s_{\delta}(B^s(0))$ for any $s$-bridge $B^s(0)$ in $\widetilde L\cap \mathcal{B}_{\delta_0/2}$.
Similarly, $\widetilde\pi^u_{\delta}(A^u(0))=A^u(\delta)$ for any $u$-bridge $A^u(0)$ in $\widetilde L\cap \mathcal{B}_{\delta_0/2}$.
Strictly, $\widetilde\pi^u_\delta(\widetilde L)$ is not contained in $\widetilde L(\delta)$ when $\delta>0$.
But it is not a crucial problem since, in our arguments below, $\delta=\Delta_k, \Delta$ can be taken sufficiently small so that the $\widetilde\pi^u_\delta$-images of any bridges in $\widetilde L$ 
used later 
are contained in $\widetilde L(\delta)$.

Recall that $\varpi^u:S\to [0,2]$ with the identification $[0,2]=[0,2]\times \{0\}$ is the vertical  projection used to parametrize $\widetilde L$ and $L$.
For $\boldsymbol{x}, \boldsymbol{x}', \boldsymbol{y}, \boldsymbol{y}'$ in $\widetilde L$, let 
$\boldsymbol{x}_{\delta}, \boldsymbol{x}_{\delta}'\in \widetilde L(\delta)$ be the $\widetilde\pi_{\delta}^s$-images of $\boldsymbol{x}$, $\boldsymbol{x}'$, while let
$\boldsymbol{y}_{\delta}, \boldsymbol{y}_{\delta}'\in \widetilde L(\delta)$ be the $\widetilde\pi_{\delta}^u$-images of  $\boldsymbol{y}, \boldsymbol{y}'$.
Since $\widetilde L(\delta)$ $C^1$-converges to $\widetilde L$ as $\delta\to 0$, it follows from the 
property (F-\ref{F3}) of compatible local foliations 
in Subsection \ref{subsec.Henon} 
that there exists a constant $C>0$ independent of $\delta$ and satisfying
\begin{equation}\label{eqn_1-Cdelta}
\begin{split}
&(1-C|\delta|)|\varpi^u(\boldsymbol{x})-\varpi^u(\boldsymbol{x}')| 
\leq |\varpi^u(\boldsymbol{x}_{\delta})- \varpi^u(\boldsymbol{x}_{\delta}')|
\leq (1+C|\delta|)|\varpi^u(\boldsymbol{x})-\varpi^u(\boldsymbol{x}')|,\\
&(1-C|\delta|)|\varpi^u(\boldsymbol{y})-\varpi^u(\boldsymbol{y}')| 
\leq |\varpi^u(\boldsymbol{y}_{\delta})- \varpi^u(\boldsymbol{y}_{\delta}')|
\leq (1+C|\delta|)|\varpi^u(\boldsymbol{y})-\varpi^u(\boldsymbol{y}')|,\\
&(1-C|\delta|)|\varpi^u(\boldsymbol{x})+\delta-\varpi^u(\boldsymbol{y})|-O(\delta^2)\\
&\qquad\qquad\qquad
\leq 
|\varpi^u(\boldsymbol{x}_{\delta})- \varpi^u(\boldsymbol{y}_{\delta})|
\leq (1+C|\delta|)|\varpi^u(\boldsymbol{x})+\delta-\varpi^u(\boldsymbol{y})|+O(\delta^2).
\end{split}
\end{equation}
Here we explain the reason why the third inequalities include the terms of $O(\delta^2)$.
Let $l_0$, $l_1$ be the leaves of $\mathcal{F}_{\mathrm{loc}}^u(\Lambda;\delta)$ and 
$f_\delta^{-(N_0+N_2)}(\mathcal{F}_{\mathrm{loc}}^s(\Gamma_m;\delta))$, respectively, passing through 
$\boldsymbol{x}_\delta$.
Let $\boldsymbol{x}_\delta'$ be the intersection point of $\widetilde L$ and $l_1$.
See Figure \ref{fig_5_5a}.
\begin{figure}[hbt]
\centering
\scalebox{0.7}{\includegraphics[clip]{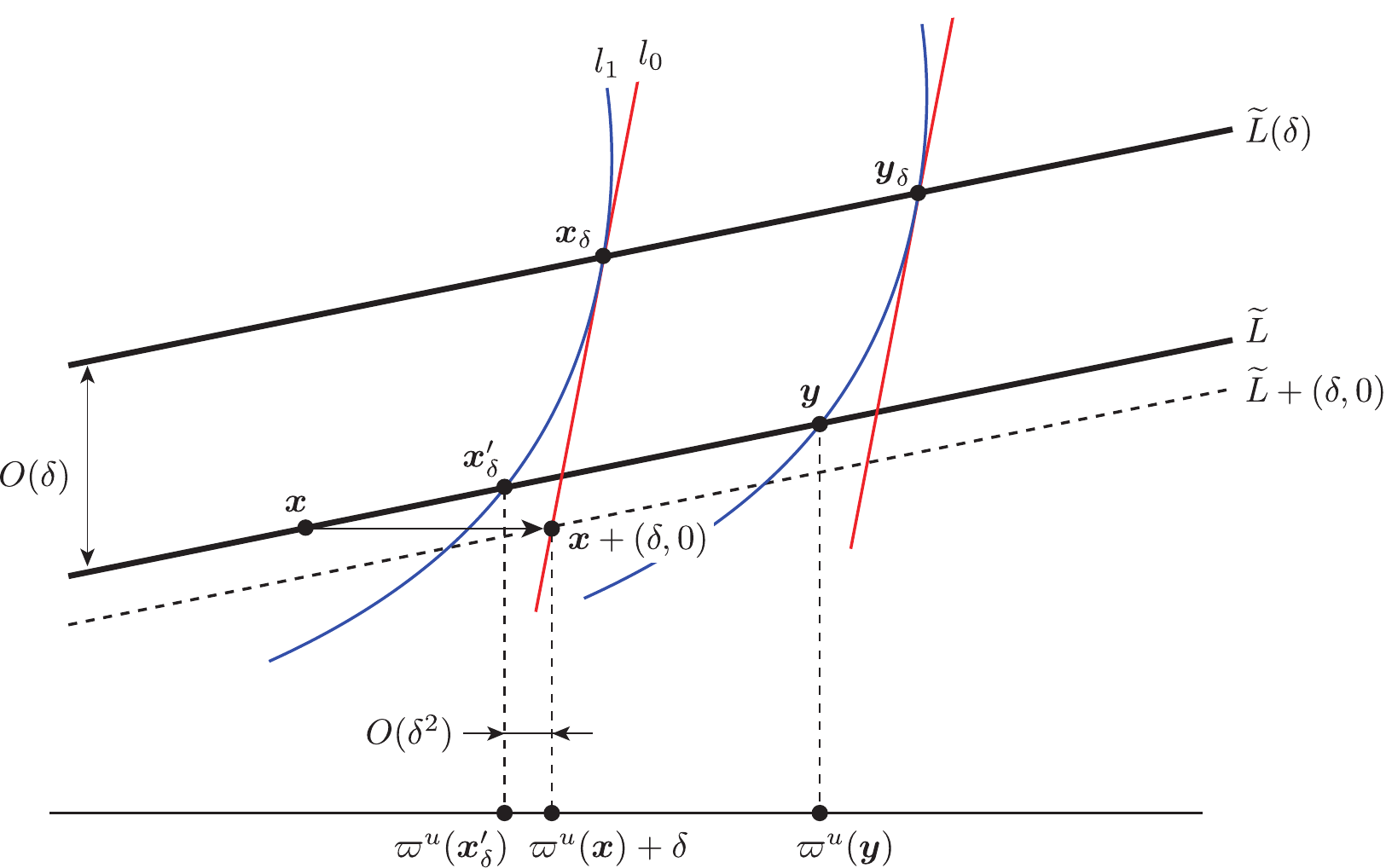}}
\caption{Explanation of the third inequalities of (\ref{eqn_1-Cdelta}).}
\label{fig_5_5a}
\end{figure}
From the second inequalities of (\ref{eqn_1-Cdelta}), 
we have
$$(1-C|\delta|)|\varpi^u(\boldsymbol{x}_\delta')-\varpi^u(\boldsymbol{y})| \leq |
\varpi^u(\boldsymbol{x}_{\delta})- \varpi^u(\boldsymbol{y}_{\delta})|
\leq (1+C|\delta|)|\varpi^u(\boldsymbol{x}_\delta')-\varpi^u(\boldsymbol{y})|.$$
Since $l_0$ and $l_1$ have a quadratic tangency at 
$\boldsymbol{x}_\delta$, we have
$$|\varpi^u(\boldsymbol{x}+(0,\delta))-\varpi^u(\boldsymbol{x}_\delta')|=
|\varpi^u(\boldsymbol{x})+\delta-\varpi^u(\boldsymbol{x}_\delta')|=O(\delta^2).$$
This shows the third inequalities of (\ref{eqn_1-Cdelta}).

\medskip

\begin{lemma}[Linking Lemma]\label{lem_LL1}
Suppose that the dynamically defined Cantor sets $K_{\Lambda,L}^{s}$ and $K_{m,L}^{u}$ satisfy (\ref{cond.Ku}).
Let
$B^{s}(0)$ be an $s$-bridge of $K_{\Lambda}^{s}(0)$ and $A^{u}(0)$ a $u$-bridge of $K_{m}^{u}(0)$ 
contained in $\mathcal{B}_{\delta_0/2}\cap \widetilde L$ 
such that the pair 
$(B^{s}(0), A^{u}(0))$ is linked.
Moreover, suppose that the generations of $B^s(0)$ and $A^u(0)$ are so large that 
the conditions of Lemma \ref{lem.bdp2} and Remark \ref{rem_BA} hold.
Then, for  any $\varepsilon$ with $0<\varepsilon\leq |B^{s}(0)\cap A^{u}(0)|$ and sufficiently smaller than the radius $\delta_0$ of $\mathcal{B}_{\delta_0}$, 
there exist 
bridges 
$B_{1}^{s}$,   $\widetilde{B}_{1}^{s}$, 
$A_{1}^{u}$, $\widetilde{A}_{1}^{u}$
and  
gaps $G^{s}$, $G^{u}$ in $\widetilde L$ with 
$$
B_{1}^{s}, \widetilde{B}_{1}^{s}\subset  B^{s}(0), \ 
G^{s}=\mathrm{Gap}(B_{1}^{s}, \widetilde{B}_{1}^{s}), \ 
A_{1}^{u}, \widetilde{A}_{1}^{u}\subset  A^{u}(0), \ 
G^{u}=\mathrm{Gap}(A_{1}^{u}, \widetilde{A}_{1}^{u})
$$ and satisfying the following properties {\rm (\ref{LL6})--(\ref{LL5})} for any $\nu$ with $|\nu|<\varepsilon$ and {\rm (\ref{LL4})--(\ref{LLGG})} for some $\delta$ with $|\delta|<\varepsilon$:
\begin{enumerate}[\rm(1)]
\item \label{LL6} 
$r_{s+}^{-1}\tau^{-5/4}\varepsilon\leq |B_{1}^{s}(\nu)|
< r_{s-}^{-1}\tau^{-3/4}\varepsilon,\quad 
r_{s+}^{-1}\tau^{-5/4}\varepsilon\leq |\widetilde B_{1}^{s}(\nu)|
< r_{s-}^{-1}\tau^{-3/4}\varepsilon$;
\item   \label{LL5} 
both $( B^{s}_{1}(\nu), A^{u}_{1}(\nu))$ and 
$(  \widetilde B^{s}_{1}(\nu),  \widetilde A^{u}_{1}(\nu))$
are $\tau^{-1}$-proportional; 
\item   \label{LL4} 
both  $(B^{s}_{1}(\delta), A^{u}_{1}(\delta))$  and 
$( \widetilde B^{s}_{1}(\delta), \widetilde A^{u}_{1}(\delta))$
are $\xi_{0}$-linked pairs, where $\xi_{0}$ is the constants given in (\ref{def.xi_0});
\item
\label{LLGG}
$G^u(\delta)$ and $G^s(\delta)$ have a common middle point.
\end{enumerate}
\end{lemma}

Note that $\xi_0$ of (\ref{def.xi_0}) is a universal constant independent of the choices of $B^s$, $A^u$ or $\varepsilon$.
Here we set $B_{1}^{s}$,  $\widetilde{B}_{1}^{s}$, $A_{1}^{u}$, $\widetilde{A}_{1}^{u}$ for 
simplicity instead of 
$B_{1}^{s}(0)$,  $\widetilde{B}_{1}^{s}(0)$, $A_{1}^{u}(0)$, $\widetilde{A}_{1}^{u}(0)$
respectively.

The following lemma is essential in the proof of Lemma \ref{lem_LL1}.

\begin{lemma}\label{lem_LL2}
Let 
$(B^{s}(0),A^{u}(0))$ be the linked pair given in Lemma \ref{lem_LL1}. 
For  any $\varepsilon$ with $0<\varepsilon\leq |B^{s}(0)\cap A^{u}(0)|$ and sufficiently smaller than $\delta_0$, 
there exist an 
interval $J_{1}$ with $J_{1}\subset (-\varepsilon,\varepsilon)$, sub-bridges 
$\widehat{B}^{s}_{1}  \subset  B^{s}(0)$ and 
$\widehat{A}^{u}_{1} \subset  A^{u}(0)$
satisfying the following conditions.
\begin{enumerate}[\rm(1)]
\item  \label{LL2} 
$\tau^{-5/4}\varepsilon\leq |\widehat{B}_{1}^{s}(\nu)|<\tau^{-3/4}\varepsilon$ for any $\nu$ with 
$|\nu|\leq \varepsilon$;
\item  \label{LL3} 
$\tau^{-1/2}|\widehat{A}_{1}^{u}(\nu)|\leq |\widehat{B}_{1}^{s}(\nu)|<\tau^{-1/4}|\widehat{A}_{1}^{u}(\nu)|$ for any $\nu$ with 
$|\nu|\leq \varepsilon$;
\item  \label{LL1}
$
\widehat{B}_{1}^{s}(\nu)\cap \widehat{A}_{1}^{u}(\nu)\neq \emptyset
$ 
if and only if $\nu\in J_{1}$.
\end{enumerate}
\end{lemma}
\begin{proof} 
(\ref{LL2})
First we consider the case of $\nu=0$.
Since the pair $(B^{s}(0), A^{u}(0))$ is linked and $\tau(B^s(0)\cap K_{\Lambda}^{s}(0))\tau(A^u(0)\cap K_{m}^{u}(0))>1$ by (\ref{eqn_rem_BA}), it follows from Gap Lemma that 
$(B^{s}(0)\cap K_{\Lambda}^{s}(0))\cap (A^{u}(0)\cap K_{m}^{u}(0))$ contains a point, say $a_0$.
Take an $s$-bridge $B^s\langle i\rangle$ with $B^s\langle i\rangle \ni a_0$ and $|B^s\langle i\rangle |<\tau^{-3/4}\varepsilon$, where 
$i$ represents the generation of $B^s\langle i\rangle$.
If $|B^s\langle i\rangle |\geq \tau^{-5/3}\varepsilon$, then we set $\widehat B_1^s=B^s\langle i\rangle$.
Otherwise, consider the $s$-bridge $B^s\langle i-1\rangle$ with $B^s\langle i-1\rangle \ni a_0$.
Since $r_{s+}^2<\tau$ by (\ref{cond.Ku}), we have from Lemma \ref{lem.bdp1} that 
$$|B^s\langle i-1\rangle | <r_{s+} |B^{s}\langle i\rangle |< r_{s+}  \tau^{-5/4}\varepsilon<\tau^{-3/4}\varepsilon
\quad\text{and}\quad |B^s\langle i-1\rangle |\geq 2|B^s\langle i\rangle |.$$
If $|B^s\langle i-1\rangle |\geq \tau^{-5/4}\varepsilon$, then we set $\widehat B_1^s=B^s\langle i-1\rangle$.
Otherwise, we repeat the same process until we get the $s$-bridge containing $a_0$ and satisfying the inequality of (\ref{LL2}).
We adopt the bridge as $\widehat B_1^s$.
This shows (\ref{LL2}) for the case of $\nu=0$.
From (\ref{eqn_1-Cdelta}), we know that it is not hard to generalize this result to the case of $|\nu|\leq \varepsilon$.

(\ref{LL3})
First we consider the case of $\nu=0$.
We mean by $A^u\langle i\rangle$ that the generation of the $u$-bridge is $i$. 
Suppose that $A^u=A^u\langle j\rangle$.
First we show that there exists a sub-bridge $A_0^u$ of $A^u$ with 
\begin{equation}\label{eqn_veAve}
\frac{\varepsilon}{3\tau^{1/4}} \leq|A_0^{u}|\leq\frac{\varepsilon}3 
\end{equation}
and contained in a closed sub-arc of $A^u$ of width $\varepsilon/3$ and containing $a_0$.
Here we do not necessarily require that $A_0^u$ contains $a_0$.
Let $A^u\langle j+1\rangle$ be the sub-bridge of $A^u$ containing $a_0$.
If $A^u\langle j+1\rangle\geq \varepsilon/3$, then we repeat the argument using $A^u\langle j+1\rangle$ instead of $A^u$.
So it suffices to consider the case of $A^u\langle j+1\rangle<\varepsilon/3$.
Suppose that $I$ is a sub-arc of $A^u$ with $|I|=\varepsilon/3$ and containing 
$a_0$ as a boundary point.
If $A^u\langle j+1\rangle\geq \varepsilon/3\tau^{1/4}$, then one can set $A_0^u=A^u\langle j+1\rangle$.
Otherwise, consider the maximum sub-arc $I'$ 
of $A^u$ with $a_1$ as a boundary point and containing $I$, where 
$a_1$ is the boundary point of $I$ other than $a_0$, see Figure \ref{fig_5_6}. 
\begin{figure}[hbt]
\centering
\scalebox{0.85}{\includegraphics[clip]{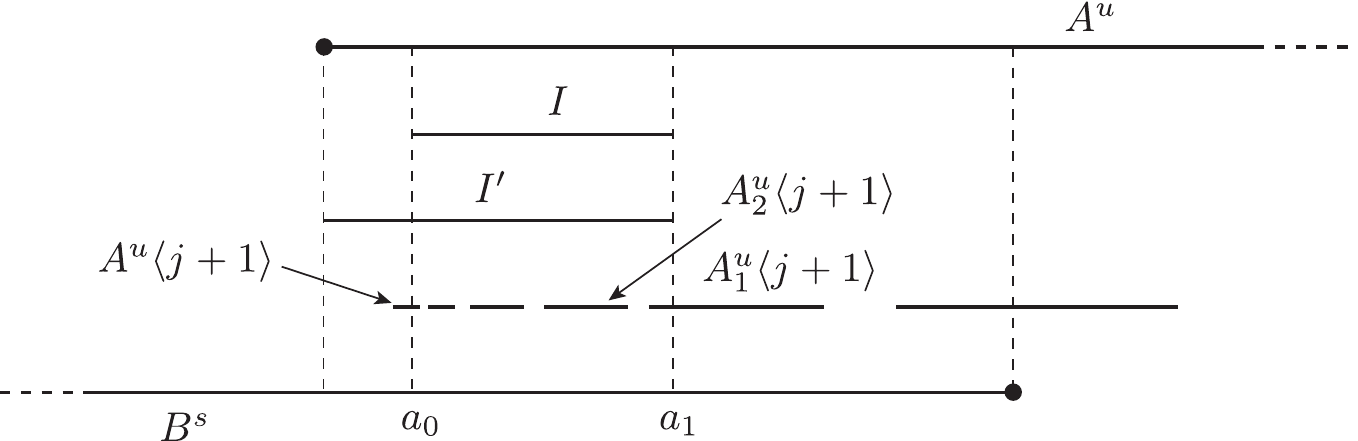}}
\caption{Detecting an unstable bridge of next generation.}
\label{fig_5_6}
\end{figure}
Let $A_1^u\langle j+1\rangle$ be the sub-bridge of $A^u$ closest to $a_0$ among all $u$-bridges not contained in $I'$.
By Lemma \ref{lem.bdp2}-(\ref{bdp2_AI}) (together with Remark \ref{rem_BA}-(ii) for strictly), $|A_1^u\langle j+1\rangle|\geq |I'|/3\geq |I|/3=\varepsilon/9$.
By Lemma \ref{lem.bdp2}-(\ref{bdp2_tilde AA}) and (\ref{cond.Ku}), the $u$-bridge $A_2^u\langle j+1\rangle$ closest to $a_1$  among all $u$-bridges contained in $I'$ satisfies $|A_2^u\langle j+1\rangle|\geq |A_1^u\langle j+1\rangle|/3>\varepsilon/27>\varepsilon/3\tau^{1/4}$.
Thus $A_0^u:=A_2^u\langle j+1\rangle$ satisfies (\ref{eqn_veAve}).

Consider the sequence of $u$-bridges 
$$A_0^u=A^u\langle k\rangle\supset A^u\langle k+1\rangle\supset \cdots \supset A^u\langle k+i\rangle\supset A^u\langle k+i+1\rangle\supset\cdots$$
such that, for any integer $i\geq 0$, the bottom point of $A^u\langle k+i+1\rangle$ is equal to the leading 
point of $A^u\langle k+i\rangle$.
By (\ref{eqn_veAve}), 
$$|A_0^u|\geq \frac{\varepsilon}{3\tau^{1/4}}=\tau^{1/4}\frac{\varepsilon}{3\tau^{1/2}}
>\tau^{1/4}\frac{\varepsilon}{\tau^{3/4}}\geq \tau^{1/4}|\widehat B_1^s|.
$$
Thus there exists $i\geq 0$ such that $|A^u\langle k+i+1\rangle|\leq \tau^{1/4}|\widehat B_1^s|<|A^u\langle k+i\rangle|$.
By Lemma \ref{lem.bdp2}-(\ref{bdp2_AA}), 
$$|\widehat B_1^s|\geq \tau^{-1/4}|A^u\langle k+i+1\rangle|\geq \tau^{-1/4}\frac{|A^u\langle k+i\rangle|}5 >\tau^{-1/2}|A^u\langle k+i\rangle|.
$$
Here we used the inequality $\tau>3^8>5^4$ derived from (\ref{cond.Ku}).
Thus $\widehat A_1^u:=A^u\langle k+i\rangle$ satisfies the inequality of (\ref{LL3}) for $\nu=0$.
Again by using (\ref{eqn_1-Cdelta}), one can generalize this result to the case of $|\nu|\leq \varepsilon$.

(\ref{LL1})
Let $\nu$ be any number with $|\nu|\leq \varepsilon$.
Since $\widehat A_1^u\subset I$, $\widehat A_1^u(\nu)$ is contained in $I(\nu)$.
By (\ref{LL2}), $|\widehat B_1^s(\nu)|<\varepsilon/3$. 
Let $\boldsymbol{x}_B$ be the left edge  of $\widehat B_1^s$ and $\boldsymbol{x}_I$ the right edge  of $I$.
Since $|\widehat B_1^s|<\varepsilon/3$ and $I$ is an arc of length $\varepsilon/3$ with $I\cap \widehat B_1^s\neq \emptyset$, 
$\varpi^u(\boldsymbol{x}_B)-\varpi^u(\boldsymbol{x}_I)>-2\varepsilon/3$.
Let $\boldsymbol{x}_B(\nu)$, $\boldsymbol{x}_I(\nu)$ be the points of $\widetilde L(\nu)$ corresponding to 
$\boldsymbol{x}_B$, $\boldsymbol{x}_I$ respectively.
By (\ref{eqn_1-Cdelta}),
$$\varpi^u(\boldsymbol{x}_B(\nu))-\varpi^u(\boldsymbol{x}_I(\nu))>\nu-\frac{2\varepsilon}3(1+C\nu)-O(\nu^2)$$
for any $0\leq \nu\leq \varepsilon$.
This implies that, if the right hand side of this inequality is positive or equivalently
$$\nu>\frac{2\varepsilon}{3-2\varepsilon C+3O(\nu)}$$
by regarding $O(\nu^2)=\nu O(\nu)$, 
then $\widehat B_1^s(\nu)$ lies in the right component of $\widetilde L(\nu)\setminus I(\nu)$.
One can choose $\varepsilon>0$ so small that the right hand side of the preceding inequality is smaller than $\varepsilon$.
Similarly, if $\nu<-2\varepsilon/(3-2\varepsilon C+3O(\nu))$, then $\widehat B_1^s(\nu)$ lies in the left component of $\widetilde L(\nu)\setminus I(\nu)$.
Thus the interval $J_1$ satisfying the condition (\ref{LL1}) is contained in $(-\varepsilon,\varepsilon)$.
This completes the proof.
\end{proof}

\begin{proof}[Proof of (\ref{LL6}) and (\ref{LL5}) of Lemma \ref{lem_LL1}] 
To show (\ref{LL6}), we will present a procedure how to define our desired sub-bridges and gaps.
Suppose that the generations of $\widehat A_1^u(\nu)$ and $\widehat B_1^s(\nu)$ given in Lemma \ref{lem_LL2} are $k$ and $l$, respectively.
Let $A_{1}^{u}$, $\widetilde{A}_{1}^{u}$ be sub-bridges of $\widehat{A}_{1}^{u}$ of 
generation $k+1$ with the connecting gap $G^{u}$ 
and such that one of $A_{1}^{u}$ and $\widetilde{A}_{1}^{u}$ contains the leading point of $\widehat A_1^u$, 
that is, $G^u$ is the leading gap of $\widehat A_1^u$.
Let $B_{1}^{s}$ and $\widetilde{B}_{1}^{s}$ be sub-bridges of $\widehat{B}_{1}^{s}$ of 
generation $l+1$ with the connecting gap $G^{s}$.
We may assume that $B_1^s$ and $A_1^u$ lie in the left sides of $G_1^s$ and $G_1^u$ respectively if 
necessary exchanging notations.
By Lemmas \ref{lem.bdp1} and \ref{lem_LL2}-(\ref{LL2}), for any $\nu$ with $|\nu|<\varepsilon$, 
$$
r_{s+}^{-1}\tau^{-5/4}\varepsilon\leq |B_{1}^{s}(\nu)|< r_{s-}^{-1}\tau^{-3/4}\varepsilon.
$$
The inequality concerning $|\widetilde B_1^s(\nu)|$ is proved in the same manner.
This shows (\ref{LL6}).

(\ref{LL5})
By Lemmas \ref{lem.bdp1}, \ref{lem.bdp2} and  \ref{lem_LL2}-(\ref{LL3}),  for any $\nu$ with $|\nu|<\varepsilon$,
\begin{align*}
|A_{1}^{u}(\nu)|&\geq 
\frac{|\widehat{A}_{1}^{u}(\nu)|}5\geq
\frac{\tau^{1/4}|\widehat{B}_{1}^{s}(\nu)|}5
\geq
\frac{\tau^{1/4}r_{s-}|B_{1}^{s}(\nu)|}5\geq |B_1^s(\nu)|,\\
|B_{1}^{s}(\nu)|&\geq r_{s+}^{-1}|\widehat B_1^s(\nu)|\geq r_{s+}^{-1}\tau^{-1/2}|\widehat A_1^u(\nu)|
\geq r_{s+}^{-1}\tau^{-1/2}|A_1^u(\nu)|\geq \tau^{-1}|A_1^u(\nu)|.
\end{align*}
This shows that $(B_1^s(\nu),A_1^u(\nu))$ is $\tau^{-1}$-proportional.
The $\tau^{-1}$-proportionality of $(\widetilde B_1^s(\nu),\widetilde A_1^u(\nu))$ 
is shown quite similarly.
This proves (\ref{LL5}).
\end{proof}

We need the following inequality in the proof of (\ref{LL4}):
\begin{equation}\label{eqn_tauGB}
\tau^{-1/2} \geq \frac{|G^{u}(\nu)|}{|\widehat{B}_{1}^{s}(\nu)|}\qquad (|\nu|<\varepsilon).
\end{equation}
In fact, by Lemma \ref{lem_LL2}-(\ref{LL3}), 
$
|\widehat{B}_{1}^{s}(\nu)|\geq \tau^{-1/2}| \widehat{A}_{1}^{u}(\nu)|\geq  \tau^{-1/2}| A_{1}^{u}(\nu)|
$.
From the definition of thickness, 
$\tau \leq |A_{1}^{u}(\nu)|/|G^{u}(\nu)|$. 
It follows that 
$$\tau^{-1/2}| A_{1}^{u}(\nu)|>  \tau^{-1/2}\tau |G^{u}(\nu)|=\tau^{1/2}|G^{u}(\nu)|.$$
Hence (\ref{eqn_tauGB}) holds.

\begin{proof}[Proof of (\ref{LL4}) and (\ref{LLGG}) of Lemma \ref{lem_LL1}]
Since $G^s(\delta)\subset \widehat B_1^s(\delta)$ and $G^u(\delta)\subset \widehat A_1^u(\delta)$, there is a $\delta\in J_1$ such 
that the middle point of $G^s(\delta)$ is equal to that of $G^u(\delta)$.
Again by Lemmas \ref{lem.bdp2}-(\ref{bdp2_AA}) and \ref{lem_LL2}-(\ref{LL3}), 
$|A_1^u(\delta)|\geq |\widehat A_1^u(\delta)|/5\geq \tau^{1/4}|\widehat B_1^s(\delta)|/5>|\widehat B_1^s(\delta)|$.
Similarly $|\widetilde A_1^u(\delta)|\geq  |\widehat B_1^s(\delta)|$.
Thus we have $\mathrm{Int} \widehat A_1^u(\delta)\supset \widehat B_1(\delta)$.
By (\ref{LL5}), $|A_1^u(\delta)|\geq |B_1^s(\delta)|$.
By Lemma \ref{lem.bdp1} and (\ref{eqn_tauGB}), $|B_1^s(\delta)|\geq r_{s+}^{-1}|\widehat B_1^s(\delta)|\geq r_{s+}^{-1}\tau^{1/2}|G^u(\delta)|\geq |G^u(\delta)|$.
This implies that $B_1^s(\delta)$ is not contained in $G^u(\delta)$.

To show that the pair $(B_{1}^{s}(\delta), A_{1}^{u}(\delta))$ is  $\xi_{0}$-linked, 
we need to consider the two 
cases of (a) $G^{s}(\delta)\subsetneq G^{u}(\delta)$ and  
(b) $G^{s}(\delta)\supset G^{u}(\delta)$, see Figure \ref{fig_5_7}.
\begin{figure}[hbt]
\centering
\scalebox{0.8}{\includegraphics[clip]{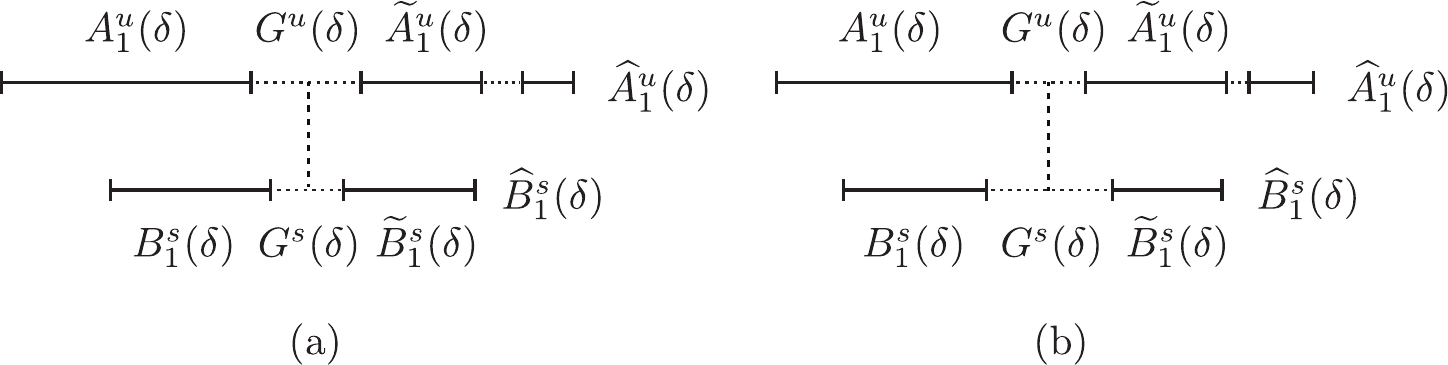}}
\caption{Stable and unstable gaps with the same middle point and bridges near the gaps.}
\label{fig_5_7}
\end{figure}

First we consider the case (a).
One of the boundary points of $B_{1}^{s}(\delta)$ is contained in $A_{1}^{u}(\delta)$ and the other is contained in $G^{u}(\delta)$.
It follows that  
$B_{1}^{s}(\delta)\cap A_{1}^{u}(\delta)\neq \emptyset$, 
$B_{1}^{s}(\delta)$ is not contained in a gap of $A_{1}^{u}(\delta)\cap K_{m}^{u}(\delta)$ and 
$A_{1}^{u}(\delta)$ is not contained in a gap of $B_{1}^{s}(\delta)\cap K_{\Lambda}^{s}(\delta)$. 
This implies that $(B_{1}^{s}(\delta), A_{1}^{u}(\delta))$ is a  linked pair.
By Lemmas \ref{lem_LL2}-(\ref{LL3}) and \ref{lem.bdp1},
$$|B_{1}^{s}(\delta)\cap A_{1}^{u} (\delta)|=|B_{1}^{s}(\delta)|+\frac{|G^{s}(\delta)|}{2}-\frac{|G^{u}(\delta)|}{2}>|B_{1}^{s}(\delta)|-\frac{|G^{u}(\delta)|}{2}\geq
\frac{|\widehat{B}^{s}_{1}(\delta)|}{r_{s+}}-\frac{|G^{u}(\delta)|}{2}.$$
Since $\min\{|B_{1}^{s}(\delta)|, |A_{1}^{u}(\delta)|\}=|B^{s}_{1}(\delta)|$, 
we have by (\ref{def.xi_0}) 
\begin{align*}
\frac{|B_{1}^{s}(\delta)\cap A_{1}^{u}(\delta)|}{\min\{|B_{1}^{s}(\delta)|, |A_{1}^{u}(\delta)|\}}
& =
\frac{|B_{1}^{s}(\delta)\cap A_{1}^{u}(\delta)|}{|B_{1}^{s}(\delta)|}
\geq 
\left(\frac{|\widehat{B}^{s}_{1}(\delta)|}{r_{s+}}-\frac{|G^{u}(\delta)|}{2}\right) \frac{r_{s-}}{|\widehat{B}^{s}_{1}(\delta)|}\\
& \geq
\left(\frac{1}{r_{s+}}-\frac{1}{2\tau^{1/2}}\right)r_{s-} =\xi_{0}.
\end{align*}
\smallskip

In the case (b), 
it is immediately seen that $B_{1}^{s}(\delta) \cap A_{1}^{u}(\delta)\neq \emptyset$ and 
$A_{1}^{u}(\delta)$ is not contained in a gap of $B_{1}^{s}(\delta)\cap K_{\Lambda}^{s}(\delta)$.
Moreover, we will show that   
$B_{1}^{s}(\delta)$ is not contained in a gap of $A_{1}^{u}(\delta)\cap K_{m}^{u}(\delta)$ by contradiction. 
Suppose that 
there would exist a gap $G_{1}^{u}(\delta)$ of $A_{1}^{u}(\delta)\cap K_{m}^{u}(\delta)$ 
with  
$G_{1}^{u}(\delta)\supset B_{1}^{s}(\delta)$. This implies that 
there is a $u$-bridge $A_{*}^{u}(\delta)$ which is adjacent to $G_{1}^{u}(\delta)$ and 
contained in $G^{s}(\delta)$. Thus, 
$$
\frac{|A_{*}^{u}(\delta)|}{|G^{s}(\delta)|}\frac{|B_{1}^{s}(\delta)|}{|G_{1}^{u}(\delta)|}=\frac{|A_{*}^{u}(\delta)|}{|G^{s}(\delta)|}\frac{|B_{1}^{s}(\delta)|}{|G_{1}^{u}(\delta)|}<1.
$$
On the other hand, we have from (\ref{eqn_rem_BA})
$$
\frac{|A_{*}^{u}(\delta)|}{|G_{1}^{u}(\delta)|}\frac{|B_{1}^{s}(\delta)|}{|G^{s}(\delta)|}>\tau(B^s(\delta)\cap K_{\Lambda}^{s}(\delta))\tau(A^u(\delta)\cap K_{m}^{u}(\delta))>1.
$$
This is a contradiction. Hence, we conclude that $B_{1}^{s}(\delta)$ is not contained in a gap of $A_{1}^{u}(\delta)\cap K_{m}^{u}(\delta)$.
Since $B_{1}^{s}(\delta)\subset A^{u}_{1}(\delta)$,  one has 
$$
\frac{|B_{1}^{s}(\delta)\cap A_{1}^{u}(\delta)|}{\min\{|B_{1}^{s}(\delta)|, |A_{1}^{u}(\delta)|\}}= 
\frac{|B_{1}^{s}(\delta)|}{|B_{1}^{s}(\delta)|}=1>\xi_{0}.
$$
This completes the proof of (\ref{LL4}).
\end{proof}

\section{Linear growth property of linked pairs}

As in the preceding section, any $s$-bridge $B^s(\delta)$ here means a bridge with respect to $K_{\Lambda}^s(\delta)$ and 
any $u$-bridge $A^u(\delta)$ means a bridge with respect to $K_{m}^u(\delta)$ for any $\delta$ with 
$|\delta|<\varepsilon$. 
The main result of this section is as follows:

\begin{lemma}[Linear Growth Lemma]\label{lem_LG1}
Let $\xi_{0}$ be the constant of (\ref{def.xi_0}) 
and 
let $(B^{s}(0), A^{u}(0))$ be  a linked pair in $\widetilde L\cap \mathcal{B}_{\delta_0/2}$. 
For any $0<\varepsilon<|B^{s}(0)\cap A^{u}(0)|$, there exist  
a constant $\Delta$ with $|\Delta|<\varepsilon\tau^{-3/4}/2$, 
collections of sub-bridges $\{B^{s}_{k}\}_{k\geq 1}$ of $B^{s}(0)$ and $\{A^{u}_{k}\}_{k\geq 1}$ of $A^{u}(0)$, positive integers $N_{s}$ and $N_{u}$ independent of $\varepsilon$
which satisfy the following {\rm (\ref{LG_1})--(\ref{LG_3})} for every $k\geq 1$.
\begin{enumerate}[\rm (1)]
\item
\label{LG_1} $(B^{s}_{k}(\Delta), A^{u}_{k}(\Delta))$ is a $u$-dominating $\xi_{0}/2$-linked pair.
\item
\label{LG_2}
For the union 
\begin{equation}\label{eqn_Ikf}
I_k=A_k^u(\Delta)\cup B_k^s(\Delta),
\end{equation}
which is an arc in $\widetilde L(\Delta)$, 
there exists a positive constant $\alpha_0$ independent of $k$ such that, for any integer $l>k$,  
the $\alpha_0|A_k^u(\Delta)|$-neighborhood of $I_k$ in $\widetilde L(\Delta)$ is disjoint from 
the $\alpha_0|A_{l}^u(\Delta)|$-neighborhood of $I_{l}$ in $\widetilde L(\Delta)$.
\item
\label{LG_3} If $n_{k}$ and $i_{k}$ are generations of $B^{s}_{k}$ and $A^{u}_{k}$, respectively, then
$$
n_{k}< n_{k+1}\leq n_{k}+N_{s},\quad 
i_{k}< i_{k+1}\leq i_{k}+N_{u}.
$$
\end{enumerate}
\end{lemma}

Here it is crucial that $\Delta$ is an arbitrarily small constant independent of $k$.
Lemma \ref{lem_LG1} follows immediately from the next technical lemma.

\begin{lemma}\label{lem_LG2}
Under the assumptions same as in Lemma \ref{lem_LG1}, there exist sequences 
\begin{itemize}
\item $\{n_{k}\}_{k\geq 1}$, $\{i_{k}\}_{k\geq 1}$ of positive integers; 
\item  $\{ \delta_{k}\}_{k\geq 1}$ of real numbers with  
$$
|\delta_{k}|
\leq 
2^{-1}\xi_{0}\tau^{-3/4}\varepsilon r_{s-}^{-k};
$$
\item $\{B^{s}_{k}\}_{k\geq 1}$, $\{\widetilde{B}^{s}_{k}\}_{k\geq 1}$ of 
$s$-bridges of generation $n_{k}$ with $B^{s}_{k}, \widetilde{B}^{s}_{k}\subset \widetilde{B}^{s}_{k-1}$, $\widetilde B_0^s=B^s(0)$ 
which have the connecting gaps $G_k^s=\mathrm{Gap}(B^{s}_{k}, \widetilde{B}^{s}_{k})$;
\item $\{A^{u}_{k}\}_{k\geq 1}$, $\{\widetilde{A}^{u}_{k}\}_{k\geq 1}$ of 
$u$-bridges of generation $i_{k}$ with $A^{u}_{k}, \widetilde{A}^{u}_{k}\subset \widetilde{A}^{u}_{k-1}$, $\widetilde A_0^u=A^u(0)$ 
which have the connecting gaps $G_k^u=\mathrm{Gap}(A^{u}_{k}, \widetilde{A}^{u}_{k})$;
\end{itemize}
satisfying the following {\rm (\ref{LG1-1})}--{\rm (\ref{LG1-3})} for each $k\geq 1$.
\begin{enumerate}[\rm (1)]
\item \label{LG1-1}
For any $t=1,\ldots, k$ and the positive number $\xi_{k-t}$ defined as (\ref{eqn_xi_k0}), both 
$(B^{s}_{t}(\Delta_{k}), A^{u}_{t}(\Delta_k))$ and 
$(\widetilde B^{s}_{t}(\Delta_{k}), \widetilde A^{u}_{t}(\Delta_k))$ are $u$-dominating $\xi_{k-t}$-linked pairs, 
where $\Delta_k=\delta_{1}+\cdots+\delta_{k}$.
\item
\label{LG1-GG}
$G_k^u(\Delta_k)$ and $G_k^s(\Delta_k)$ have a common middle point.
\item  \label{LG1-3}
There exist integers $1\leq N_{s}$, $N_{u}<\infty $  independent of  $k$ such that 
$$
n_{k}<n_{k+1}\leq n_{k}+N_{s},\quad 
i_{k}< i_{k+1}\leq i_{k}+N_{u}.
$$
Moreover, $\Delta:=\sum_{k=1}^{\infty}  \delta_{k}$ is an absolutely convergent series 
with $\Delta_*:=\sum_{k=1}^\infty|\delta_k|<\varepsilon\tau^{-3/4}/2$. 
\end{enumerate}
\end{lemma}

Let $\{n_k\}_{k=1}^\infty$ be the strictly increasing sequence of generations given in Lemma \ref{lem_LG2} 
and $n_0=0$.
For any integer $k\geq 1$, let
\begin{equation}\label{eqn_xi_k0}
\xi_{k}:=\xi_{0}\Bigl(1-\frac1{2}\sum_{i=1}^{k}r_{s-}^{-\widetilde n_{i}}\Bigr),
\end{equation}
where $\{\widetilde n_{i}\}_{i=1}^\infty$ is the sequence defined by 
$$\widetilde n_i=\inf\{n_{i+l}-n_l\,;\, l=0,1,2,\dots\}.$$
Since $n_{l+1}\geq n_l+1$, we have $\widetilde n_i\geq i$ for any $i\geq 1$.
The inequality $r_{s-}>2$ of Lemma \ref{lem.bdp1} implies
\begin{equation}\label{eqn_xi_k}
\xi_k\geq \xi_{0}\Bigl(1-\frac1{2}\sum_{i=1}^{k}r_{s-}^{-i}\Bigr)
\geq \xi_0\Bigl(1-\frac1{2}\frac1{r_{s-}-1}\Bigr) \geq \xi_0\Bigl(1-\frac12\Bigr)=\frac{\xi_0}2>0.
\end{equation}

\begin{proof}[Proof of Lemma \ref{lem_LG2}]
Applying Lemma \ref{lem_LL1} to the linked pair
$(B^{s}(0), A^{u}(0))$, 
we obtain a 
constant $\delta_{1}$ with $|\delta_1|<\varepsilon$, sub-bridges 
$B^{s}_{1}, \widetilde{B}^{s}_{1}$ of $B^{s}(0)$ and 
$A_{1}^{u},  \widetilde{A}^{u}_{1}$ of $A^{u}(0)$
such that 
$(B^{s}_{1}(\delta_{1}), A^{u}_{1}(\delta_1))$ 
and 
$(\widetilde{B}^{s}_{1}(\delta_1), \widetilde{A}^{u}_{1}(\delta_1))$ are $u$-dominating $\xi_{0}$-linked pairs 
and the connecting gaps $G_1^s(\delta_1)$ and $G_1^u(\delta_1)$ have a common middle point.
Let $a$ be the constant defined as  
\begin{equation}\label{eqn_const_a}
a=\max\left\{2,\ \frac{4\xi_0(1+C\varepsilon)}{r_{s+}(1-r_{s-}^{-1})(1-2r_{s-}^{-1})}
\right\},
\end{equation}
which will be used later to prove Lemma \ref{lem_LG1}-(\ref{LG_2}).
Again applying Lemma \ref{lem_LL1} to the linked pair
$(\widetilde B_1^{s}(\delta_1), \widetilde A_1^{u}(\delta_1))$ for $\varepsilon_1:=\xi_0|\widetilde B_1^s|/2ar_{s+}$ instead of $\varepsilon$,
we obtain a 
constant $\delta_{2}$ with $|\delta_2|\leq \varepsilon_1$, sub-bridges 
$B^{s}_{2}(\delta_1), \widetilde{B}^{s}_{2}(\delta_1)$ of $\widetilde B_1^{s}(\delta_1)$ and 
$A_{2}^{u}(\delta_1),  \widetilde{A}^{u}_{2}(\delta_1)$ of $\widetilde A_1^{u}(\delta_1)$
such that 
$(B^{s}_{2}(\Delta_{2}), A^{u}_{2}(\Delta_2))$ 
and 
$(\widetilde{B}^{s}_{2}(\Delta_2), \widetilde{A}^{u}_{2}(\Delta_2))$ are $u$-dominating $\xi_{0}$-linked pairs 
and the connecting gaps $G_2^s(\Delta_2)$ and $G_2^u(\Delta_2)$ have a common middle point.
Similarly, for any $k\geq 2$, there exist a constant $\delta_{k}$ with $|\delta_k|\leq \varepsilon_{k-1}:=\xi_0|\widetilde B_{k-1}^s|/2ar_{s+}$, sub-bridges 
$B^{s}_{k}(\Delta_{k-1}), \widetilde{B}^{s}_{k}(\Delta_{k-1})$ of $\widetilde B_{k-1}^{s}(\Delta_{k-1})$ and 
$A_{k}^{u}(\Delta_{k-1}),  \widetilde{A}^{u}_{k}(\Delta_{k-1})$ of $\widetilde A_{k-1}^{u}(\Delta_{k-1})$
such that 
$(B^{s}_{k}(\Delta_k), A^{u}_{k}(\Delta_k))$ 
and 
$(\widetilde{B}^{s}_{k}(\Delta_k), \widetilde{A}^{u}_k(\Delta_k))$ are $u$-dominating $\xi_{0}$-linked pairs 
and the connecting gaps $G_k^s(\Delta_k)$ and $G_k^u(\Delta_k)$ have a common middle point.
This shows (\ref{LG1-GG}).

Now we will show that, for any $k\geq 1$, 
$(B^{s}_{t}(\Delta_k), A^{u}_{t}(\Delta_k))$ 
is a $u$-dominating $\xi_{k-t}$-linked pair for each $t=1,\ldots, k$.
Suppose that the assertion holds until the $k$-th step and consider the $(k+1)$-st step.
When $t=k+1$, the proof is already done.
So we may suppose that $t\leq k$.
Since $(B^{s}_{t}(\Delta_k), A^{u}_{t}(\Delta_k))$ is a $u$-dominating 
$\xi_{k-t}$-linked pair, 
$$
|B^{s}_{t}(\Delta_k)\cap A^{u}_{t}(\Delta_k)|
\geq \xi_{k-t} |B^{s}_{t}(\Delta_k)|.
$$
By this inequality together with (\ref{eqn_1-Cdelta}), 
\begin{align*}
|B_t^s(\Delta_{k+1})\cap A_t^u(\Delta_{k+1})|&\geq (1-C|\delta_{k+1}|\,)\bigl(\,|B_t^s(\Delta_{k})\cap A_t^u(\Delta_{k})|-|\delta_{k+1}|\,\bigr)-O(\delta_{k+1}^2)\\
&\geq (1-C|\delta_{k+1}|\,)\bigl(\,\xi_{k-t}|B_t^s(\Delta_{k})|-|\delta_{k+1}|\,\bigr)-O(\delta_{k+1}^2)\\
&\geq (1-C|\delta_{k+1}|\,)\left(\xi_{k-t}\frac{|B_t^s(\Delta_{k+1})|}{1+C|\delta_{k+1}|}-|\delta_{k+1}|\right)-O(\delta_{k+1}^2)\\
&=\xi_{k-t}|B_t^s(\Delta_{k+1})|-\left(1+\frac{2C\xi_{k-t}|B_t^s(\Delta_{k+1})|}{1+C|\delta_{k+1}|}
+O(\delta_{k+1})\right)|\delta_{k+1}|.
\end{align*}
Since $|B_t^s(\Delta_{k+1})|<\varepsilon$ by Lemma \ref{lem_LL1} and $\xi_{k-t}<1$ by (\ref{eqn_xi_k0}), 
one can choose $\varepsilon>0$ so that the contribution of the last parenthesis is smaller than two.
Then 
\begin{align*}
|B^{s}_{t}(\Delta_{k+1})\cap A^{u}_{t}(\Delta_{k+1})|
&\geq \xi_{k-t} |B^{s}_{t}(\Delta_{k+1})|-2|\delta_{k+1}|\\
&=(\xi_{k-t} -2|\delta_{k+1}|\,|B^{s}_{t}(\Delta_{k+1})|^{-1})|B^{s}_{t}(\Delta_{k+1})|.
\end{align*}
Let $n_k$ be the generation of $\widetilde B_k^s$.
By Lemma \ref{lem.bdp1}, $|\widetilde B_k^s(\Delta_{k+1})|\leq r_{s-}^{-(n_k-n_{t-1})}|\widetilde B_{t-1}^s(\Delta_{k+1})|$ 
and $|\widetilde B_{t-1}^s(\Delta_{k+1})|\leq r_{s+}|B_t^s(\Delta_{k+1})|$.
Since $a\geq 2$ by (\ref{eqn_const_a}),
\begin{align*}
2|\delta_{k+1}|\,|B_t^s(\Delta_{k+1})|^{-1}&\leq 2\frac{\xi_0|\widetilde B_k^s(\Delta_{k+1})|}{2ar_{s+}} |B_t^s(\Delta_{k+1})|^{-1}\\
&\leq 
\frac{\xi_0 r_{s-}^{-(n_k-n_{t-1})}|\widetilde B_{t-1}^s(\Delta_{k+1})|}{ar_{s+}}r_{s+}|\widetilde B_{t-1}^s(\Delta_{k+1})|^{-1}\\
&\leq \frac{\xi_0 r_{s-}^{-(n_k-n_{t-1})}}2.
\end{align*}
Then
\begin{align*}
\xi_{k-t} -2|\delta_{k+1}|\,|B^{s}_{t}(\Delta_{k+1})|^{-1}
&\geq \xi_{0}\Bigl(1-\frac1{2}\sum_{i=1}^{k-t}r_{s-}^{-\widetilde n_{i}}\Bigr)-\xi_0 \frac{r_{s-}^{-(n_k-n_{t-1})}}2\\
&= \xi_{0}\Bigl(1-\frac1{2}\Bigl(\sum_{i=1}^{k-t}r_{s-}^{-\widetilde n_{i}}+r_{s-}^{-(n_k-n_{t-1})}\Bigr)\Bigr)\\
&
\geq \xi_{0}\Bigl(1-\frac1{2}\Bigl(\sum_{i=1}^{k-t}r_{s-}^{-\widetilde n_{i}}+r_{s-}^{-\widetilde n_{k+1-t}}\Bigr)\Bigr)=\xi_{k+1-t}.
\end{align*}
Since $\xi_{k+1-t}\geq \xi_0/2$ by (\ref{eqn_xi_k}), 
it follows that $(B^{s}_{t}(\Delta_{k+1}),A^{u}_{t}(\Delta_{k+1}))$ is a $\xi_0$/2-linked pair.
This shows (\ref{LG1-1}).

By Lemma \ref{lem_LL1}-(\ref{LL6}), 
the length of $\widetilde{B}_{k+1}^{s}$ 
 is evaluated as follows:
\begin{equation}\label{eqn.1_in_LGL}
|\widetilde{B}_{k+1}^{s}|
\geq 
r_{s+}^{-1}\tau^{-5/4}\varepsilon_k
\geq
r_{s+}^{-1}\tau^{-5/4}\frac{\xi_{0}|\widetilde{B}^{s}_{k}|}{2r_{s+}}
=2^{-1}r_{s+}^{-2}\tau^{-5/4}\xi_{0}|\widetilde{B}^{s}_{k}|.
\end{equation}
Since the generation of $\widetilde B^{s}_{k}$ is $n_{k}$, by Lemma \ref{lem.bdp1}, 
$$
r_{s-}^{-(n_{k+1}-n_{k})}\geq  |\widetilde B_{k+1}^{s}| |\widetilde B_{k}^{s}|^{-1}.
$$
This implies that
$$n_{k+1}-n_k\leq \dfrac{\log(2^{-1}r_{s+}^{-2}\tau^{-5/4}\xi_{0})}{\log(r_{s-}^{-1})}.$$
Thus, the maximum integer $N_s$ not greater than the right hand side of 
this inequality satisfies $n_{k+1}\leq n_k+N_s$ for any $k\geq 1$.  
It follows from  (\ref{eqn.1_in_LGL})  and the proportionality condition in Lemma 
\ref{lem_LL1}-(\ref{LL5}) that 
\begin{align*}
|\widetilde A_{k+1}^{u}|
&\geq 
|\widetilde B_{k+1}^{s}|
\geq 
2^{-1}r_{s+}^{-2}\tau^{-5/4}\xi_{0}|\widetilde B_{k}^{s}|
\geq 
2^{-1}r_{s+}^{-2}\tau^{-5/4}\xi_{0}\tau^{-1}|\widetilde A_{k}^{u}|\\
&=
2^{-1}r_{s+}^{-2}\tau^{-9/4}\xi_{0}|\widetilde A_{k}^{u}|.
\end{align*}
We suppose that the generation of $\widetilde A^{u}_{k}$ is $i_{k}$. 
Since $|\widetilde A_{k+1}^u| |\widetilde A_{k}^u|^{-1}\leq\bigl(\frac23\bigr)^{i_{k+1}-i_k}$ by Lemma \ref{lem.bdp2}-(\ref{bdp2_AA}), 
one has a positive integer $N_{u}$ independent of $k$ and satisfying 
$$
i_{k+1}-i_{k}
\leq 
N_{u}
\leq 
\frac
{\log(2^{-1}r_{s+}^{-2}\tau^{-9/4}\xi_{0})}
{\log\bigl(\frac2{\,3\,}\bigr)}
$$
for any $k\geq 1$.

Since $a\geq 1$, $n_{k-1}-n_1\geq (k-1)-1=k-2$ and $0<\xi_0<1$, it follows from 
Lemmas \ref{lem.bdp1} and \ref{lem_LL1}-(\ref{LL6}) that  
\begin{align*}
|\delta_k|&\leq \frac{\xi_0|\widetilde B_{k-1}^s(\Delta_{k-1})|}{2ar_{s+}}
\leq \frac{\xi_0}{2r_{s+}} r_{s-}^{-n_{k-1}+n_1}(r_{s-}^{-1}\tau^{-3/4}\varepsilon)
\leq \frac{\xi_0}{2r_{s+}} r_{s-}^{-k+2}(r_{s-}^{-1}\tau^{-3/4}\varepsilon)\\
&=\frac{\xi_0}{2r_{s+}r_{s-}^{-1}}r_{s-}^{-k}\tau^{-3/4}\varepsilon
<\frac{\varepsilon}{2}r_{s-}^{-k}\tau^{-3/4}.
\end{align*}
This shows that
$$
\sum_{k=1}^\infty |\delta_k|<\frac{\varepsilon}{2}\tau^{-3/4}\sum_{k=1}^\infty r_{s-}^{-k}
=\frac{\varepsilon}{2}\tau^{-3/4}\frac{1}{r_{s-}-1}<\frac{\varepsilon}{2}\tau^{-3/4}.
$$
In particular, $\Delta=\sum_{k=1}^\infty \delta_k$ is an absolutely convergent series with  $\Delta_*=\sum_{k=1}^\infty |\delta_k|<\varepsilon\tau^{-3/4}/2$.
This shows (\ref{LG1-3}) and completes the proof.
\end{proof}

The proof of (\ref{LG_1}) and (\ref{LG_3}) of Lemma \ref{lem_LG1} is obtained immediately 
from Lemma \ref{lem_LG2}.
So it remains to prove (\ref{LG_2}).

\begin{proof}[Proof of (\ref{LG_2}) of Lemma \ref{lem_LG1}]
Since $A_k^u(\Delta)\cap B_k^s(\Delta)\neq \emptyset$, $I_k=A_k^u(\Delta)\cup B_k^s(\Delta)$ 
is an arc in $\widetilde L(\Delta)$.
The union $\widehat B_k^s(\Delta)=B_k^s(\Delta)\cup G_k^s(\Delta)\cup \widetilde B_k^s(\Delta)$ is the smallest $s$-bridge 
containing $B_k^s(\Delta)$ and $\widetilde B_k^s(\Delta)$.
By Lemma \ref{lem.bdp1}, $|B_k^s(\Delta)|$, $|\widetilde B_k^s(\Delta)|\leq r_{s-}^{-1}|\widehat B_k^s(\Delta)|$.
It follows that
$$
|G_k^s(\Delta)|=|\widehat B_k^s(\Delta)|-(|B_k^s(\Delta)|+|\widetilde B_k^s(\Delta)|)\geq 
|\widehat B_k^s(\Delta)|(1-2r_{s-}^{-1})\geq |\widetilde B_k^s(\Delta)|(1-2r_{s-}^{-1}).$$
Similarly we have $|G_k^s(\Delta)|\geq |B_k^s(\Delta)|(1-2r_{s-}^{-1})$.
As in the proof of Lemma \ref{lem_LG2}, for any integer $l\geq k+1$,
\begin{align*}
|\delta_l|&<\frac{\xi_0|\widetilde B_{l-1}^s(\Delta_{l-1})|}{2ar_{s+}}\leq 
\frac{(1+C\varepsilon)\xi_0|\widetilde B_{l-1}^s(\Delta)|}{2ar_{s+}}\leq
\frac{(1+C\varepsilon)\xi_0}{2ar_{s+}}r_{s-}^{-n_{l-1}+n_k}|\widetilde B_k^s(\Delta)|\\
&\leq \frac{(1+C\varepsilon)\xi_0}{2ar_{s+}}r_{s-}^{-(l-1-k)}|\widetilde B_k^s(\Delta)|.
\end{align*}
Thus, by (\ref{eqn_const_a}), we have
\begin{equation}\label{eqn_D-D1}
|\Delta-\Delta_k|\leq \sum_{l=k+1}^\infty |\delta_l|\leq \frac{(1+C\varepsilon)\xi_0|\widetilde B_k^s(\Delta)|}{2ar_{s+}}\frac1{1-r_{s-}^{-1}}\leq \frac{|G_k^s(\Delta)|}8.
\end{equation}
Define the constant $\alpha_0$ as
$$\alpha_0=\max\left\{\dfrac{1-2r_{s-}^{-1}}{8\tau},\,\frac1{2^{m+3}}\right\}.$$
Then we have
\begin{equation}\label{eqn_AG1}
\alpha_0|A_k^u(\Delta)|\leq \alpha_0\tau |B_k^s(\Delta)|\leq \alpha_0\frac{\tau}{1-2r_{s-}^{-1}}|G_k^s(\Delta)|=\frac{|G_k^s(\Delta)|}8.
\end{equation}
Since $|A_k^u(\Delta)|\leq 2^{m+1}|G_k^u(\Delta)|$ by Lemma \ref{lem.bdp2}-(\ref{bdp2_tilde AA}),
\begin{equation}\label{eqn_AG2}
\alpha_0|A_k^u(\Delta)|\leq \alpha_0 2^{m+1}|G_k^u(\Delta)|\leq \frac{|G_k^u(\Delta)|}4.
\end{equation}
We have chosen $\Delta_k$ so that $G_s^s(\Delta_k)$ and $G_k^u(\Delta_k)$ have the common middle point $\boldsymbol{x}_k$, see Figure \ref{fig_6_1}.
\begin{figure}[hbt]
\centering
\scalebox{0.8}{\includegraphics[clip]{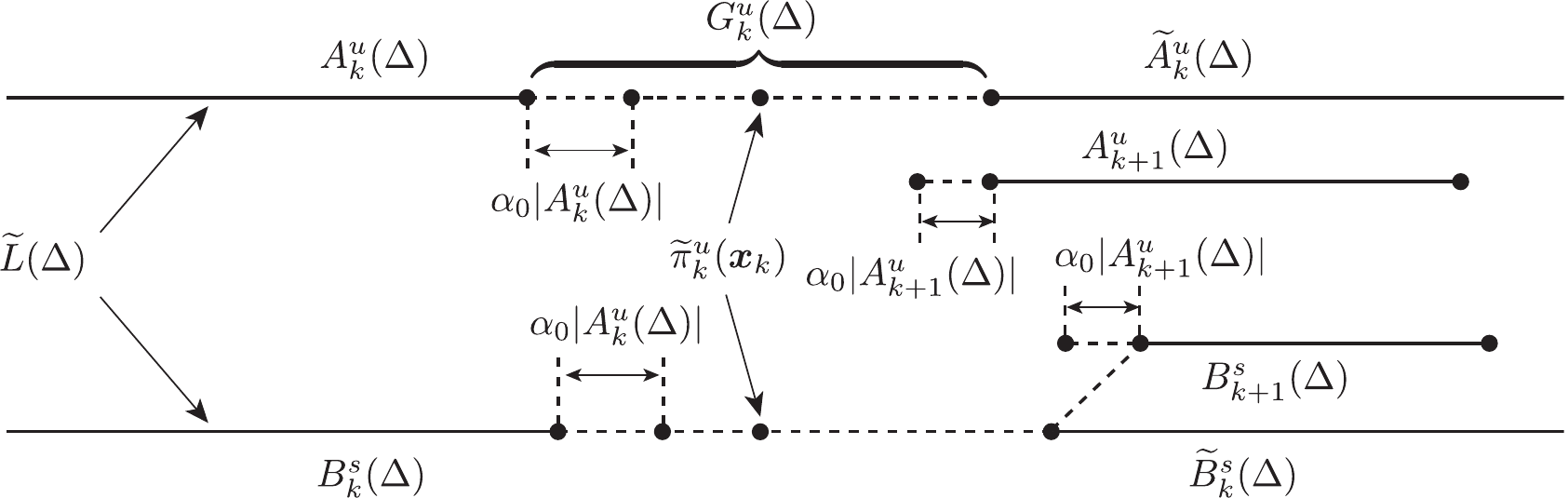}}
\caption{Detecting a sequence of $\xi_0/2$-proportional linked pairs.}
\label{fig_6_1}
\end{figure}
By (\ref{eqn_1-Cdelta}) and (\ref{eqn_D-D1}), for any $\boldsymbol{x}\in B_k^s(\Delta_k)$ and all sufficiently large $k$,
\begin{align*}
|\varpi^u(\widetilde\pi^s_k(\boldsymbol{x}))-\varpi^u(\widetilde\pi^u_k(\boldsymbol{x}_k))|
&\geq (1-C|\Delta-\Delta_k|)|\varpi^u(\boldsymbol{x})-\varpi^u(\boldsymbol{x}_k)|-|\Delta-\Delta_k|-O(|\Delta-\Delta_k|^2)\\
&\geq \frac{(1-C|\Delta-\Delta_k|)|G_k^s(\Delta)|}2-|\Delta-\Delta_k|-O(|\Delta-\Delta_k|^2)\\
&\geq \frac{|G_k^s(\Delta)|}4,
\end{align*}
where 
$\widetilde\pi^s_k:\widetilde L(\Delta_k)\to \widetilde L(\Delta)$ is the composition of 
the shift map $\boldsymbol{x}\mapsto \boldsymbol{x}+(\Delta-\Delta_k,0)$ followed by the projection 
along the leaves of $\mathcal{F}_{\mathrm{loc}}^u(\Lambda;\Delta)$ and  
$\widetilde\pi^u_k:\widetilde L(\Delta_k)\to \widetilde L(\Delta)$ is the projection along the leaves of $f_\Delta^{-(N_0+N_2)}(\mathcal{F}_{\mathrm{loc}}^s(\Gamma_m;\Delta))$.
Thus, by (\ref{eqn_AG1}), $\mathcal{N}(B_k^s(\Delta);\alpha_0|A_k^u|)$ does not contain $\widetilde\pi^u_k(\boldsymbol{x}_k)$, where $\mathcal{N}(J,\eta)$ denotes the $\eta$-neighborhood of $J$ in $\widetilde L(\Delta)$ for $\eta>0$ and a compact subset $J$ of $\widetilde L(\Delta)$.
Since $B_{k+1}(\Delta_k)$ is contained in $\widetilde B_k(\Delta_k)$ and $|A_{k+1}^u(\Delta)|\leq |A_k^u(\Delta)|$, 
one can show similarly that $\mathcal{N}(B_{k+1}^s(\Delta);\alpha_0|A_{k+1}^u(\Delta)|)$ does not contain $\widetilde\pi^u_k(\boldsymbol{x}_k)$.
By (\ref{eqn_AG2}), $\mathcal{N}(A_k^u(\Delta);\alpha_0|A_{k}^u(\Delta)|)$ also does not contain $\widetilde\pi^u_k(\boldsymbol{x}_k)$.
Since $A_{k+1}^u(\Delta)\subset \widetilde A_k^u(\Delta)$ and $B_{k+1}^s(\Delta)\subset \widetilde B_k^s(\Delta)$, we have similarly  
$\mathcal{N}(A_{k+1}^u(\Delta);\alpha_0|A_{k+1}^u(\Delta)|)\not\ni \widetilde\pi^u_k(\boldsymbol{x}_k)$.
This shows that 
\begin{equation}\label{eqn_NA_k}
\mathcal{N}(I_k;\alpha_0|A_{k}^u(\Delta)|)\cap 
\mathcal{N}(I_{k+1};\alpha_0|A_{k+1}^u(\Delta)|)=\emptyset.
\end{equation}
Since $\widetilde A_{k+1}^u(\Delta)\subset \widetilde A_k^u(\Delta)$ and $\widetilde B_{k+1}^s(\Delta)\subset \widetilde B_k^s(\Delta)$, one can also prove that 
$$\mathcal{N}(I_k;\alpha_0|A_{k}^u(\Delta)|)\cap 
\mathcal{N}(\widetilde I_{k+1};\alpha_0|A_{k+1}^u(\Delta)|)=\emptyset,$$
where $\widetilde I_{k+1}=\widetilde A_{k+1}^u(\Delta)\cup \widetilde B_{k+1}^s(\Delta)$.
From this fact together with $I_{k+2}\subset \widetilde I_{k+1}$, it follows that
\begin{equation}\label{eqn_NA_k+1}
\mathcal{N}(I_k;\alpha_0|A_{k}^u(\Delta)|)\cap 
\mathcal{N}(I_{k+2};\alpha_0|A_{k+2}^u(\Delta)|)=\emptyset.
\end{equation}
The assertion (\ref{LG_2}) is completed by applying an argument similar to that from 
(\ref{eqn_NA_k}) to (\ref{eqn_NA_k+1}) repeatedly.
\end{proof}

\section{Critical chains}\label{sec.CC}
Since $|\Delta|<\varepsilon \tau^{-3/4}/2$ by Lemma \ref{lem_LG1}, 
we may suppose that the perturbed diffeomorphism $f_\Delta$ given in Subsection \ref{subsec.LP} is arbitrarily $C^r$-close to the original diffeomorphism $f$.
So we reset the notations and write $f_\Delta=f$, $\widetilde L(\Delta)=\widetilde L$, $L(\Delta)=L$, 
$\mathcal{F}_{\mathrm{loc}}^u(\Lambda;\Delta)=\mathcal{F}_{\mathrm{loc}}^u(\Lambda)$, 
$\mathcal{F}_{\mathrm{loc}}^s(\Gamma_m;\Delta)=\mathcal{F}_{\mathrm{loc}}^s(\Gamma_m)$ and so on.
Since $\widetilde L$ and $L$ are parametrized so that $f^{N_2}|_{\widetilde L}:\widetilde L\to L$ is parameter-preserving, 
Linear Growth Lemma (Lemma \ref{lem_LG1}) holds for the bridges $B_{L,k}^s=f^{N_2}(B_k^s(\Delta))$ and 
$A_{L,k}^u=f^{N_2}(A_k^u(\Delta))$ with respect to the Cantors set $K_{\Lambda,L}^s=f^{N_2}(K_\Lambda^s(\Delta))$ and $K_{m,L}^u=f^{N_2}(K_m^u(\Delta))$ in $L$.

\subsection{Encounter of $s$-bridges and $u$-bridges, II}\label{subsec.CCC}
In Subsection \ref{subsec.LP}, we have studied the heteroclinical connection between $s$-bridges $B^s(\Delta)$ of $K_{\Lambda}^s(\Delta)$ and $u$-bridges $A^u(\Delta)$ of $K_{m}^u(\Delta)$ in $\widetilde L$ 
for $\delta=\Delta$.
To construct  wandering domains, we also need to study the homoclinical connection between 
$s$-bridges $B_L^{s*}$ of $K_{\Lambda,L}^s$ and $u$-bridges $B_L^u$ of $K_{\Lambda,L}^u$ in $L$.

We write $\underline{1}^{(n)}:=\underbrace{1\ldots\ldots 1}_{n}$, $\underline{2}^{(n)}:=\underbrace{2\ldots\ldots 2}_{n}$ and prove the following key lemma.
Consider a positive integer $z_0$ independent of $k$ and satisfying the conditions (\ref{eqn_z0+1}) and (\ref{eqn_eta}) 
which are given later.
Let $\{z_k\}_{k=1}^\infty$ be any sequence of integers such that each entry $z_k$ is either $z_0$ or $z_0+1$.

\begin{lemma}[Critical Chain Lemma]\label{lem_CC}
Let 
$(B^{s}_{k}(\Delta), A^{u}_{k}(\Delta))$ $(k=0,1,2,\dots)$ be 
the sequence of the $u$-dominating $\xi_{0}/2$-linked pairs of bridges for $(K^s_{\Lambda}(\Delta),K^u_{m}(\Delta))$ given in
Lemma \ref{lem_LG1}.
Then there exists a constant $T_0>0$ such that, for any $T\geq T_0$ and integers $k\geq 1$,
there are 
\begin{itemize}
\item 
a $u$-bridge $\widehat A_{L,k}^u$ of $K_{m,L}^u$ with $\widehat A_{L,k}^u\subset A_{L,k}^u$ and a $u$-bridge $B_{L,k}^{u}$ of $K^{u}_{\Lambda,L}$ contained in the leading gap of $\widehat A_{L,k}^{u}$, 
\item $s$-bridges $\widehat B_{L,k+1}^{s}$, $B_{L,k+1}^{s*}$ of $K^{s}_{\Lambda,L}$ with 
 $B_{L,k+1}^{s*}\subset \widehat B_{L,k+1}^{s}\subset {B}_{L,k+1}^{s}$;
 \item positive constants $C_{1}$, $C_{2}$ independent of $k$ 
 \end{itemize}
satisfying the following conditions. 
\begin{enumerate}[\rm (1)]
\item \label{CC1}
There exists an interval $J_{k+1}^*\subset (-\varepsilon_0r_{s-}^{-T(k+1)},\varepsilon_0r_{s-}^{-T(k+1)})$ such that 
$(B_{L,k+1}^{s*}+t)\cap B_{L,k+1}^{u}\neq \emptyset$ if and only if 
$t\in J_{k+1}^*$, where $\varepsilon_0=\xi_0|B_0^s(\Delta)|/2$.
\item \label{CC3}
Let $\widehat A_{k}^u$ be the bridge of $K_m^u$ with $\pi^u(\widehat A_{k}^u)=\widehat A_{L,k}^u$ 
 and $\mathbb{G}(\widehat A_{k}^u)$ the strip of the leading gap $G(\widehat A_{k}^u)$ of  defined as in Subsection \ref{subsec.Hetero}, 
where $\pi^u:E\to L$ is the projection of (\ref{eqn_pi^u}).
See Figure \ref{fig_5_2}.
Suppose that $\widehat i_k$ is the minimum positive integer satisfying $f^{\widehat i_k}(\mathbb{G}(\widehat A_{k}^u))\subset \mathbb{G}^u(0)$ and having $N_*+n_*$ as a divisor, 
where $N_*$ and $n_*$ are the integers given in Theorem \ref{lem.renormalization} and 
Subsection \ref{subsec.LP}, respectively.
In other words, $\bar i_k=\widehat i_k/(N_*+n_*)$ is the minimum integer with 
$(\varphi^{\bar i_k}(\mathbb{G}(\widehat A_{k}^u)),\varphi^{\bar i_k}(\flat\mathbb{G}(\widehat A_{k}^u)))\subset (\mathbb{G}^u(0),\flat \mathbb{G}^u(0))$. 
Let $\widehat n_{k+1}$ be the generation of $\widehat B_{k+1}^{s}$.  
Then the inequality  $\widehat{i}_{k}+\widehat{n}_{k+1}<C_{1}Tk+C_{2}$ holds.
\end{enumerate}
Moreover, if the itinerary of $\widehat B_{k+1}^{s}$ is $\widehat{\underline{w}}_{\,k+1}$, then one can suppose that 
$B_{L,k+1}^u=\pi_{\widehat A_k^u}(B_{k+1}^u)$ and 
$B_{L,k+1}^{s*}=\pi^s(B_{k+1}^{s*})$ for 
the projections $\pi_{\widehat A_k^u}:B^u(0)\to L$ of (\ref{eqn_pi_A^u}) and 
$\pi^s:S\to L$ of (\ref{eqn_pi^s})
and for the bridges $B_{k+1}^u$ of $K_\Lambda^u(\Delta)$ and $B_{k+1}^{s*}$ of $K_\Lambda^s(\Delta)$
with 
the itineraries $\underline{1}^{(z_kk^2)}\underline{2}^{(k^2)}\,\underline{v}_{\,k+1}\,[\widehat{\underline{w}}_{\,k+1}]^{-1}$ and $\widehat{\underline{w}}_{\,k+1}\,[\underline{v}_{\,k+1}]^{-1}\underline{2}^{(k^2)}\,\underline{1}^{(z_kk^2)}$ respectively.
That is, 
\begin{equation}\label{eqn_CC1}
\begin{split}
B_{k}^u&=B^u(z_kk^2+\langle k\rangle;\,\underline{1}^{(z_kk^2)}\underline{2}^{(k^2)}\,\underline{v}_{\,k+1}\,[\widehat{\underline{w}}_{\,k+1}]^{-1}),\\
B_{k+1}^{s*}&=B^s(z_kk^2+\langle k\rangle;\,\widehat{\underline{w}}_{\,k+1}\,[\underline{v}_{\,k+1}]^{-1}\underline{2}^{(k^2)}\,\underline{1}^{(z_kk^2)}),
\end{split}
\end{equation}
where
\begin{equation}\label{eqn_lakra}
\langle k\rangle=\widehat n_{k+1}+k^2+k
\end{equation}
and $\underline{v}_{\,k+1}$ is an arbitrarily chosen element of $\{1,2\}^{k}$.
In other words, 
$\underline{v}_{\,k+1}$ is a non-specified itinerary of length $k$.
\end{lemma}

See Figure \ref{fig_7_1} for the situation of Lemma \ref{lem_CC}, where 
$\widehat G_{L,k+1}^u=\pi^u(G(\widehat A_{k+1}^u))$.
Figure \ref{fig_7_2} illustrates the transition from $B_{k}^u$ to $B_{L,k}^u$ via 
$f^{-N_0}\circ f^{-\widehat i_k}\circ f^{-N_1}$ schematically.
\begin{figure}[hbt]
\centering
\scalebox{0.8}{\includegraphics[clip]{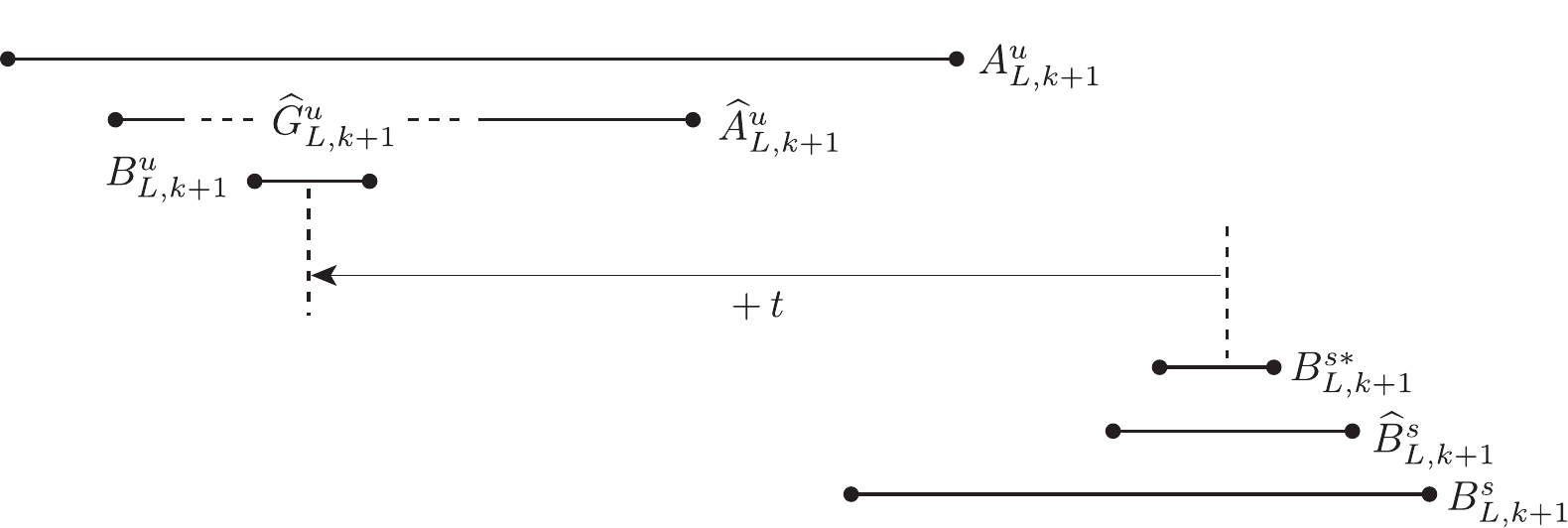}}
\caption{Sliding of $B_{L,k+1}^{s*}$ by $+t$.}
\label{fig_7_1}
\end{figure}
\begin{figure}[hbt]
\centering
\scalebox{0.8}{\includegraphics[clip]{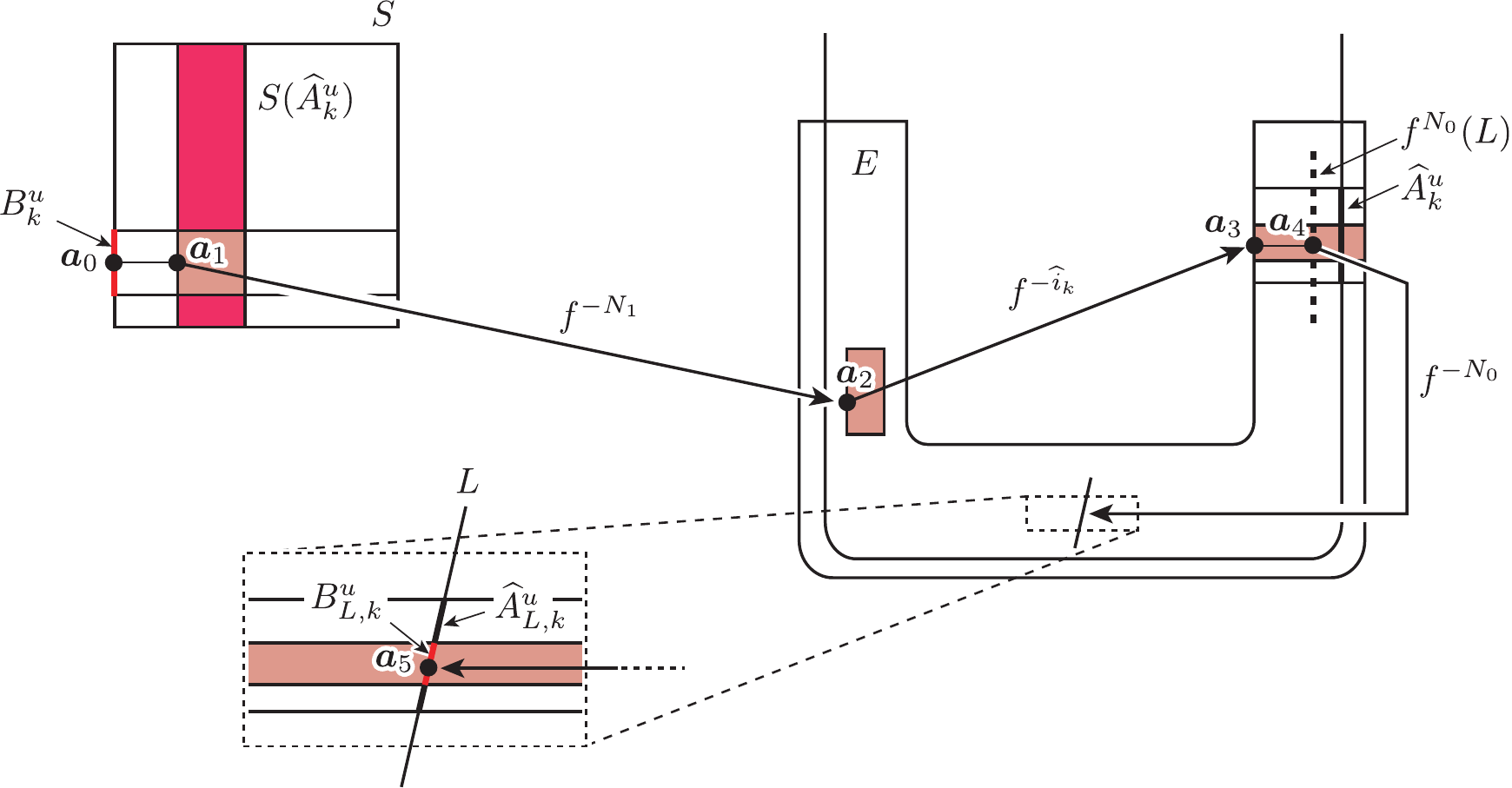}}
\caption{Backward transition from $B_k^u$ to $B_{L,k}^u$. 
$\boldsymbol{a}_0$ and $\boldsymbol{a}_1$ lie on the same horizontal line and $\boldsymbol{a}_3$ and $\boldsymbol{a}_4$ also do.}
\label{fig_7_2}
\end{figure}

\begin{proof}[Proof of Lemma \ref{lem_CC}]
Let $T_0$ be the constant defines as
$$T_0=\frac{N_s\log(r_{s+})}{\log(r_{s-})}.$$
Since $\varepsilon_0=\xi_0|B_0^s(\Delta)|/2$, 
by Lemmas \ref{lem.bdp1} and \ref{lem_LG1}-(2), 
$$|B_{L,k}^s\cap A_{L,k}^u|\geq \frac{\xi_0|B_k^s(\Delta)|}2 
\geq \frac{\xi_0|B_0^s(\Delta)|}2\, r_{s+}^{-N_sk}=\varepsilon_0 r_{s-}^{-T_0k}$$
for any integer $k\geq 1$.
Here we fix an integer $T\geq T_0$.
By Lemma \ref{lem_LL2}, 
there exist an 
interval $J_{k}$ with $J_{k}\subset (-\varepsilon_0r_{s-}^{-Tk},\varepsilon_0r_{s-}^{-Tk})$, 
sub-bridges 
 $\widehat{B}^{s}_{L,k}\subset {B}_{L,k}^{s}$ and $\widehat{A}^{u}_{L,k}\subset {A}_{L,k}^{u}$ satisfying the  following conditions;
\begin{enumerate}[\hspace*{15pt}(a)]
\item  \label{lem.CC-b}
$\tau^{-5/4}\varepsilon_0r_{s-}^{-Tk}\leq |\widehat{B}_{L,k}^{s}|<\tau^{-3/4}\varepsilon_0r_{s-}^{-Tk}$;
\item  \label{lem.CC-c}
$\tau^{-1/2}|\widehat{A}_{L,k}^{u}|\leq |\widehat{B}_{L,k}^{s}|<\tau^{-1/4}|\widehat{A}_{L,k}^{u}|$;
\item \label{lem.CC-a}
$(\widehat{B}_{L,k}^{s}+t)\cap \widehat{A}_{L,k}^{u}\neq\emptyset$ if and only if   $t\in J_{k}$.
\end{enumerate}
Suppose that $\widehat{A}_{L,k}^{u}$ is of generation $i_{k}> 0$.
By (\ref{eqn_widehat m}), we have the integer $\bar i_k$ with 
$(\varphi^{\bar i_k}(\mathbb{G}(A_{k}^u)),\varphi^{\bar i_k}(\flat\mathbb{G}(A_{k}^u)))\subset (\mathbb{G}^u(0),\flat \mathbb{G}^u(0))$ and $i_k\leq \bar i_k\leq (m-1)i_k$.
We set $\widehat i_k=\bar i_k(N_*+i_*)$. 
Let $B_{k+1}^{s*}$ and $B_{k+1}^{u}$ be the bridges defined as (\ref{eqn_CC1}).
Since $B_{L,k+1}^{s*}=\pi^s(B_{k+1}^{s*})\subset \widehat B_{L,k+1}^s$ 
and $B_{L,k+1}^u=\pi_{\widehat A_k^u}(B_{k+1}^u)\subset \widehat A_{L,k+1}^u$, 
by the condition (\ref{lem.CC-a}),  
there is a sub-interval $J_{k+1}^*$ of $J_{k+1}$ such that 
$(B_{L,k+1}^{s*}+t)\cap B_{L,k+1}^{u}\neq \emptyset$ if and only if 
$t\in J_{k+1}^*$, see Figures \ref{fig_7_1} and \ref{fig_7_2}.
This completes the proof of (\ref{CC1}).

Now we will show (\ref{CC3}). 
By Lemma \ref{lem.bdp1}, one has
$$
r_{s+}^{-\widehat{n}_{k}}|{B}_{0}^{s}(\Delta)| \leq  | \widehat{B}_{k}^{s}(\Delta)|\leq r_{s-}^{-\widehat{n}_{k}} |{B}_{0}^{s}(\Delta)|,
$$
and from (\ref{lem.CC-b})
$$
\tau^{-5/4}\varepsilon_0r_{s-}^{-Tk}  \leq r_{s-}^{-\widehat{n}_{k}} |{B}_{0}^{s}(\Delta)|.
$$
Hence,  for every $l\geq 0$, 
\begin{equation}\label{order bar-n}
\widehat{n}_{k+l}\leq T (k+l)+\frac{\log (\tau^{5/4}\varepsilon_0^{-1}|{B}_{0}^{s}(\Delta)|)}{\log r_{s-}}.
\end{equation}
It follows from (\ref{lem.CC-c}) together with Lemma \ref{lem.bdp2} that 
$$\tau^{-5/4}\varepsilon_0r_{s-}^{-Tk}\leq |\widehat{B}_{L,k}^{s}|<\tau^{-1/4}\Bigl(\frac32\Bigr)^{-i_{k}}|{A}_{L,0}^{u}|.$$
By using the inequalities as above, we have
\begin{align*}
\log\Bigl(\frac32\Bigr)i_k&\leq \log (r_{s-})Tk+\log (\tau\varepsilon_0^{-1}|A_{L,0}^u|)+\widehat n_k\log (r_{s+}/r_{s-})\\
&\leq \log (r_{s-})Tk+\log (\tau\varepsilon_0^{-1}|A_{L,0}^u|)+\left(Tk+\frac{\log(\tau^{5/4}\varepsilon_0^{-1}|B_0^s(\Delta)|)}{\log (r_{s-})}\right)\log (r_{s+}/r_{s-})\\
&=\log (r_{s+})Tk+C_3,
\end{align*}
where $\displaystyle C_3=\log (\tau\varepsilon_0^{-1}|A_{L,0}^u|)+\frac{\log(\tau^{5/4}\varepsilon_0^{-1}|B_0^s(\Delta)|)\log (r_{s+}/r_{s-})}{\log (r_{s-})}$.
Hence  
\begin{equation}\label{order bar-m}
\widehat{i}_{k}
\leq (m-1)i_k(N_*+n_*)\leq 
\frac{\log (r_{s+}^{(N_*+n_*)(m-1)})}{\log\bigl(\frac32\bigr)}Tk+\frac{C_3(N_*+n_*)(m-1)}{\log\bigl(\frac32\bigr)}.
\end{equation}
By (\ref{order bar-n}) and (\ref{order bar-m}), one can get positive constants $C_1$ and $C_2$ satisfying (\ref{CC3}).
\end{proof}

\subsection{The second perturbation of $f$}\label{SS_Pert}
In Subsection \ref{subsec.LP}, we perturbed $f$ by performing the `$\Delta$-sliding' of $f$ along the arc  $\widetilde L$.
The second perturbation presented here is `switchback slidings' in neighborhoods of the brides $B_k^s(\Delta)$ $(k=1,2,\dots)$ which are done individually by using bump functions with mutually disjoint supports in $S$.

Consider bridges $B^s$ in $[0,2]\times \{0\}\subset W_{\mathrm{loc}}^s(p)$ and $B^u$ in $\{0\}\times [0,2]\subset W_{\mathrm{loc}}^u(p)$ associated with the Cantor sets $K_\Lambda^s$ and $K_\Lambda^u$ of (\ref{def.Ks}), respectively.
Set $\mathbb{B}^s=(\pi_{\mathcal{F}_{\mathrm{loc}}^u(\Lambda)})^{-1}(B^s)$ and $\mathbb{B}^u=(\pi_{\mathcal{F}_{\mathrm{loc}}^s(\Lambda)})^{-1}(B^u)$, where $\pi_{\mathcal{F}_{\mathrm{loc}}^u(\Lambda)}$ and $\pi_{\mathcal{F}_{\mathrm{loc}}^s(\Lambda)}$ are the projections given in Subsection \ref{subsec.horse}.
Note that $\mathbb{B}^s$ is a sub-strip of $(S,\sharp S)$ and $\mathbb{B}^u$ is a sub-strip of $(S,\flat S)$.
We call that $\mathbb{B}^{u(s)}$ is the \emph{bridge strip}  for $B^{u(s)}$.
From our definition of ${B}_{k}^u$ and ${B}_{k+1}^{s*}$ in (\ref{eqn_CC1}), 
$f^{z_kk^2+\langle k\rangle}(\mathbb{B}_{k}^u)=\mathbb{B}_{k+1}^{s*}$, see Figure \ref{fig_7_3}.
The \emph{gap strip} $\mathbb{G}^{u(s)}$ for a gap $G^{u(s)}$ of $K_\Lambda^{u(s)}$ is defined similarly.

Recall that the bridges $B^s(\Delta)$, $A^u(\Delta)$ in Lemma \ref{lem_LG1} are taken in $\widetilde L\cap \mathcal{B}_{\delta_0/2}$ so that our perturbation of the diffeomorphisms $f$ does 
not affect the invariant set $\Gamma_m$ and local stable foliation $\mathcal{F}_{\mathrm{loc}}^s(\Gamma_m)$ on $E$.

Now we need to choose them more carefully.
Let $D_{n_*}$ be the rectangle used in Section \ref{sec.prep} to define the basic set $\Gamma_m$, which satisfies the conditions (S-\ref{setting6}) and 
(S-\ref{setting7}).
The dynamics of $\varphi=\varphi_{n_*}$ on $D_{n_*}$ is determined by that of $f$ on $X_{n_*}=D_{n_*}\cup f(D_{n_*})\cup \cdots\cup f^{N_*+n_*}(D_{n_*})$, 
where $N_*$ and $n_*$ are the integers given in Theorem \ref{lem.renormalization} and Subsection \ref{subsec.LP}, respectively.
So our aim is accomplished by perturbing $f$ in the complement of a neighborhood of 
$X_{n_*}$.
By (S-\ref{setting6}) and Theorem \ref{lem.renormalization}, one can choose the parameter 
$\mu_*=\Theta_{n_*}(\bar \mu_*)$ and bridges $B^s(\Delta)$ of $K_{\Lambda}(\Delta)$ and $A^u(\Delta)$ of $K_{m}(\Delta)$ so that 
$f=f_{\mu_*}$ satisfies the conditions of Lemma \ref{lem_LL1} and the following extra conditions 
for a small constant $\alpha_1>0$.
Recall that $\varpi^u:S\to [0,2]\times \{0\}\subset W_{\mathrm{loc}}^s(p)$ is the vertical projection 
with respect to the orthogonal coordinate on $S$, which is given in Subsection \ref{subsec.LP}.
\begin{enumerate}[({B-}i)]
\item\label{B1}
For the sub-arc $I=B^s(\Delta)\cup A^u(\Delta)$ of $\widetilde L$, the 
$\alpha_1|I|$-neighborhood $\widehat{\mathbb{I}}$ of the strip $\mathbb{I}=(\varpi^u)^{-1}\circ \varpi^u(I)$ in 
$S$ is disjoint from $X_{n_*}$.
\item\label{B2}
For the sub-bridges $B_k^s(\Delta)$ of $B^s(\Delta)$ and $A_k^u(\Delta)$ of $A^u(\Delta)$  given in Lemma \ref{lem_LG1} and the sub-arc $I_k=B_k^s(\Delta)\cup A_k^u(\Delta)$ of (\ref{eqn_Ikf}), 
the strips $\widehat{\mathbb{I}}_k=(\varpi^u)^{-1}(\widehat I_k)$ 
$(k=1,2,\dots)$ in $S$ 
are mutually disjoint, where $\widehat I_k$ is the $\alpha_1|\varpi^u(I_k)|$-neighborhood of $\varpi^u(I_k)$ in $[0,2]\times \{0\}$.
\end{enumerate}

In fact, the assertion (B-\ref{B1}) is guaranteed by choosing the parameter $\bar\mu_*$ and the initial linked pair $B^s(0)$ and $A^u(0)$ of Lemma \ref{lem_LL1} suitably so that the 
sub-arc $I$ is contained in an arbitrarily small neighborhood of the 
left component of $\flat S$ in $S$, which is represented by the shaded region in 
Figure \ref{fig_3_2}.
We refer to Subsection 6.5 in \cite{PT93} (and also Subsection 5.3 in \cite{KS08}) for such an argument.
The assertion (B-\ref{B2}) holds if 
we take $\alpha_1$ sufficiently small comparing with $\alpha_0$ of Lemma \ref{lem_LG1}-(\ref{LG_2}).

We will define an auxiliary stable foliation on $S$ for Lemma \ref{l_slope} below.
Consider a $C^1$-foliation $\mathcal{G}^s(0)$ on $\mathbb{G}^u(0)$ such that $\widetilde L$ and any components of $\sharp\mathbb{G}^u(0)$ are leaves of $\mathcal{G}^s(0)$ and each leaf meets $\mathcal{F}_{\mathrm{loc}}^u(\Lambda)$ 
exactly.
Then $\mathcal{G}^s(0)$ is uniquely extended to a local stable $C^1$-foliation $\mathcal{G}_{\mathrm{loc}}^{s}(\Lambda)$ 
on $S$ compatible with $W_{\mathrm{loc}}^s(\Lambda)$.

Let $\widehat A_{L,k}^u$ and $B_{L,k}^u$ be the bridges of $K_{m,L}^u$ and $K_{\Lambda,L}^u$ given in Lemma \ref{lem_CC}, respectively.
Note that $\widehat A_{L,k}^u$ contains $B_{L,k}^u$, see Figure \ref{fig_7_1}.
Consider the curves $\widetilde L_k$ in $S$ and $L_k$ in $U(q)$ defined by
$$\widetilde L_k=f^{-(z_kk^2+\langle k\rangle)}(\mathbb{B}_{k+1}^{s*}\cap \widetilde L),
\quad L_k=f^{-N_0}\circ f^{-\widehat i_k}\circ f^{-N_1}(\widetilde L_k\cap 
S(\widehat A_k^u)),$$
see Figure \ref{fig_7_3}, 
where  
$S(\widehat A_k^u)$ is the strip in $S$ associated with $\widehat A_k^u$, see also Figure \ref{fig_5_2}.
\begin{figure}[hbt]
\centering
\scalebox{0.8}{\includegraphics[clip]{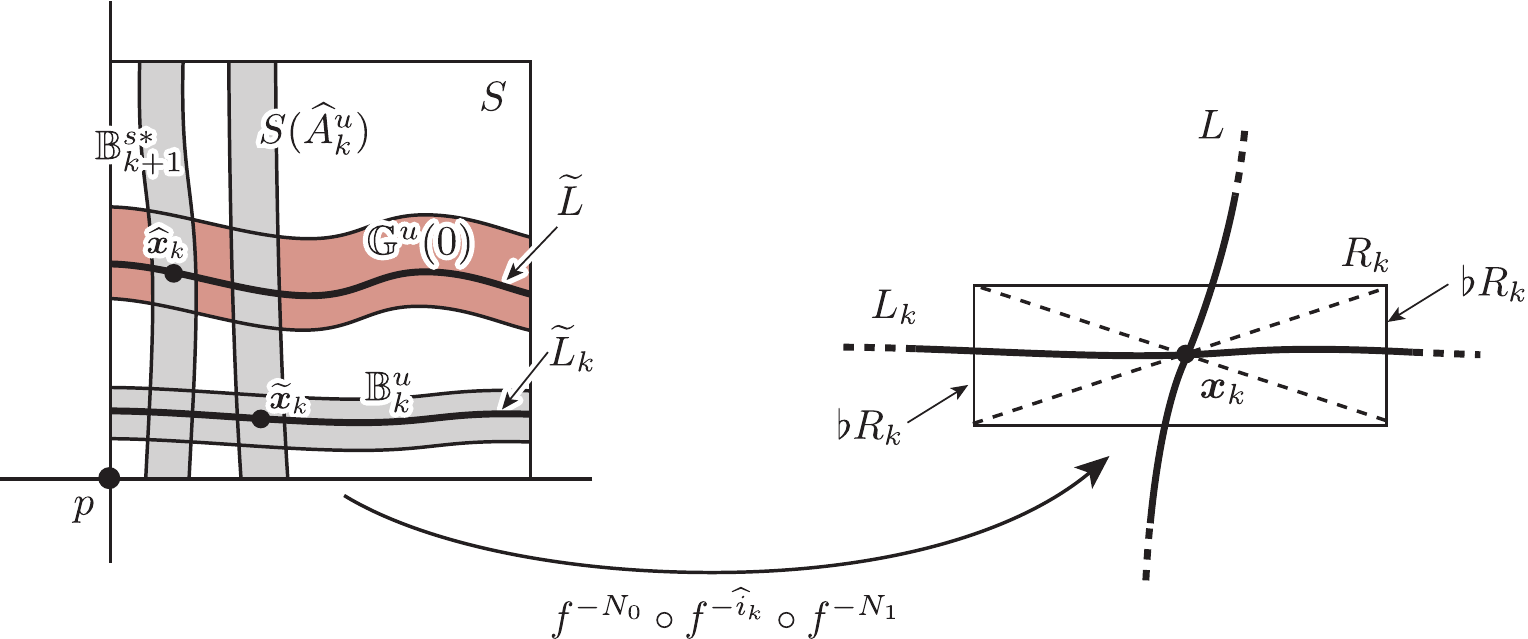}}
\caption{Backward transition from $L$ to $L_k$ via $\widetilde L$ and $\widetilde L_k$.}
\label{fig_7_3}
\end{figure}
Note that $\widetilde L_k$ is a leaf of the foliation $\mathcal{G}_{\mathrm{loc}}^s(\Lambda)$.

Suppose that $l$ is a compact $C^1$-curve in $S$ such that any line tangent to $l$ is not vertical.
For any $\boldsymbol{a}\in l$, let $\mathrm{slope}_{\boldsymbol{a}}(l)\geq 0$ be the (absolute) slope 
of $l$ at $\boldsymbol{a}$ and define the maximum slope $\mathrm{slope}(l)$ of $l$ by $\max\{\mathrm{slope}_{\boldsymbol{a}}(l)\,;\,\boldsymbol{a}\in l\}$.
The slope is defined similarly for compact $C^1$-curves in $U(L)$ which do not have vertical tangent lines, 
where $U(L)$ is supposed to have the $C^{1+\alpha}$-coordinate introduced in Subsection \ref{subsec.LP}.

\begin{lemma}\label{l_slope}
There exists a constant $\alpha>1$ satisfying the following inequalities for any integer $k>0$ 
and any leaf $l$ of $\mathcal{G}_{\mathrm{loc}}^{s}(\Lambda)$ contained in $\mathbb{B}_{k}^u$.
$$
\mathrm{slope}(l)<\alpha(\sigma^{-1}\lambda)^{z_kk^2},\quad \mathrm{slope}(L_k)<\alpha^k(\sigma^{-1}\lambda)^{z_kk^2}.$$
In particular, $\mathrm{slope}(\widetilde L_k)<\alpha(\sigma^{-1}\lambda)^{z_kk^2}$.
\end{lemma}

\begin{proof}
Since the left component $\{0\}\times [0,2]$ of $\flat S$ is a leaf of $\mathcal{F}_{\mathrm{loc}}^u(\Lambda)$, 
any leaf of $\mathcal{G}_{\mathrm{loc}}^{s}(\Lambda)$ meets $\{0\}\times [0,2]$ transversely.
If necessary replacing $\delta$ by a smaller positive number, we may assume that,  
for any leaf $l$ of $\mathcal{G}_{\mathrm{loc}}^{s}(\Lambda)$, the restriction $l_{[0,\delta]}=l\cap ([0,\delta]\times [0,2])$ has no vertical tangent line.
From the compactness of $\mathcal{G}_{\mathrm{loc}}^{s}(\Lambda)$, we have a constant $\alpha_1>0$ independent 
of $l\in \mathcal{G}_{\mathrm{loc}}^{s}(\Lambda)$ such that $\mathrm{slope}(l_{[0,\delta]})<\alpha_1$.
We denote the left edge of $l$ by $e(l)$.
Let $n_0$ be the smallest integer with $n_0\geq \log\delta/\log\lambda$.
For any $n\geq n_0$, the component $l_n$ of $f^{-n}(l_{[0,\delta]})\cap S$ containing $f^{-n}(e(l))$ is a leaf of $\mathcal{G}_{\mathrm{loc}}^{s}(\Lambda)$.
Since $f^n$ is the linear map $(\sigma^nx,\lambda^ny)$ on a small neighborhood of $l_n$ in $S$, $\mathrm{slope}(l_n)<\alpha_1(\sigma^{-1}\lambda)^n$.

Now we suppose that $l$ is any leaf of $\mathcal{G}_{\mathrm{loc}}^{s}(\Lambda)$ contained in $\mathbb{B}_{k}^u$.
It follows from the form (\ref{eqn_CC1}) of $B_{k}^u$ that, for any integer $k>0$ with $z_kk^2\geq z_0k^2> n_0$, 
$f^{z_kk^2}$ is the linear map $(\sigma^{z_kk^2}x,\lambda^{z_kk^2}y)$ on a small neighborhood of $l$ 
with $f^{z_kk^2}(l)\subset l_{[0,\delta]}'$ for some leaf $l'$ of $\mathcal{G}_{\mathrm{loc}}^{s}(\Lambda)$.
This shows that $\mathrm{slope}(l)<\alpha_1(\sigma^{-1}\lambda)^{z_kk^2}$.
By replacing $\alpha_1$ by a larger constant $\alpha>1$ if necessary, one can suppose that 
$\mathrm{slope}(l)<\alpha(\sigma^{-1}\lambda)^{z_kk^2}$ for all $k> 0$.
Thus the first inequality holds.

Since $M$ is compact, there exists a constant $\gamma>1$ satisfying
\begin{equation}\label{eqn_bDfb}
\gamma^{-1}< \|Df^{-1}_{\boldsymbol{x}}(\boldsymbol{v})\|<\gamma
\end{equation}
for any point $\boldsymbol{x}$ of $M$ and any unit tangent vector $\boldsymbol{v}\in T_{\boldsymbol{x}}(M)$.
By Lemma \ref{lem_CC}-(\ref{CC3}), there exists a constant $H>0$ with $N_0+\widehat i_k+N_1<Hk$ for any $k>0$.
Note that $f^{-N_0}\circ f^{-\widehat i_k}\circ f^{-N_1}$ maps $\widetilde L_k\cap S(\widehat A_k)$ onto $L_k$ and the horizontal segment passing through any point of $\widetilde L_k\cap S(\widehat A_k)$  to the horizontal segment passing through a point of $L_k$. 
Since moreover $\mathrm{slope}(\widetilde L_k)<\alpha(\sigma^{-1}\lambda)^{z_kk^2}$,  
$$\mathrm{slope}(L_k)<\alpha\gamma^{2Hk}(\sigma^{-1}\lambda)^{z_kk^2}
<(\alpha\gamma^{2H})^k(\sigma^{-1}\lambda)^{z_kk^2}.$$
Thus we have the second inequality by denoting $\alpha\gamma^{2H}$ newly by $\alpha$.
\end{proof}

In particular, Lemma \ref{l_slope} implies that, for all sufficiently large $k$, $L_k$ is almost horizontal 
and hence $L_k$ meets $L$ transversely at a single point $\boldsymbol{x}_k=(x_k,y_k)$.
Note that $\widetilde{\boldsymbol x}_k=f^{N_1}\circ f^{\widehat i_k}\circ f^{N_0}(\boldsymbol{x}_k)$ is a point of $\widetilde L_k$.
There exists a neighborhood $\mathcal{N}_0$ of $\widetilde L$ in $\mathbb{G}^u(0)\setminus \sharp\mathbb{G}^u(0)$ consisting of leaves of $\mathcal{G}^s(0)$.
In particular, $\mathcal{N}_0$ is a sub-strip of $(\mathbb{G}^u(0),\flat \mathbb{G}^u(0))$ 
and the components $l_{1,0}$, $l_{2,0}$ of $\sharp\mathcal{N}_0$ meet $\mathcal{F}_{\mathrm{loc}}^u(\Lambda)$ exactly.
We set
\begin{equation}\label{eqn_calN0}
\mathrm{dist}(\sharp \mathcal{N}_0,\widetilde L)=\zeta>0.
\end{equation}
Let $\mathcal{N}_k$ be the component of $f^{-(z_kk^2+\langle k\rangle)}(\mathcal{N}_0)\cap S$ containing $\widetilde L_k$.
Since $\mathcal{N}_0\subset \mathbb{G}^u(0)$, $\mathcal{N}_k$ is contained in some 
gap strip in $\mathbb{B}_{k}^u$.

For sequences $\{u_k\}$, $\{v_k\}$ of positive numbers, $u_k\asymp v_k$ means that
there exist constants $0<c_1<c_2$ independent of $k$ such that 
$c_1\leq \dfrac{u_k}{v_k}\leq c_2$ holds for all $k$.

\begin{lemma}\label{l_[0,2]}
There exists a constant $0<\nu<1$ such that, 
for all sufficiently large $k$, 
$[0,2]\times \bigl[\widetilde y_k-\nu^{k^2}\sigma^{-z_kk^2},\widetilde y_k+\nu^{k^2}\sigma^{-z_kk^2}\bigr]$ 
is contained in $\mathcal{N}_{k}$, where 
$\widetilde y_k$ is the $y$-entry of $\widetilde{\boldsymbol x}_k=(\widetilde x_k,\widetilde y_k)$, see 
Figure \ref{fig_7_4}.
\end{lemma}
\begin{figure}[hbt]
\centering
\scalebox{0.8}{\includegraphics[clip]{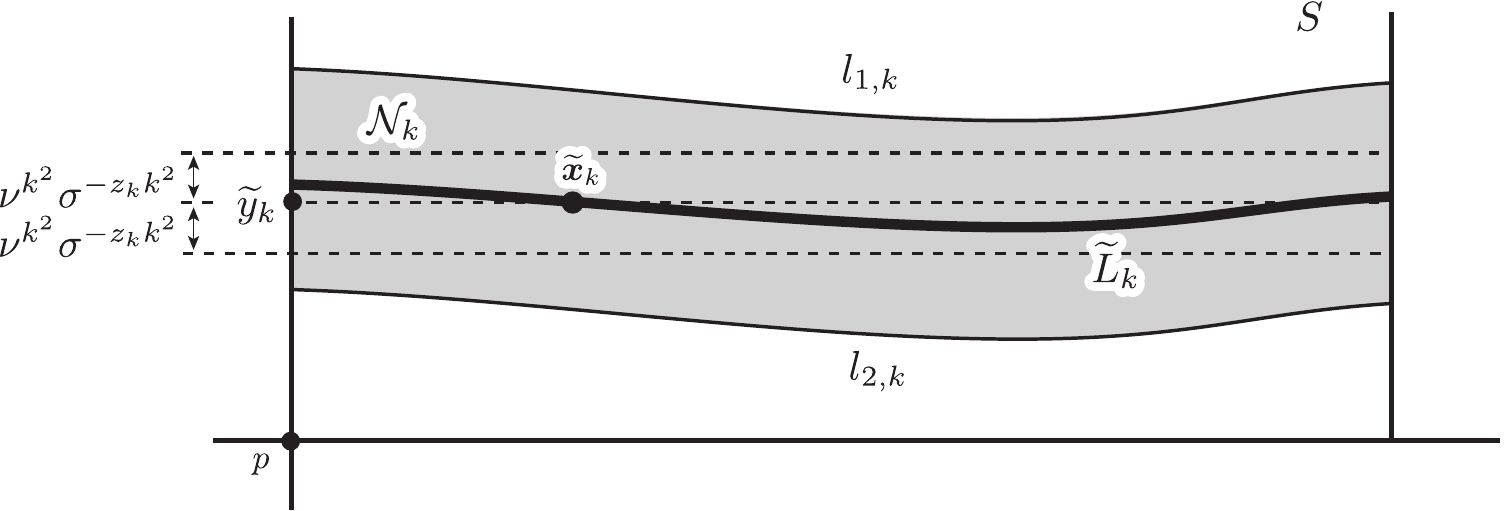}}
\caption{Situation in $\mathcal{N}_k$.}
\label{fig_7_4}
\end{figure}
\begin{proof}
By (\ref{eqn_bDfb}) and (\ref{eqn_calN0}), 
we have 
$$\mathrm{dist}(\sharp\mathcal{N}_k,\widetilde L_k)\geq \gamma^{-\langle k\rangle}\sigma^{-z_kk^2}\zeta.$$
By Lemma \ref{lem_CC}-(\ref{CC3}), $\widehat n_{k+1}\asymp k$ and hence 
$\langle k\rangle=\widehat n_{k+1}+k^2+k\asymp k^2$.
Thus one can take a constant $0<\nu<1$ satisfying
$$
\mathrm{dist}(\sharp\mathcal{N}_k,\widetilde L_k)\geq 2\nu^{k^2}\sigma^{-z_kk^2}
$$
for any $k>0$.
Since the components $l_{1,k}$, $l_{2,k}$ of $\sharp\mathcal{N}_k$ are leaves of $\mathcal{G}_{\mathrm{loc}}^s(\Lambda)$,   
$\mathrm{slope}(l_{i,k})<\alpha(\sigma^{-1}\lambda)^{z_kk^2}$ for $i=1,2$ by Lemma \ref{l_slope}.
Here we suppose that the integer $z_0$ satisfies
\begin{equation}\label{eqn_lambdaz0nu}
\lambda^{z_0}<\nu.
\end{equation}
We will see later that the condition (\ref{eqn_lambdaz0nu}) is implied by the condition (\ref{eqn_eta}).
Since   $$\frac{\mathrm{slope}(l_{i,k})}{2\nu^{k^2}\sigma^{-z_kk^2}}<
\frac{\alpha(\nu^{-1})^{k^2}\lambda^{z_0k^2}}2\to 0\quad (k\to \infty)$$
for $i=1,2$, it follows that
$$[0,2]\times \bigl[\widetilde y_k-\nu^{k^2}\sigma^{-z_kk^2},\widetilde y_k+\nu^{k^2}\sigma^{-z_kk^2}\bigr]\ 
\subset\ \mathcal{N}_{k}$$
for all sufficiently large $k$. 
\end{proof}

Consider the map $h_k$ defined by
\begin{equation}\label{eqn_h_k}
h_k:= f^{z_kk^2+\langle k\rangle}\circ f^{N_{1}}\circ f^{{\widehat i}_{k}}\circ f^{N_{0}}
\end{equation}
and the sequence $\{\widehat{\boldsymbol{x}}_k\}$ with $\widehat{\boldsymbol{x}}_k=f^{z_kk^2+\langle k\rangle}(\widetilde{\boldsymbol x}_k)=h_k(\boldsymbol{x}_k)$.
Let $\boldsymbol{u}_k=(u_k,v_k)$ be the vector with
\begin{equation}\label{eqn_uk}
\widehat{\boldsymbol{x}}_k+\boldsymbol{u}_k=f^{-N_2}(\boldsymbol{x}_{k+1}).
\end{equation}
Figure \ref{fig_7_5} illustrates the transition of base points schematically.
\begin{figure}[hbt]
\centering
\scalebox{0.8}{\includegraphics[clip]{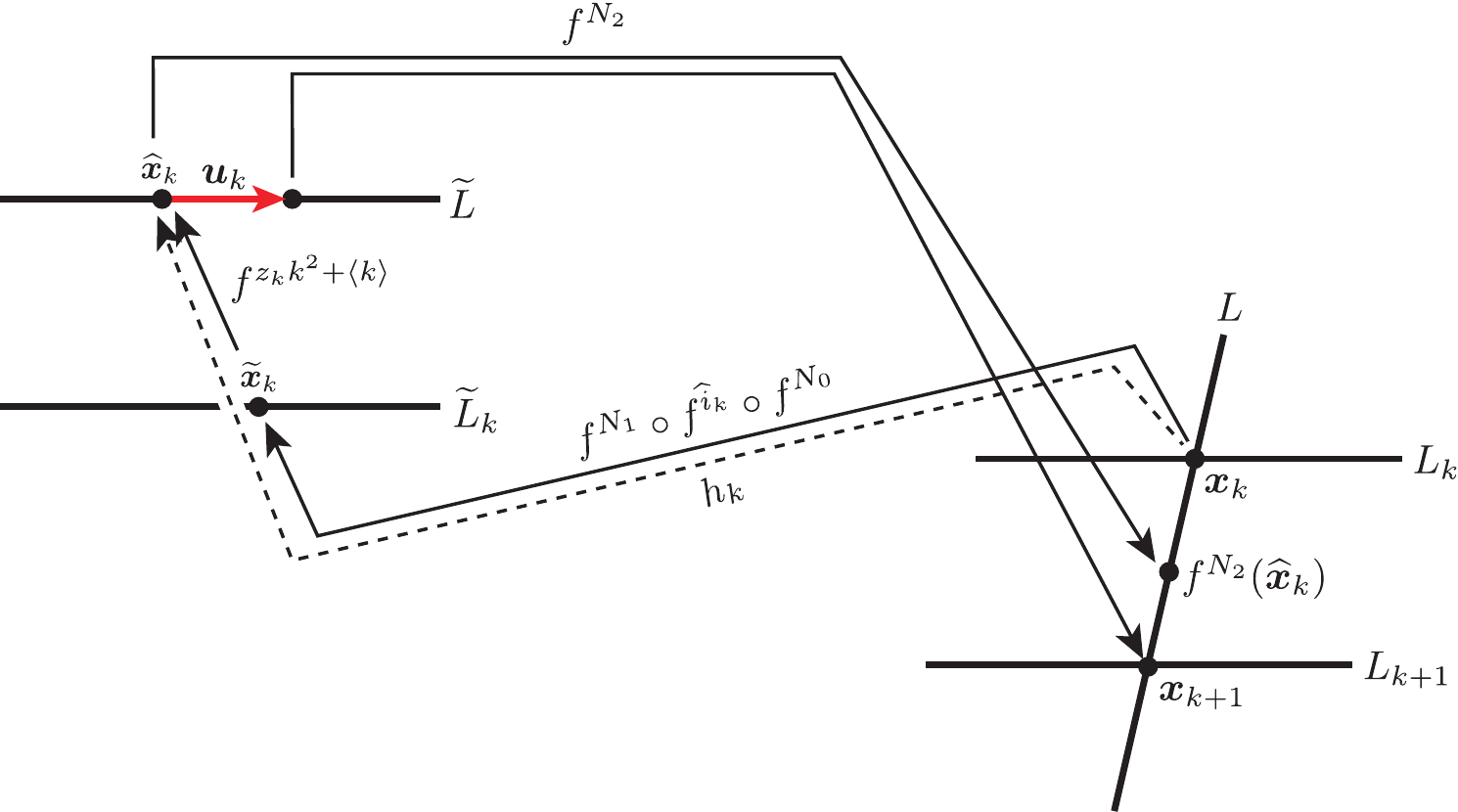}}
\caption{A transition from $\boldsymbol{x}_k$ to $f^{N_2}(\widehat{\boldsymbol x}_k)$ and  
a shifted transition from $\boldsymbol{x}_k$ to ${\boldsymbol x}_{k+1}$.}
\label{fig_7_5}
\end{figure}

\begin{lemma}\label{l_uvt}
There exists a constant $\beta>0$ independent of $k$ such that
$\|\boldsymbol{u}_k\|\leq \beta r_{s-}^{-Tk}$, where $T$ is a number not smaller than the constant 
$T_0$ given in Lemma \ref{lem_CC}.
\end{lemma}
\begin{proof}
Since $f^{N_2}(\widehat{\boldsymbol x}_k)\in B_{L,k+1}^{s*}$ and 
$\boldsymbol{x}_{k+1}\in B_{k+1}^u$, it follows from Lemma \ref{lem_CC}-(\ref{CC1}) that there exists a vector $\boldsymbol{t}_{k+1}=(s_{k+1},t_{k+1})$ with $|t_{k+1}|<2^{-1}\xi_0|B_0^s(\Delta)|
r_{s-}^{-T(k+1)}$ and $f^{N_2}(\widehat{\boldsymbol x}_k)+\boldsymbol{t}_{k+1}=\boldsymbol{x}_{k+1}$.
Since $N_2$ is independent of $k$, $\|\boldsymbol{u}_k\|\asymp \|\boldsymbol{t}_{k+1}\|
\asymp |t_{k+1}|$.
Thus we have a constant $\beta>0$ satisfying 
$\|\boldsymbol{u}_k\|=\|\boldsymbol{x}_{k+1}-f^{-N_2}(\widehat{\boldsymbol x}_k)\|\leq \beta r_{s-}^{-Tk}$ for any $k\geq 1$.
\end{proof}

This proof suggests that $\boldsymbol{t}_{k+1}$ and hence $\boldsymbol{u}_k$ depend on $T$.
So we will write 
$$\boldsymbol{t}_k=\boldsymbol{t}_k(T)\quad\text{or}\quad \boldsymbol{u}_k=\boldsymbol{u}_k(T)$$
when we emphasize the dependence.

\medskip

Recall that we supposed that $r$ is an integer with $3\leq r<\infty$.
The projection $\pi_{\mathcal{F}_{\mathrm{loc}}^s(\Lambda)}$ is a $C^{1+\alpha}$-function but not necessary of $C^r$-class.
So we need a suitable substitute for the map.
Let $\varpi^s:\mathbb{G}^u(0)\to \{0\}\times [0,2]\subset W_{\mathrm{loc}}^u(p)$ be a $C^r$-map arbitrarily $C^1$-close to $\pi_{\mathcal{F}_{\mathrm{loc}}^s(\Lambda)}|_{\mathbb{G}^u(0)}$.
Since $\mathcal{N}_0$ is contained in $\mathbb{G}^u(0)\setminus \sharp\mathbb{G}^u(0)$, 
there exist a $d>0$ and sub-intervals $H$, $\widehat H$ of $G^u(0)$ such that 
$G^u(0)=[\,\min H-2d|H|,\max H+2d|H|\,]$, $\widehat H=[\,\min H-d|H|,\max H+d|H|\,]$ and $\pi^{-1}_{\mathcal{F}_{\mathrm{loc}}^s(\Lambda)}(H)$ contains $\mathcal{N}_0$.
Then one can choose the $C^r$-map $\varpi^s$ so that 
$\mathbb{H}=(\varpi^s)^{-1}(H)$ and $\widehat{\mathbb{H}}=(\varpi^s)^{-1}(\widehat H)$ are 
strips in $S$ with
$$\mathcal{N}_0\ \subsetneq\ \mathbb{H}\ \subsetneq \ \widehat{\mathbb{H}}\ \subsetneq \ \mathbb{G}^u(0).$$
See Figure \ref{fig_7_6}.
\begin{figure}[hbt]
\centering
\scalebox{0.8}{\includegraphics[clip]{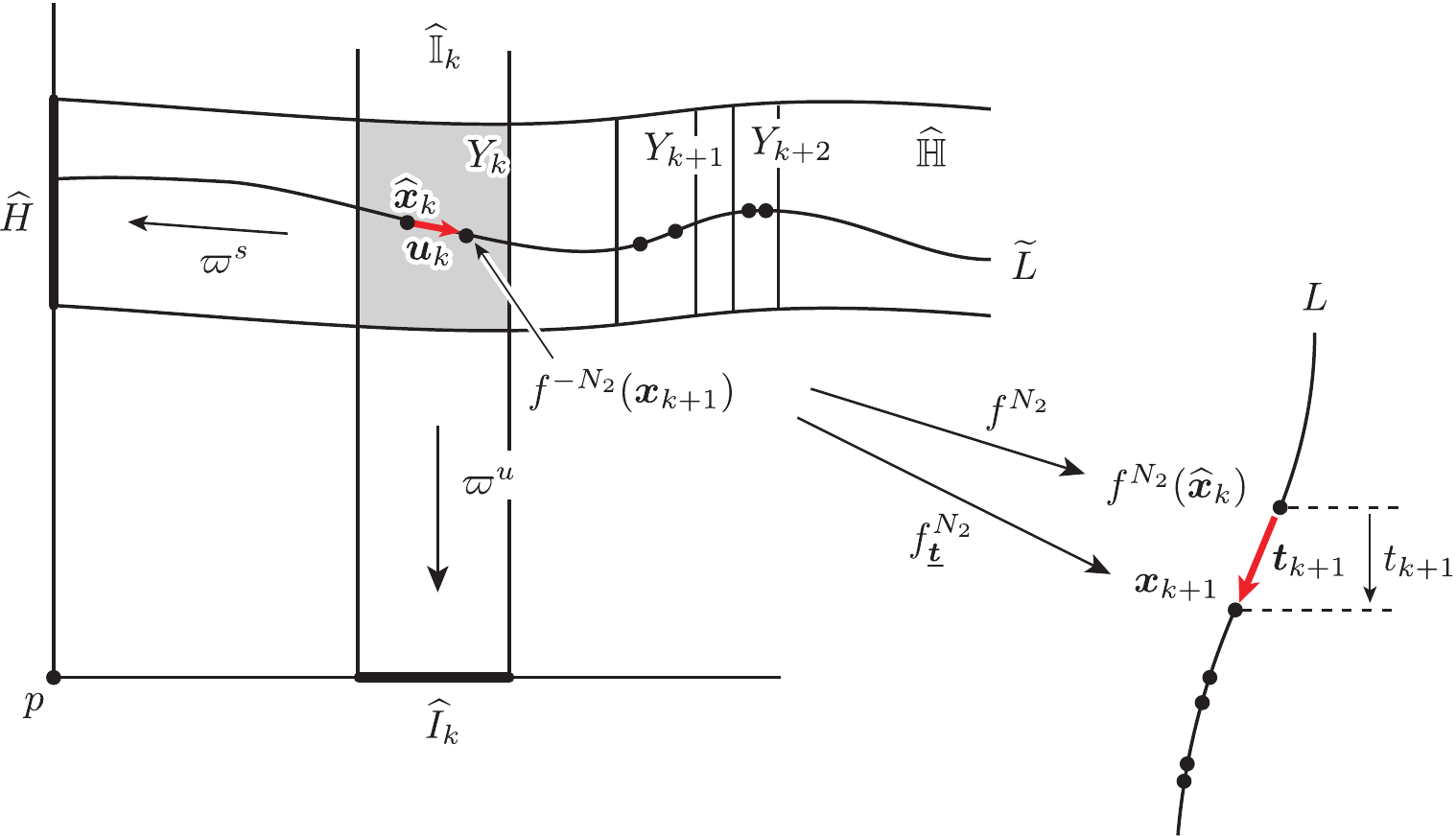}}
\caption{A shifting from $\widehat{\boldsymbol{x}}_k$ to $f^{-N_2}(\boldsymbol{x}_{k+1})$ and 
the resultant shifting from $f^{N_2}(\widehat{\boldsymbol{x}}_k)$ to $\boldsymbol{x}_{k+1}$.}
\label{fig_7_6}
\end{figure}

Now we define the bump functions $\theta_k$ and $\widetilde \theta$ supported on $\widehat I_k$ and $\widehat H$ respectively, where $\widehat I_k$ is the interval in $[0,2]\times \{0\}$ given in (B-\ref{B2}).
First, consider a  non-decreasing $C^{\infty}$ function on $\mathbb{R}$ with 
$$s(x)=
\begin{cases}
0 & \text{if }x\leq -1;\\
1 & \text{if }x\geq 0.
\end{cases}
$$
For $\rho>0$ and the interval $[a,b]$, let $S_{\rho, [a,b]}$ be the non-negative $C^{\infty}$ function on $\mathbb{R}$ defined as 
$$S_{\rho, [a,b]}(x):=s\left(\frac{x-a}{\rho(b-a)}\right)+s\left(\frac{b-x}{\rho(b-a)}\right)-1.$$
The support of $S_{\rho, [a,b]}$ is $[a-\rho(b-a), b+\rho(b-a)]$ and the height on $[a,b]$ is identical to $1$. 
Since $S_{\rho, [a,b]}$ is symmetric with respect to $x=\frac{a+b}2$, one has 
\begin{equation}\label{eqn_||S||}
\| S_{\rho, [a,b]} \|_{r}\leq \frac{1}{\bigl(\rho(b-a)\bigr)^{r}}\| s \|_{r},
\end{equation}
where $\|\cdot\|_{r}$ is the norm given by the derivative until order $r$.
Then our desired bump functions are defined by
\begin{equation}\label{eqn_theta_k}
\theta_{k}:=S_{\alpha_1/2, \varpi^u(I_{k})}, \quad \widetilde\theta:=S_{d, H},
\end{equation}
where $I_{k}$ is in $\widetilde L$  defined as (\ref{eqn_Ikf}) and $\alpha_1$ is the constant given in (B-\ref{B2}).

\begin{lemma}\label{lem.perturb_2}
Let $T$ be any real number with 
$$T>N_{u}r \log (5\cdot 2^{m-3})/\log(r_{s-}),$$
where 
$N_{u}$ is the integer given in Lemma \ref{lem_LG2}-(\ref{LG1-3}).
Then there exists a constant $C_T>0$ satisfying 
$$\lim_{T\to +\infty}C_{T}=0\quad\text{and}\quad  \sum_{k=1}^{\infty}\frac{\|\boldsymbol{u}_{k}(T)\|}{| {A}_{L,k+1}^{u}|^{r}}<C_{T}.$$
\end{lemma}

\begin{proof}
By  Lemmas \ref{lem.bdp2}-(\ref{bdp2_AA}) and \ref{lem_LG2}-(\ref{LG1-3}), we have 
$$
|A^{u}_{L,k}|\geq (5\cdot 2^{m-3})^{-i_{k}} |{A}^{u}_{L,0}|\quad\text{and}\quad
i_{k}\leq N_{u}k+i_0.
$$
Then  
$$
\frac{1}{|{A}^{u}_{L,k}|}<\frac{1}{(5\cdot 2^{m-3})^{-i_{k}} |{A}^{u}_{L,0}|}
<\frac{(5\cdot 2^{m-3})^{i_0}}{(5\cdot 2^{m-3})^{-N_{u}k }|{A}^{u}_{L,0}|}.
$$
By Lemma \ref{l_uvt}, one has
$$
\sum_{k=1}^{\infty}\frac{\|\boldsymbol{u}_{k}(T)\|}{| {A}_{L,k+1}^{u}|^{r}}
<
\frac{(5\cdot 2^{m-3})^{i_0+N_ur}\beta}{|{A}^{u}_{L,0}|^{r}}
\sum_{k=1}^{\infty}\left(\frac{(5\cdot 2^{m-3})^{N_{u} r }}{r_{s-}^T}\right)^{k}.
$$
Since
$T>N_{u}r \log (5\cdot 2^{m-3})/\log(r_{s-})$, the right-hand side of the inequality 
is equal to  
$$
C_{T}:=\frac{(5\cdot 2^{m-3})^{i_0+N_ur}\beta}{|A_{L,0}^u|^r(r_{s-}^T-(5\cdot 2^{m-3})^{N_ur})}.
$$
Since $r_{s-}\geq 2$, $\lim_{T\to \infty}C_T=0$.
\end{proof}

The square $S=[0,2]\times [0,2]$ is naturally supposed to be embedded in $\mathbb{R}^2$.
We may assume that the ambient surface $M$ has a Riemannian metric whose restriction on $S$ coincides 
with the standard Euclidean metric on $\mathbb{R}^2$.
The curvature of a leaf of $\mathcal{F}_{\mathrm{loc}}^u(\Lambda)$ (resp.\ 
$f^{-(N_0+N_2)}(\mathcal{F}_{\mathrm{loc}}^s(\Gamma_m))$) at $\boldsymbol{x}\in 
\widetilde L$ is denoted by $\kappa_{\Lambda}(\boldsymbol{x})$ (resp.\ $\kappa_{\Gamma_m}(\boldsymbol{x})$).
By (F-\ref{F3}) in Subsection \ref{subsec.Henon}, both $\kappa_{\Lambda}(\boldsymbol{x})$ and $\kappa_{\Gamma_m}(\boldsymbol{x})$ vary $C^1$ along $\widetilde L$. 
Since $\mathcal{F}_{\mathrm{loc}}^u(\Lambda)$ and $f^{-(N_0+N_2)}(\mathcal{F}_{\mathrm{loc}}^s(\Gamma_m))$ have 
quadratic tangencies along $\widetilde L$, there exists a constant $K>0$ with  
$|\kappa_{\Lambda}(\boldsymbol{x})-\kappa_{\Gamma_m}(\boldsymbol{x})|\geq K$ for 
any $\boldsymbol{x}\in \widetilde L$.
Moreover, by Lemma \ref{l_uvt}, 
\begin{equation}\label{eqn_kappa}
|\kappa_{\Lambda}(\widehat{\boldsymbol{x}}_k)-\kappa_{\Gamma_m}(\widehat{\boldsymbol{x}}_k+\boldsymbol{u}_k)|\geq K/2
\end{equation}
for all sufficiently large $k$.
Let $l_k$, $\widehat l_k$ be the leaves of $\mathcal{F}_{\mathrm{loc}}^u(\Lambda)$ passing though $\widehat{\boldsymbol{x}}_k$ and $\widehat{\boldsymbol{x}}_k+\boldsymbol{u}_k$ respectively, and 
let $-\pi\leq \omega_k\leq \pi$ be the angle of $l_k+\boldsymbol{u}_k$ and $\widehat l_k$ at $\widehat{\boldsymbol{x}}_k+\boldsymbol{u}_k$. 
See Figure \ref{fig_7_7}.
\begin{figure}[hbt]
\centering
\scalebox{0.8}{\includegraphics[clip]{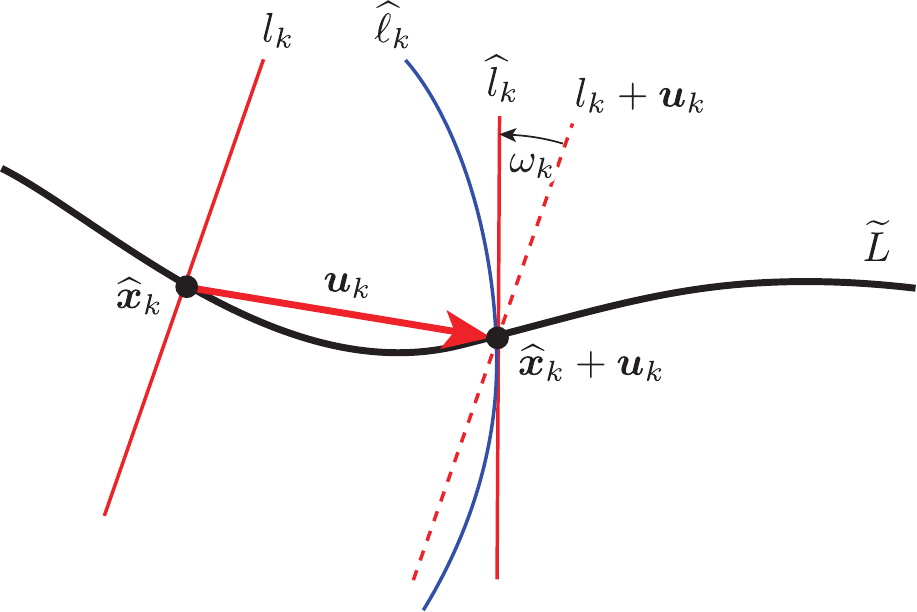}}
\caption{A small parallel translation and a small rotation.}
\label{fig_7_7}
\end{figure}
Since the directions of leaves of $\mathcal{F}_{\mathrm{loc}}^u(\Lambda)$ also vary $C^1$ along $\widetilde L$ 
by (F-\ref{F3}) in Subsection \ref{subsec.Henon}, 
there exists a constant $C>0$ independent of $k$ with 
\begin{equation}\label{eqn_omega_C_u}
|\omega_k|\leq C\|\boldsymbol{u}_k\|.
\end{equation}

Consider the orientation-preserving isometry $\xi_k:\mathbb{R}^2\to \mathbb{R}^2$ defined by
$$\xi_k(\boldsymbol{x})=A_k(\boldsymbol{x}-\widehat{\boldsymbol{x}}_k)+\widehat{\boldsymbol{x}}_k
+\boldsymbol{u}_k,$$
where $A_k$ is the orthogonal matrix of rotation $\omega_k$.
Since 
$$A_k-E=\begin{pmatrix}\cos\omega_k-1&-\sin\omega_k\\ \sin\omega_k&\cos\omega_k-1\end{pmatrix},$$
the inequality (\ref{eqn_omega_C_u}) implies that the $C^r$-norm of $A_k-E$ as a linear map is
\begin{equation}\label{eqn_DA-E}
\|A_k-E\|_r=\|A_k-E\|=O(\|\boldsymbol{u}_k\|),
\end{equation}
where $E$ is the unit matrix of order two.

It follows from the definition of $\xi_k$ that $\xi_k(l_k)$ is a curve tangent to $\widehat l_k$ at $\widehat{\boldsymbol{x}}_k+\boldsymbol{u}_k$ and hence to the leaf $\widehat\ell_k$ of 
$f^{-(N_0+N_2)}(\mathcal{F}_{\mathrm{loc}}^s(\Gamma_m))$ passing through $\widehat{\boldsymbol{x}}_k
+\boldsymbol{u}_k$.
Note that $\xi_k(l_k)$ is in general not identical to $\widehat l_k$.
Since any orientation-preserving isometry on $\mathbb{R}^2$ preserves curvature, by (\ref{eqn_kappa}), $\widehat{\boldsymbol{x}}_k
+\boldsymbol{u}_k$ is a quadratic tangency of $\xi_k(l_k)$ and $\widehat\ell_k$.
Since 
$\xi_k(\boldsymbol{x})-\boldsymbol{x}=(A_k-E)(\boldsymbol{x}-\widehat{\boldsymbol{x}}_k)+\boldsymbol{u}_k$ 
and $S$ is bounded, 
(\ref{eqn_DA-E}) implies 
\begin{equation}\label{eqn_upsilon_k}
\|(\xi_k-\mathrm{Id}_{\mathbb{R}^2})|_S\|_r=O(\|\boldsymbol{u}_k\|).
\end{equation}

The sequence $\underline{\boldsymbol{t}}=(\boldsymbol{t}_{2}, \boldsymbol{t}_{3}, \ldots, \boldsymbol{t}_{k},\ldots)$ of vectors with 
$f^{N_2}(\widehat{\boldsymbol x}_k)+\boldsymbol{t}_{k+1}=\boldsymbol{x}_{k+1}$ is 
called the \emph{perturbation sequence}.
Note that  $\underline{\boldsymbol{t}}$ depends on the constant $T$ given in Lemma \ref{lem_CC} and each entry $\boldsymbol{t}_k=\boldsymbol{t}_k(T)$ converges 
to the zero vector as $T\rightarrow \infty$.
Using the bump functions $\{\theta_{k}\}_{k\geq 1}$, $\widetilde\theta$ and 
the isometries $\xi_k(x,y)$, 
we define the sequence of the $C^r$-perturbation maps $\psi_{\underline{\boldsymbol{t}},a}:M\to M$ $(a=1,2,\dots)$ supported on the disjoint union $\bigcup_{k=1}^aY_k\subset S$ as  
\begin{equation}\label{eqn_psi_ta}
\psi_{\underline{\boldsymbol{t}},a}(\boldsymbol{x}):=\boldsymbol{x}+\sum_{k=1}^{a} \vartheta_k(\boldsymbol{x})(\xi_k(\boldsymbol{x})-\boldsymbol{x}),
\end{equation}
where
$$\vartheta_k(\boldsymbol{x})=\theta_{k}(\varpi^u(\boldsymbol{x})) \widetilde{\theta}(\varpi^s(\boldsymbol{x}))\quad\text{and}
\quad Y_k=\widehat{\mathbb{I}}_{k}\cap\ \widehat{\mathbb{H}}.$$
Each $Y_k$ is a rectangle with curvilinear top and bottom as illustrated in Figure \ref{fig_7_6}.

\begin{lemma}\label{lem.perturb}
For any $\underline{\boldsymbol{t}}=\underline{\boldsymbol{t}}(T)$, 
the sequence $\{\psi_{\underline{\boldsymbol{t}},a}\}_{a=1}^\infty$ $C^r$-converges to the $C^r$-map 
$\psi_{\underline{\boldsymbol{t}}}$ with 
\begin{equation}\label{eqn_psi_t}
\psi_{\underline{\boldsymbol{t}}}(\boldsymbol{x}):=\boldsymbol{x}+\sum_{k=1}^{\infty} \vartheta_k(\boldsymbol{x})(\xi_k(\boldsymbol{x})-\boldsymbol{x})
\end{equation}
for $(x,y)\in S$.
Moreover, $\psi_{\underline{\boldsymbol{t}}}=\psi_{\underline{\boldsymbol{t}}(T)}$ are $C^r$-diffeomorphisms on $M$ for all sufficiently large $T$ which $C^r$-converge  to the identity as $T\to \infty$. 
\end{lemma}

\begin{proof}
Recall that $I_{k+1}=A_{k+1}^u(\Delta)\cup B_{k+1}^s(\Delta)$ is the arc of (\ref{eqn_Ikf}) in $\widetilde L$.
Since $N_2$ is independent of $k$, 
it follows from $|A_{k+1}^u(\Delta)|\geq |B_{k+1}^s(\Delta)|$ that
$$|I_{k+1}|= |A_{k+1}^u(\Delta)\cup B_{k+1}^s(\Delta)|\asymp |A_{k+1}^u(\Delta)|\asymp |A_{L,k+1}^u|.$$ 
By this fact together with (\ref{eqn_||S||}) and (\ref{eqn_upsilon_k}), for any integers $a,b$ with $1\leq a<b$,
\begin{equation}\label{def.perturb_2}
\|
\psi_{\underline{\boldsymbol{t}},b}-\psi_{\underline{\boldsymbol{t}},a}\|_{r}
\leq
C_0 \| S_{\rho, [-1,1]} \|_{r}\sum_{k=a+1}^{b}\frac{\|\boldsymbol{u}_{k}\|}{| {A}_{L,k+1}^{u}|^{r}},
\end{equation}
where $C_0$ is a constant independent of $\underline{\boldsymbol{t}}$.
By Lemma \ref{lem.perturb_2}, $\{\psi_{\underline{\boldsymbol{t}},a}\}_{a=1}^\infty$ is a Cauchy sequence in 
the space $(\mathrm{Map}^r(M),\|\cdot\|_r)$ of $C^r$-maps on $M$, which is a complete metric space.
Thus $\psi_{\underline{\boldsymbol{t}},a}$ $C^r$-converges to the $C^r$-map 
$\psi_{\underline{\boldsymbol{t}}}=\psi_{\underline{\boldsymbol{t}}(T)}$ defined by (\ref{eqn_psi_t}) as $a\to \infty$.
Again by Lemma \ref{lem.perturb_2}, 
$$
\|
\psi_{\underline{\boldsymbol{t}}(T)}-\mathrm{id}_M\|_{r}
\leq
C_0 \| S_{\rho, [-1,1]} \|_{r}\sum_{k=1}^{\infty}\frac{\|\boldsymbol{u}_{k}\|}{| {A}_{L,k+1}^{u}|^{r}}
\leq C_0C_T.
$$
Since $\lim_{T\to \infty}C_T=0$, the map $\psi_{\underline{\boldsymbol{t}}(T)}$ $C^r$-converges to the identity as $T\to \infty$.
Since the identity is a diffeomorphism, $\psi_{\underline{\boldsymbol{t}}(T)}$ is also a diffeomorphism for all sufficiently large $T$.
This completes the proof. 
\end{proof}

\begin{remark}\label{r_Jk}
We note that $\widetilde L$ is no longer a tangency curve of $\psi_{\underline{\boldsymbol{t}}}(\mathcal{F}_{\mathrm{loc}}^u(\Lambda))$ and 
$f^{-(N_0+N_2)}(\mathcal{F}_{\mathrm{loc}}^s(\Gamma_m))$.
However, from our construction (\ref{eqn_psi_t}) of $\psi_{\underline{\boldsymbol{t}}}$, the leaves of them passing through $\widehat{\boldsymbol{x}}_k+\boldsymbol{u}_k$ still 
have a quadratic tangency.
In fact, (\ref{eqn_theta_k}) implies $\vartheta_k(\boldsymbol{x})=1$ and $\vartheta_{k'}(\boldsymbol{x})=0$ $(k'\neq k)$ for any $\boldsymbol{x}\in S$ sufficiently close to $\widehat{\boldsymbol{x}}_k$, 
and hence 
$$\psi_{\underline{\boldsymbol{t}}}(\boldsymbol{x})=\boldsymbol{x}+(\xi_k(\boldsymbol{x})-\boldsymbol{x})
=\xi_k(\boldsymbol{x}).$$
It follows that $\psi_{\underline{\boldsymbol{t}}}(l_k)$ and $\widehat{\ell}_k$ have a quadratic 
tangency at $\widehat{\boldsymbol{x}}_k+\boldsymbol{u}_k$.

Let $J_k$ be a short segment in $U(L)$ which is vertical with respect to 
the $C^{1+\alpha}$-coordinate on $U(L)$ given in Subsection \ref{subsec.LP} and passes through 
$\boldsymbol{x}_k$, for example see Figure \ref{fig_8_1}.
From the definition of the foliation $\mathcal{F}_{\mathrm{loc}}^u(\Lambda)$, 
we know that $h_k(J_k)$ is a segment intersecting $\widetilde L$ transversely at $\widehat{\boldsymbol{x}}_k$ and 
contained in a leaf $l_k$ of $\mathcal{F}_{\mathrm{loc}}^u(\Lambda)$, where 
$h_k$ is the diffeomorphism defined as (\ref{eqn_h_k}).
Thus $\psi_{\underline{\boldsymbol{t}}}(h_k(J_k))$ and some leaf of $f^{-(N_0+N_2)}(\mathcal{F}_{\mathrm{loc}}^s(\Gamma_m))$ has a quadratic tangency at 
$\widehat{\boldsymbol{x}}_k+\boldsymbol{u}_k$.
\end{remark}

From the definition (\ref{eqn_theta_k}) of $\theta_{k}$, $\widetilde{\theta}$ and the form (\ref{eqn_psi_t}) of $\psi_{\underline{\boldsymbol{t}}}$, 
we know that the support $\mathrm{supp}(\psi_{\underline{\boldsymbol{t}}})$ of $\psi_{\underline{\boldsymbol{t}}}$ is contained in the disjoint union $\bigcup_{k=1}^\infty Y_k$.
We now define the $C^r$-map $f_{\underline{\boldsymbol{t}}}$ by
\begin{equation}\label{def_f_mu,t}
f_{\underline{\boldsymbol{t}}}:=f\circ \psi_{\underline{\boldsymbol{t}}}, 
\end{equation}
for the perturbation sequence  $\underline{\boldsymbol{t}}=(\boldsymbol{t}_{2}, \boldsymbol{t}_{3}, \ldots, \boldsymbol{t}_{k}, \ldots)$.
Since $\mathrm{supp}(\psi_{\underline{\boldsymbol{t}}})$ is contained in a small region in $S$ sufficiently 
close to the left component of $\flat S$, $\mathrm{supp}(\psi_{\underline{\boldsymbol{t}}})$ is 
disjoint from $\bigcup_{i=1}^{N_2}f^i (\mathrm{supp}(\psi_{\underline{\boldsymbol{t}}}))$, 
see Figure \ref{fig_3_2}.
It follows that $f_{\underline{\boldsymbol{t}}}$ is equal to $f$ on $\bigcup_{i=1}^{N_2}f^i (\mathrm{supp}(\psi_{\underline{\boldsymbol{t}}}))$.
Hence, by (\ref{eqn_uk}) and (\ref{eqn_psi_t}), 
\begin{equation}\label{eqn_fxx}
f_{\underline{\boldsymbol{t}}}^{N_2}(\widehat{\boldsymbol{x}}_k)=
f_{\underline{\boldsymbol{t}}}^{N_2-1}\circ (f\circ \psi_{\underline{\boldsymbol{t}}})(\widehat{\boldsymbol{x}}_k)
=f^{N_2-1}\circ f(\widehat{\boldsymbol{x}}_k+{\boldsymbol{u}}_k)=\boldsymbol{x}_{k+1},
\end{equation}
see Figure \ref{fig_7_6}.
By Lemma \ref{lem.perturb}, one can suppose that $f_{\underline{\boldsymbol{t}}}$ is a $C^r$-diffeomorphism arbitrarily $C^{r}$-close to $f$ if we take 
$T$ sufficiently large.

\begin{remark}\label{rem_disjoint}
(1)
Since $\bigcup_{k=1}^\infty Y_k\subset \widehat{\mathbb{I}}$, by (B-\ref{B1}) the support $\mathrm{supp}(\psi_{\underline{\boldsymbol{t}}})$ is disjoint from 
$X_{n_*}$.
Thus both the invariant set $\Gamma_m$ and local stable foliation $\mathcal{F}_{\mathrm{loc}}^s(\Gamma_m)$ for 
the perturbed diffeomorphism $f_{\underline{\boldsymbol{t}}}$ are the same as the originals. 
\smallskip

\noindent(2)
It is possible to rearrange our construction so that both $f$ and $\psi_{\underline{\boldsymbol{t}},a}$ are of $C^\infty$-class.
However, even in the case, Lemma \ref{lem.perturb} only asserts that $\psi_{\underline{\boldsymbol{t}}}$ is of $C^r$-class with $2\leq r<\infty$ but not necessarily of $C^\infty$-class.
In fact, though $\psi_{\underline{\boldsymbol{t}}}$ is a $C^\infty$-map on the 
neighborhood $Y_k$ of any $\widehat{\boldsymbol{x}}_k$, the authors do not know whether  
it holds also at the limit point of $\{\widehat{\boldsymbol{x}}_k\}$.
So it may be impossible to suppose that the composition $f_{\underline{\boldsymbol{t}}}:=f\circ \psi_{\underline{\boldsymbol{t}}}$ of (\ref{def_f_mu,t}) is a $C^\infty$-diffeomorphism.
The same problem would occur in the original example in \cite{CV01} though they assert that 
it is of class $C^\infty$.
\end{remark}

\section{Detection of wandering domains with historic behavior}\label{sec.DWD}

\subsection{Some constants for wandering domains}\label{subsec.DWD.preliminary}
From now on, we write $f_{\underline{\boldsymbol{t}}}=f$ for short. 
Recall that $\sigma$ is the unstable eigenvalue of the derivative $Df_p$ of $f$ at $p$ given in (S-\ref{setting3}).
Define the constant $\omega$ by
\begin{equation}\label{eqn_omega}
\omega=\max\bigl\{\nu^{-1},\ \sup\bigl\{\|Df_{\boldsymbol{x}}\|\,;\,\boldsymbol{x}\in S\bigr\},\  
\sup\bigl\{\|(Df_{\boldsymbol{x}})^{-1}\|\,;\,\boldsymbol{x}\in S\bigr\}\bigr\},
\end{equation}
where $\nu$ is the constant given in Lemma \ref{l_[0,2]}.
In this section, the constant  
\begin{equation}\label{def.b_k}
b_{k}:=\varepsilon (5\cdot 2^{m-3})^{-p_k}\omega^{-q_k} \sigma^{-r_k}
\end{equation}
plays an important role, where `$5\cdot 2^{m-3}$' is the number presented in Lemma \ref{lem.bdp2}-(\ref{bdp2_AA}) and
\begin{equation}\label{eqn_pkqkrk}
p_k=\sum_{i=0}^{\infty}\frac{\widehat{i}_{k+i}}{2^{i}},\quad
q_k=\sum_{i=0}^{\infty}\frac{\langle k+i\rangle}{2^{i}},\quad r_k=\sum_{i=0}^{\infty}\frac{z_{k+i}(k+i)^2}{2^{i}},
\end{equation}
where $\widehat{i}_{k+i}$ is the constant given in Lemma \ref{lem_CC}-(\ref{CC3}).
The factor $\varepsilon$ in (\ref{def.b_k}) is a positive constant independent of $k$ which will be fixed later.
Remember that each $z_k$ is either $z_0$ or $z_0+1$ for the integer $z_0$ given in Subsection \ref{subsec.CCC} 
and $\langle k+i\rangle=\widehat n_{k+1+i}+(k+i)^2+k+i$ by (\ref{eqn_lakra}).

\begin{lemma}\label{lem.preliminary1}
\begin{enumerate}[\rm (1)]
\item  \label{preliminary1.2}
For any $\eta>0$, suppose that $z_0$ satisfies (\ref{eqn_z0+1}).
Then there exists an integer $k_*>0$ such that, for any $k\geq k_*$, 
\begin{align*}
&b_{k}^{\frac{1}{2}}
\leq 
\varepsilon^{-\frac{1}{2}}
(5\cdot 2^{m-3})^{-\frac{\widehat i_k}2+\frac34 p_{k+1}}
\omega^{-\frac{\langle k\rangle}2 +\frac34 q_{k+1}}
 \sigma^{(1+\eta)z_kk^2}
b_{k+1},\\
&\frac{\langle k\rangle}2+\frac34 q_{k+1}<4k^2.
\end{align*}
\item  \label{preliminary1.3}
For any integer $k>0$,
$$
b_{k+1}=\varepsilon^{-1} (5\cdot 2^{m-3})^{2\widehat{i}_{k}} \omega^{2\langle k\rangle}\sigma^{2z_kk^2}b_{k}^{2}.
$$
\end{enumerate}
\end{lemma}

\begin{proof}
(\ref{preliminary1.2})
By (\ref{def.b_k}),
$$
\frac{b_{k}^{\frac{1}{2}}}{b_{k+1}}
=\varepsilon^{-\frac{1}{2}}
(5\cdot 2^{m-3})^{p_{k+1}-\frac{p_k}2}
\omega^{q_{k+1}-\frac{q_k}2}
 \sigma^{r_{k+1}-\frac{r_k}2}.
$$
It is immediate from (\ref{eqn_pkqkrk}) that $p_{k+1}-\frac{p_k}2=-\frac{\widehat i_k}2+\frac34 
 p_{k+1}$ and $q_{k+1}-\frac{q_k}2=-\frac{\langle k\rangle}2+\frac34 
 q_{k+1}$.
 
Since $f(x)=\frac{4}{2-(1+x)^2}-\frac1{1+x}-2$ is a monotone increasing function on $0\leq x<\sqrt2-1$ 
with $f(0)=1$ and $\lim_{x\nearrow \sqrt2-1}f(x) =\infty$, for any $\eta>0$, there exists a unique 
$0<\eta_1<\sqrt2-1$ with
$$\frac{4}{2-(1+\eta_1)^2}-\frac1{1+\eta_1}-2=1+\eta.$$
Since $\lim_{k\to \infty}\frac{(k+1)^2}{k^2}=1$, there exists an integer $k_*>0$ such that 
$(k+1)^2\leq (1+\eta_1)k^2$ for any $k\geq k_*$.
We choose the integer $z_0$ so as to satisfy
\begin{equation}\label{eqn_z0+1}
z_0+1\leq (1+\eta_1)z_0.
\end{equation}
Then $z_{k+1}(k+1)^2\leq  (z_0+1)(k+1)^2\leq (1+\eta_1)^2z_kk^2$ holds.
Repeating a similar argument, one can have $z_{k+1+i}(k+1+i)^2\leq z_k(1+\eta_1)^{2(i+1)}k^2$ for any $k\geq k_*$ and $i\geq 0$.
Since $0<\left(\frac{1+\eta_1}2\right)^2<1$, 
$$r_{k+1}\leq \left(\sum_{i=0}^\infty\frac{(1+\eta_1)^{2(i+1)}}{2^i}\right)z_kk^2
=\left(\frac{4}{2-(1+\eta_1)^2}-2\right)z_kk^2.$$
On the other hand, since $r_k\geq \sum_{i=0}^\infty \frac1{2^i}\frac{z_k}{1+\eta_1}k^2=\frac{2}{1+\eta_1}z_kk^2$, 
$$r_{k+1}-\frac{r_k}2\leq \left(\frac{4}{2-(1+\eta_1)^2}-2-\frac1{1+\eta_1}\right)z_kk^2=(1+\eta)z_kk^2.$$
This shows the first inequality of (\ref{preliminary1.2}).

Since $\widehat n_{k+1}\asymp k$ by Lemma \ref{lem_CC}-(\ref{CC3}), one can choose $k_*$ 
so that $\langle k\rangle=\widehat n_{k+1}+k^2+k\leq 2k^2$ for any $k\geq k_*$.
Since 
$$q_{k+1}-\sum_{i=0}^\infty \frac{k^2}{2^i}=\sum_{i=0}^{\infty}\frac{\langle k+1+i\rangle}{2^{i}}
-\sum_{i=1}^\infty \frac{k^2}{2^i}\asymp k,$$
we may assume that 
$q_{k+1}<2\sum_{i=0}^\infty \frac{k^2}{2^i}=4k^2$ for any $k\geq k_*$.
It follows that
$$\frac{\langle k\rangle}2+\frac34 q_{k+1}<k^2+\frac34\cdot 4k^2=4k^2.$$
This shows the second inequality of (\ref{preliminary1.2}).

\noindent(\ref{preliminary1.3})
The equality of (\ref{preliminary1.3}) is derived immediately from (\ref{def.b_k}) together 
with the equalities 
$$2p_k-p_{k+1}=2\widehat i_k,\quad 2q_k-q_{k+1}=2\langle k\rangle,\quad 
2r_k-r_{k+1}=2z_kk^2.$$
This completes the proof.
\end{proof}

Since $z_0$ is required to satisfy (\ref{eqn_z0+1}), we need to choose $z_0$ sufficiently large 
according as $\eta>0$ is taken small. 
By (S-\ref{setting3}) in Section \ref{sec.prep}, $\lambda\sigma<1$.
We take and fix a sufficiently small $\eta>0$ with $\lambda\sigma^{1+\eta}<1$.
So one can choose $z_0$ so as to satisfy 
\begin{equation}\label{eqn_eta}
\omega^4(\lambda\sigma^{1+\eta})^{z_0}<1.
\end{equation}
Since $\nu^{-1}\leq \omega$, it is not hard to show that the condition (\ref{eqn_eta}) implies (\ref{eqn_lambdaz0nu}).
We will see later that the superscript `4' of $\omega$ corresponds to the coefficient `4' of 
$k^2$ in the second inequality of Lemma \ref{lem.preliminary1}-(\ref{preliminary1.2}).

\subsection{Critical chain of rectangles}
For every $k\geq 1$, let $\boldsymbol{x}_k=(x_k,y_k)$ be the intersection point of 
$L_k$ and $L$ given in Subsection \ref{subsec.CCC}.
We consider the rectangle 
$$R_k=\bigl[x_k-b_k^{\frac12},x_k+b_k^{\frac12}\bigr]\times [y_k-b_k,y_k+b_k]$$
centered at $\boldsymbol{x}_k$ with respect to the $C^{1+\alpha}$-coordinate on $U(L)$ defined in 
Subsection \ref{subsec.LP}, see Figures \ref{fig_5_4} and \ref{fig_7_3}.
The absolute slope of the diagonals of $R_k$ is $b_k^{\frac12}$.
By (\ref{def.b_k}), there exists a constant $0<\gamma<1$ such that 
$b_k^{\frac12}>\gamma^k \omega^{-k^2}\sigma^{-z_kk^2}$.
By Lemma \ref{l_slope} and (\ref{eqn_eta}), 
$$\frac{\mathrm{slope}(L_k)}{b_k^{\frac12}}<\frac{\alpha^k(\sigma^{-1}\lambda)^{z_kk^2}}{\gamma^k\omega^{-k^2}\sigma^{-z_kk^2}}=(\alpha\gamma^{-1})^k(\omega\lambda^{z_k})^{k^2}\rightarrow +0\quad 
(k\rightarrow \infty).$$
Thus one can suppose that $L_k\cap R_k$ is an arc in $R_k$ passing through $\boldsymbol{x}_k$ and well approximating the horizontal line $y=y_k$ if $k\geq k_*$.
In particular, the arc connects the components of the edge $\flat R_k$.
The strip $\mathbb{B}_{k}^u$ is divided by $\widetilde L_k$ into the two strips $\mathbb{B}_{k}^{u\pm}$.
Similarly, $R_k$ is divided by 
$L_k$ into the two strips $R_k^{\pm}$.
See Figure \ref{fig_7_3} again.
By Lemma \ref{lem.bdp2}-(1), one can have a constant $C>0$ satisfying
$$\mathrm{dist}\bigl(f^{N_1}\circ f^{\widehat i_k}\circ f^{N_0}(l_k^\pm),\widetilde L_k\bigr)\leq C(5\cdot 2^{m-3})^{\widehat i_k}b_k,$$
where $l_k^\pm=\sharp R_k\cap R_k^\pm$.
By (\ref{eqn_omega}) and (\ref{def.b_k}), 
$$(5\cdot 2^{m-3})^{\widehat i_k}b_k\leq 
(5\cdot 2^{m-3})^{\widehat i_k}\varepsilon (5\cdot 2^{m-3})^{-\widehat i_k}\omega^{-k^2}\sigma^{-z_kk^2}
\leq \varepsilon \nu^{k^2}\sigma^{-z_kk^2}.$$
It follows from Lemma \ref{l_[0,2]} that $f^{N_1}\circ f^{\widehat i_k}\circ f^{N_0}(R_k^\pm)\subset \mathbb{B}_{k}^{u\pm}\cap \mathcal{N}_k$ holds for any $k\geq k_*$ 
if we take $\varepsilon >0$ sufficiently small.
In particular, this implies that $h_k(R_k)\subset \mathbb{B}_{k+1}^{s*}\cap \widehat{\mathbb{H}}$, 
where $h_k$ is the diffeomorphism defined as (\ref{eqn_h_k}).
See Figure \ref{fig_7_3} for $\mathbb{B}_{k+1}^{s*}$ and Figure \ref{fig_7_6} for 
$\widehat{\mathbb{H}}$.

For any integer $k>0$, consider the composition  
\begin{equation}\label{def.transition}
g_{k}:= f^{N_{2}}\circ f^{z_kk^2+\langle k\rangle}\circ f^{N_{1}}\circ f^{\widehat{i}_{k}}\circ f^{N_{0}}=f^{N_2}\circ h_k.
\end{equation}
By (\ref{eqn_fxx}),  
$f^{N_2}(\widehat{\boldsymbol{x}}_{k})=\boldsymbol{x}_{k+1}$ and hence 
$g_k(\boldsymbol{x}_k)=\boldsymbol{x}_{k+1}$.
Moreover $f$ is chosen so that 
$f^{N_2}(\mathcal{F}_{\mathrm{loc}}^u(\Lambda))$ and $f^{-N_0}(\mathcal{F}_{\mathrm{loc}}^s(\Gamma_m))$ have a 
quadratic tangency at $\boldsymbol{x}_{k+1}$.

For sequences $\{u_k(\varepsilon)\}$, $\{v_k(\varepsilon)\}$ of positive numbers, $u_k(\varepsilon)\prec v_k(\varepsilon)$ means that
there exists a positive constant $a$ independent of $k$, $\varepsilon$ and satisfying  
$u_k(\varepsilon)<a v_k(\varepsilon)$ for all $k$ and $\varepsilon>0$.

The main result of this section is as follows:

\begin{lemma}[Rectangle Lemma]\label{lem_RL}
There exist an integer $k_0\geq k_*$ and $\varepsilon >0$ such that, for any $k\geq k_0$, $g_k(R_k)\subset \mathrm{Int} R_{k+1}$.
\end{lemma}

\begin{proof}
For any $\boldsymbol{x}\in U(L)$, let $L_{\boldsymbol{x}}^{\mathrm{hori}}$ and 
$L_{\boldsymbol{x}}^{\mathrm{vert}}$ be the horizontal and vertical lines in $U(L)$ passing through $\boldsymbol{x}$ respectively.
See Figure \ref{fig_5_4}.
Let 
$$\pi_{\boldsymbol{x}}^{\mathrm{hori}}:U(L)\to L_{\boldsymbol{x}}^{\mathrm{vert}},\quad \pi_{\boldsymbol{x}}^{\mathrm{vert}}:U(L)\to L_{\boldsymbol{x}}^{\mathrm{hori}}$$
be the projections along the 
horizontal and vertical lines on the $C^{1+\alpha}$-coordinate.

First we show that the $g_k$-image of the center vertical segment $J_k=\{x_k\}\times [y_k-b^k,y_k+b_k]$ of $R_k$ is contained in 
$$\tfrac12 R_{k+1}=
\bigl[x_{k+1}-\tfrac12 b_{k+1}^{\frac12},x_{k+1}+\tfrac12 b_{k+1}^{\frac12}\bigr]\times \bigl[y_{k+1}-\tfrac12 b_{k+1},y+\tfrac12 b_{k+1}\bigr].$$
See Figure \ref{fig_8_1}.
\begin{figure}[hbt]
\centering
\scalebox{0.7}{\includegraphics[clip]{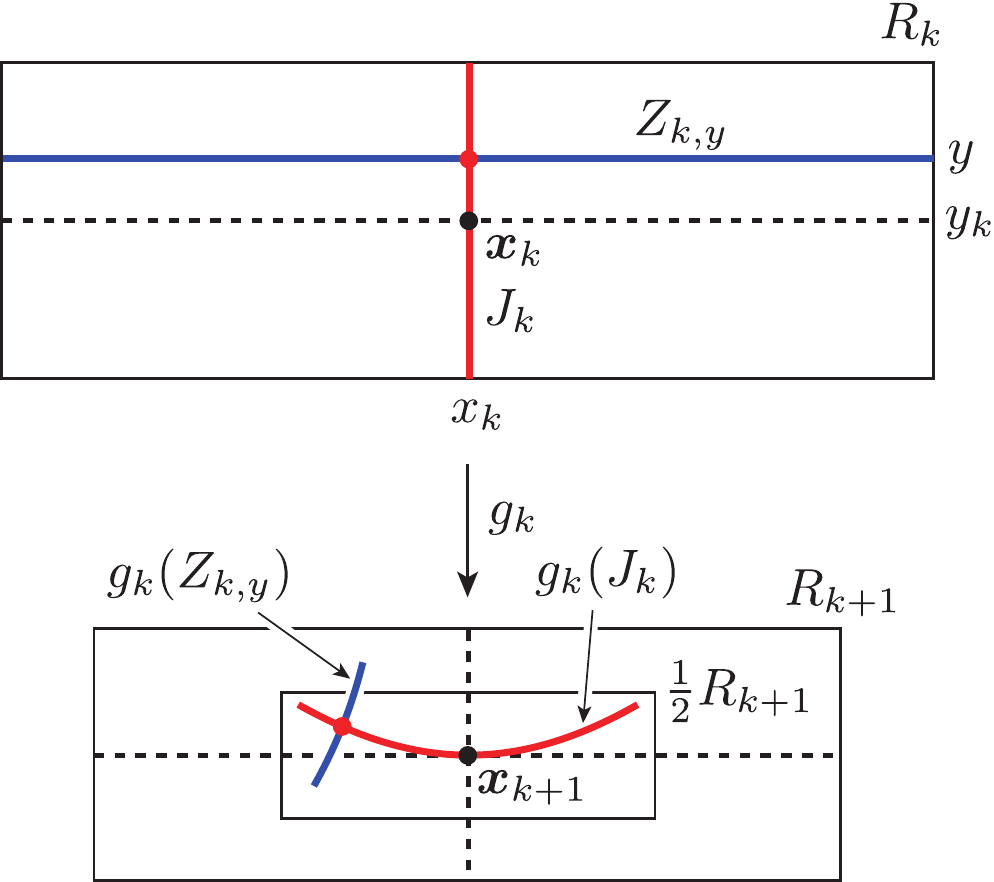}}
\caption{A transition from $R_k$ into $R_{k+1}$.}
\label{fig_8_1}
\end{figure}
We set $\widetilde J_{k}=h_k(J_{k})$.
By Remark \ref{r_Jk},  
$\widetilde J_k$ is in a leaf of $\mathcal{F}_{\mathrm{loc}}^u(\Lambda)$.
Note that, by (\ref{def.b_k}), $b_k$ and hence $|J_{k}|$ depend on $\varepsilon$.
Since $f^{z_kk^2}$ coincides with the linear map $(x,y)\mapsto (\lambda^{z_kk^2}x,\sigma^{z_kk^2}y)$ on 
$f^{N_1}\circ f^{\widehat i_k}\circ f^{N_0}(R_k)$ and $N_0$, $N_1$ are independent of $k$,
$$|\widetilde J_{k}|\prec (5\cdot 2^{m-3})^{\widehat i_k}\omega^{\langle  k\rangle}\sigma^{z_kk^2}b_k.$$
Since $g_k(J_{k})=f^{N_2}(\widetilde J_{k})$, by Lemma \ref{lem.preliminary1}-(\ref{preliminary1.3}), we have
$$
|\pi_{\boldsymbol{x}_{k+1}}^{\mathrm{vert}}(g_k(J_{k}))|\prec (5\cdot 2^{m-3})^{\widehat i_k}
\omega^{\langle k\rangle}\sigma^{z_kk^2}b_k=\varepsilon^{\frac12}b_{k+1}^{\frac12}.
$$
By Remark \ref{r_Jk}, $g_k(J_k)$ and $L_{\boldsymbol{x}_{k+1}}^{\mathrm{hori}}$ have a quadratic tangency at $\boldsymbol{x}_{k+1}$.
Since the curvatures of leaves of $f^{N_2}(\mathcal{F}_{\mathrm{loc}}^u(\Lambda))$ and  $f^{-N_0}(\mathcal{F}_{\mathrm{loc}}^s(\Gamma_m))$ vary $C^1$ along $L$ 
by (F-\ref{F3}) in Subsection \ref{subsec.Henon}, 
there exists a constant $A>0$ independent of $k$ with 
$|\pi_{\boldsymbol{x}_{k+1}}^{\mathrm{hori}}(g_k(J_{k}))|\leq A
|\pi_{\boldsymbol{x}_{k+1}}^{\mathrm{vert}}(g_k(J_{k}))|^2$.
This implies that
$$
|\pi_{\boldsymbol{x}_{k+1}}^{\mathrm{hori}}(g_k(J_{k}))|\prec 
\varepsilon b_{k+1}.
$$
One can choose $\varepsilon>0$ so that
$$|\pi_{\boldsymbol{x}_{k+1}}^{\mathrm{vert}}(g_k(J_{k}))|\leq \frac12 b_{k+1}^{\frac12}\quad\text{and}\quad
|\pi_{\boldsymbol{x}_{k+1}}^{\mathrm{hori}}(g_k(J_{k}))|\leq \frac12 b_{k+1}.$$
It follows that $g_k(J_{k})\subset \frac12 R_{k+1}$.

Consider next the horizontal segment $Z_{k,y}=\bigl[x_k-b_k^{\frac12},x_k+b_k^{\frac12}\bigr]\times\{y\}$ in 
$R_k$ with $y_k-b_k\leq y \leq y_k+b_k$.
Note that the intersection $Z_{k,y}\cap J_k$ consists of the single point $(x_k,y)$.
By our constructions of the unstable foliation $\mathcal{F}_{\mathrm{loc}}^u(\Gamma_m)$ in 
Subsection \ref{subsec.Hetero} and the horizontal coordinate on $U(L)$ in Subsection \ref{subsec.LP}, 
the image $f^{N_1}\circ f^{\widehat i_k}\circ f^{N_0}(Z_{k,y})$ is a strictly horizontal segment in 
$S$ with 
$$|f^{N_1}\circ f^{\widehat i_k}\circ f^{N_0}(Z_{k,y})|\prec
(5\cdot 2^{m-3})^{\widehat i_k}b_k^{\frac12}.$$
Since $f^{z_kk^2}$ is the linear map $(\sigma^{z_kk^2}x,\lambda^{z_kk^2}y)$ in a small neighborhood 
of $f^{N_1}\circ f^{\widehat i_k}\circ f^{N_0}(Z_{k,y})$ in $S$, 
the curve $\widetilde Z_{k,y}=h_k(Z_{k,y})$ satisfies
$$|\widetilde Z_{k,y}|\prec (5\cdot 2^{m-3})^{\widehat i_k}\omega^{\langle  k\rangle}\lambda^{z_kk^2}b_k^{\frac12}.$$
By Lemma \ref{lem.preliminary1}-(\ref{preliminary1.2}), 
\begin{align*}
|g_k(Z_{k,y})|&\prec
(5\cdot 2^{m-3})^{\widehat i_k}\omega^{\langle  k\rangle}\lambda^{z_kk^2}b_k^{\frac12}\\
&\leq  
\varepsilon^{-\frac{1}{2}}
(5\cdot 2^{m-3})^{\frac{\widehat i_k}2+\frac34p_{k+1}}
\omega^{\frac{\langle k\rangle}2+\frac34 q_{k+1}}
 \bigl(\lambda\sigma^{1+\eta}\bigr)^{z_kk^2}
b_{k+1}\\
&<  
\varepsilon^{-\frac{1}{2}}
(5\cdot 2^{m-3})^{\frac{\widehat i_k}2+\frac34p_{k+1}}
\omega^{4k^2}
 \bigl(\lambda\sigma^{1+\eta}\bigr)^{z_kk^2}
b_{k+1}.
\end{align*}
By Lemma \ref{lem_CC}-(\ref{CC3}) and (\ref{eqn_pkqkrk}), $\frac{\widehat i_k}2+\frac34 p_{k+1}\asymp k$.
Since $\omega^4(\lambda\sigma^{1+\eta})^{z_k}<1$ by (\ref{eqn_eta}), 
there exists an integer $k_0\geq k_*$ such that, for any $k\geq k_0$, 
$|g_k(Z_{k,y})|<\frac12 b_{k+1}$.
Since $g_k(Z_{k,y})\cap g_k(J_k)\neq \emptyset$ and $g_k(J_k)\subset \frac12 R_{k+1}$, 
$g_k(Z_{k,y})$ is contained in $\mathrm{Int} R_{k+1}$ for all $y\in [y_k-b_k,y_k+b_k]$.
It follows that $g_k(R_k)\subset \mathrm{Int} R_{k+1}$.
This completes the proof.
\end{proof}

\subsection{Proof of Theorem \ref{main-thm}}
Now we will present the proof of Theorem \ref{main-thm} under the notation same as in the preceding subsection.

\begin{proposition}\label{prop_exist_wd}
Let $f$ be a $C^r$ diffeomorphism on $M$ contained in a Newhouse open set $\mathcal{N}$ of $\mathrm{Diff}^r(M)$.
Then there exist $C^r$-diffeomorphisms $f_n$ on $M$ which admit  wandering domains 
and $C^r$-converge to $f$.
\end{proposition}
\begin{proof}
Let $f_{\underline{\boldsymbol{t}}}$ be the diffeomorphism given in (\ref{def_f_mu,t}).
By Lemmas \ref{lem.perturb} and \ref{lem_RL}, 
there exists a perturbation sequence  
$\underline{\boldsymbol{t}_n}=(\boldsymbol{t}_{n,2},\boldsymbol{t}_{n,3},\dots)$ such that 
$f_n:=f_{\underline{\boldsymbol{t}_n}}$ $C^r$-converge to $f$.
Then the interior $D$ of the rectangle $R_{k_0}$ given in Lemma \ref{lem_RL} is a wandering domain of $f_n$.
If not, there would exist an $a\in \mathbb{N}$ such that $f_n^a(D)\cap D\neq \emptyset$.
Take an integer $k$ with $k\geq k_0$ and $z_kk^2>a$.
By (\ref{eqn_CC1}), there exists an integer  $q\geq 1$ such that 
$$f_n^{q+z_kk^2}(D)\subset \mathbb{B}^u(\langle k\rangle;\underline{2}^{(k^2)}\,\underline{v}_{\,k+1}\,[\widehat{\underline{w}}_{\,k+1}]^{-1})
\subset \mathbb{B}^u(1;2),$$
see Figure \ref{fig_4_1}.
Moreover we have
$$f^{q+z_kk^2-a}(D)\subset \mathbb{B}^u(
\langle k\rangle+a;\underline{1}^{(a)}\underline{2}^{(k^2)}\,\underline{v}_{\,k+1}\,[\widehat{\underline{w}}_{\,k+1}]^{-1}
)\subset \mathbb{B}^u(1;1).$$
Hence the intersection $f_n^{q+z_kk^2-a}(f_n^a(D)\cap D)
=f_n^{q+z_kk^2}(D)\cap f_n^{q+z_kk^2-a}(D)$ is empty.
This contradicts that $f_n^a(D)\cap D\neq \emptyset$.
Thus $D$ is a wandering domain of $f_n$.
It is immediate from our construction of $D$ that $\lim_{j\to \infty}\mathrm{diam}\,f_n^j(D)=0$.
Moreover, for a fixed $x_0\in D$, one can suppose that the $\omega$-limit set $\omega(x_0)$ contains $\Lambda$ by taking 
the words $\underline{v}_{\,k+1}\in \{1,2\}^{k}$ of (\ref{eqn_CC1}) suitably.
Since $\lim_{j\to \infty}\mathrm{diam}\,f_n^j(D)=0$, we also have $\omega(x)\supset \Lambda$ for any $x\in D$.
This shows that $D$ is a wandering domain of $f_n$.
\end{proof}

Note that the first perturbation of $f$ in Subsection \ref{subsec.LP} does not 
depend on the sequence $\boldsymbol{z}=\{z_k\}_{k=1}^\infty$ such that each entry $z_k$ is either $z_0$ or $z_0+1$ but the second perturbation in Subsection \ref{SS_Pert} does.
The wandering domain $D=\mathrm{Int} R_{k_0}$ also depends on $\boldsymbol{z}$.
We express the dependence by $f_{n,\boldsymbol{z}}$ and $D_{\boldsymbol{z}}=\mathrm{Int} R_{k_0,\boldsymbol{z}}$.
On the other hand, since the support of the second perturbation is fully contained in $\mathbb{G}^u(0)$,  
$\mathbb{B}^s\cap \mathbb{B}^u$ is independent of the sequence $\boldsymbol{z}$ 
for any bridges $B^s$ of $K_\Lambda^s$ and $B^u$ of $K_\Lambda^u$.

\begin{remark}\label{r_fzDz}
For any integer $j\geq k_0$, consider the integer $m_j$ defined by 
\begin{align*}
m_j&=N_2+z_j j^2+\langle j\rangle+N_1+\widehat i_j+N_0\\
&=N_2+z_j j^2+\widehat n_{j+1}+j^2+j +N_1+\widehat i_j+N_0,
\end{align*}
and $\widehat m_k=\sum_{j=k_0}^km_j$ if $k\geq k_0$.
By (\ref{def.transition}), $g_j=f^{m_j}$.
Our wandering domain $D_{\boldsymbol{z}}$ satisfies $f_{n,\boldsymbol{z}}^{\widehat m_k}(D_{\boldsymbol{z}})\subset \mathrm{Int} R_{k,\boldsymbol{z}}$ for any integer $k\geq k_0$ and 
$f_{n,\boldsymbol{z}}^j(D_{\boldsymbol{z}})$ stays in $S$ for $j$ with 
$\widehat u_k\leq j< \widehat m_{k}-N_2$, where 
$\widehat u_k=\widehat m_{k-1}+N_0+\widehat i_k+N_1$.
See Figure \ref{fig_7_3}.
Fix a sufficiently large positive integer $l$ and suppose that $k>\max\{k_0,l\}$.
From the form of ${B}_k^u$ of (\ref{eqn_CC1}) and the fact that 
$f_{n,\boldsymbol{z}}^{\widehat u_k}(D_{\boldsymbol{z}})\subset \mathbb{B}_k^u$, 
$f_{n,\boldsymbol{z}}^j(D_{\boldsymbol{z}})$ is contained in $\mathbb{B}^u(l;\,\underline{1}^{(l)})\cap \mathbb{B}^s(l;\,\underline{1}^{(l)})$
for any $j$ with $\widehat u_{k}+l\leq j< \widehat u_{k}+z_kk^2-l$
and in $\mathbb{B}^u(l;\,\underline{2}^{(l)})\cap \mathbb{B}^s(l;\,\underline{2}^{(l)})$ 
for any $j$ with $\widehat u_{k}+z_kk^2+l\leq j< \widehat u_{k}+z_kk^2+k^2-l$.
Since by 
Lemma \ref{lem_CC}-(\ref{CC3}) $N_2+\widehat n_{j+1}+j +\widehat i_j+N_0\asymp j$, it follows that 
$$
\lim_{k\to \infty}\frac{(z_kk^2-2l)+(k^2-2l)}{\widehat m_k-\widehat m_{k-1}}=
\lim_{k\to \infty}\frac{z_kk^2+k^2-4l}{m_k}=1.
$$
Thus, for any integer $l>0$ and almost all $j>0$, we have
\begin{equation}\label{eqn_f_nz^j}
f_{n,\boldsymbol{z}}^j(D_{\boldsymbol{z}})\subset (\mathbb{B}^u(l;\,\underline{1}^{(l)})\cap \mathbb{B}^s(l;\,\underline{1}^{(l)}))\cup (\mathbb{B}^u(l;\,\underline{2}^{(l)})\cap \mathbb{B}^s(l;\,\underline{2}^{(l)})).
\end{equation}
Here `\emph{almost all} $j$' means that, if $d_l(a)$ is the number of integers $j$ $(0<j\leq a)$ satisfying (\ref{eqn_f_nz^j}) for $a>0$, then $\lim_{a\to \infty}d_l(a)/a=1$ holds.
\end{remark}

Now we are ready to prove our main theorem.

\begin{proof}[Proof of Theorem \ref{main-thm}]
Let $f_{n,\boldsymbol{z}}$ be the diffeomorphism and $D_{\boldsymbol{z}}$ the wandering domain of $f_{n,\boldsymbol{z}}$ as above.
We will show that the sequence $\boldsymbol{z}=\{z_k\}_{k=1}^\infty$ can be chosen so that, 
for any $x\in D_{\boldsymbol{z}}$, the forward orbit $\{x,f_{n,\boldsymbol{z}}(x), f_{n,\boldsymbol{z}}^2(x),\dots\}$ has historic behavior.

Note that the horseshoe $\Lambda$ has two fixed points $p$, $\widehat p$ with
$$\{p\}=\bigcap_{l=1}^\infty \mathbb{B}^u(l;\,\underline{1}^{(l)})\cap \mathbb{B}^s(l;\,\underline{1}^{(l)})\quad\text{and}\quad
\{\widehat p\}=\bigcap_{l=1}^\infty\mathbb{B}^u(l;\,\underline{2}^{(l)})\cap \mathbb{B}^s(l;\,\underline{2}^{(l)}).$$

Consider the space $\mathcal{P}(M)$ of probability measures on $M$ with the weak topology.
For any $x\in D_{\boldsymbol{z}}=\mathrm{Int} R_{k_0,\boldsymbol{z}}$ and any non-negative integer $m$, the element $\mu_{x}(m)$ of $\mathcal{P}(M)$ is defined as
$$\mu_{x}(m)=\frac{1}{m+1}\sum_{i=0}^m\delta_{f_{n,\boldsymbol{z}}^i(x)}.$$
We are concerned with the subsequence $\{\mu_{x}(\widehat m_k)\}_{k= k_0}^\infty$ of $\{\mu_{x}(m)\}$.
Let $\nu_{0}$ and $\nu_{1}$ be the elements of $\mathcal{P}(M)$ defined as
$$\nu_{0}=\frac1{z_0+1}\bigl(z_0\delta_{p}+\delta_{\widehat p}\bigr)\quad\text{and}\quad
\nu_1=\frac1{z_0+2}\bigl((z_0+1)\delta_{p}+\delta_{\widehat p}\bigr),$$
and $U_0$, $U_1$ arbitrarily small neighborhoods of $\nu_0$ and $\nu_1$ in $\mathcal{P}(M)$ with $U_0\cap U_1=\emptyset$ respectively.

Let $k_1$ be an integer  sufficiently larger than $k_0$.
Then one can suppose that the integer $l<k_1$ in (\ref{eqn_f_nz^j}) is also large enough.
If we take the entries of $\boldsymbol{z}$ so that $z_k=z_0$ for $k=1,\dots,k_1$, then $\mu_x(\widehat m_{k_1})$ is contained in $U_0$.

For any integer $m'>m$, let $\mu_{x}(m',m)$ be the measure on $M$ defined as $\mu_{x}(m',m)=\frac{1}{m'+1}\sum_{i=m+1}^{m'}\delta_{f_n^i(x)}$.
If $k>k_1$, then 
$$\mu_{x}(\widehat m_{k})=\frac{\widehat m_{k_1}+1}{\widehat m_{k}+1}\mu_{x}(\widehat m_{k_1})+\mu_{x}(\widehat m_{k},\widehat m_{k_1}).$$
Thus the contribution of the first term goes to zero as $k\to \infty$.
It follows that, if $k_2$ is sufficiently larger than $k_1$, then for any sequence $\boldsymbol{z}=\{z_j\}_{j=1}^\infty$ with $z_j=z_0$ $(j=1,\dots,k_1)$ and $z_j=z_0+1$ $(j=k_1+1,\dots,k_2)$, $\mu_{x}(\widehat m_{k_2})$ is contained in $U_1$.

Repeating a similar argument, we have a monotone increasing sequence $\{k_a\}_{a=1}^\infty$ such that
$\{\mu_{x}(\widehat m_{k_a},\boldsymbol{z})\}_{a=1}^\infty$ has two subsequences one of which converges to $\nu_0$ and the other to $\nu_1$, where $\boldsymbol{z}=\{z_j\}_{j=1}^\infty$ is the sequence with
$$
\begin{cases}
z_j=z_0 & \text{for \ $j=1,\dots,k_1$, $k_{2a}+1,\dots,k_{2a+1}$ $(a=1,2,\dots)$}\\
z_j=z_0+1&\text{for \ $j=k_{2a-1}+1,\dots,k_{2a}$ $(a=1,2,\dots)$}.
\end{cases}
$$
In particular, this implies that the limit of $\mu_{x}(m)$ does not exist.
It follows that, for any $x\in D_{\boldsymbol{z}}$, the forward orbit of $x$ under $f_{n,\boldsymbol{z}}$ is historic.
This completes the proof of our main theorem.
\end{proof}

\subsection*{Acknowledgements}
We would like to thank 
Sebastian van Strien, 
Dmitry Turaev, 
Eduardo Colli
and 
Edson Vargas
for helpful conversations, 
which were very important in improving this paper.
This work was partially supported by 
JSPS KAKENHI Grant Numbers 25400112 and 26400093.


\end{document}